\documentclass[11pt]{amsart}
\usepackage{xcolor}
\usepackage{amsmath}
\usepackage{amsthm}
\usepackage{lipsum}
\usepackage{amsopn}
\usepackage{amssymb}
\usepackage{fullpage}
\usepackage{enumerate}
\numberwithin{equation}{section}
\usepackage{float}
\usepackage{graphicx}

\usepackage{dsfont}
\usepackage{mathtools}
\usepackage{fancyhdr}
\usepackage{textcomp}
\usepackage{ulem}
\usepackage{fancyhdr}
\newtheorem{theorem}{Theorem}[section]

\newtheorem{corollary}[theorem]{Corollary}
\newtheorem{lemma}[theorem]{Lemma}
\newtheorem{proposition}[theorem]{Proposition}
\theoremstyle{definition}
\newtheorem{definition}[theorem]{Definition}
\newtheorem{remark}[theorem]{Remark}
\newtheorem{example}[theorem]{Example}

\DeclareMathOperator{\Arg}{Arg}

\usepackage{hyperref}
\hypersetup{
	colorlinks=true,
	linkcolor=blue,
	filecolor=magenta,      
	urlcolor=cyan,
	pdftitle={L\"owner chains via close-to-convex functions},
	pdfpagemode=FullScreen,
}

  \usepackage{float}
 \usepackage{blkarray}
\usepackage{graphicx}
\usepackage{ulem}

\usepackage{dsfont}
\usepackage{fancyhdr}
\usepackage{nicematrix}
\usepackage{bm}
\usepackage{textcomp}
\usepackage{xcolor}
\usepackage{tcolorbox}
\usepackage{tikz}
\usetikzlibrary{arrows.meta,calc}
\usetikzlibrary{shapes.misc}
\usetikzlibrary{decorations.pathreplacing}
\usetikzlibrary{decorations.markings}
\usepackage{subcaption}
\def\CB{\color{black} }

\newcommand{\mn}{\mu_{\nu}}

\allowdisplaybreaks

\makeatletter

\renewcommand{\l@section}{%
  \@tocline{1}{0pt}{0pt}{2.3pc}{}%
}

\renewcommand{\l@subsection}{%
  \@tocline{2}{0pt}{2.2pc}{3.2pc}{}%
}

\renewcommand{\l@subsubsection}{%
  \@tocline{3}{0pt}{4.4pc}{4.0pc}{}%
}

\makeatother

\begin{document}

\title{Loewner chains with multiple boundary attraction points: A construction via close-to-convex functions}
\author{Georgios Nikolaidis, Eleftherios K. Theodosiadis, Konstantinos Zarvalis}

\address{Department of Mathematics, Aristotle University of Thessaloniki, 54124, Thessaloniki, Greece}
\email{nikolaidg@math.auth.gr}
\address{Department of Mathematics, Aristotle University of Thessaloniki, 54124, Thessaloniki, Greece}
\email{etheodog@math.auth.gr}
\address{Department of Mathematics, Aristotle University of Thessaloniki, 54124, Thessaloniki, Greece}
\email{zarkonath@math.auth.gr}

\fancyhf{}
\renewcommand{\headrulewidth}{0pt}
\fancyhead[RO,LE]{\small \thepage}
\fancyhead[CE]{\footnotesize }
\fancyhead[CO]{\footnotesize } 

\fancyfoot[L,R,C]{}
\date{}

\begin{abstract}
In this work we construct multi-slit \textit{chordal Loewner chains} with multiple attraction points on the real line. We build a model domain with the use of close-to-convex functions, in a way that allow us to prescribe the attraction points. These families of Loewner chains resemble the classical model for semigroups of holomorphic maps of the upper half-plane. Our work is to choose an appropriate set of arbitrarily many parameters and to study the geometric behavior of the model domain, as it evolves with respect to time, in order to derive the aforementioned Loewner chain. Finally, with the use of harmonic measure, we calculate the angles of convergence of each slit of the chain.
\end{abstract}

\maketitle

\tableofcontents
\addtocontents{toc}{\protect\setcounter{tocdepth}{1}}
\section{Introduction}

In its classical setting, Loewner Theory describes the dynamics of a family of \textit{slit domains}, that is a family of simply connected domains excluding a simple curve (the slit) that emanates from the boundary. To put it into perspective, if $(H_t)_{0\le t<T}$ is a  family of single-slit domains of the upper half-plane $\mathbb{H}$, then $H_t:=\mathbb{H}\setminus\gamma([0,t])$ for a continuous curve $\gamma:[0,T]\to\overline{\mathbb{H}}$, with $\gamma(0)\in\mathbb{R}$ and $\gamma((0,T))\subset\mathbb{H}$. By the Riemann Mapping Theorem, there exists a family of conformal maps $g_t=g(\cdot,t):H_t\overset{onto}{\longrightarrow}\mathbb{H}$, normalized in such a way that $g_t(z)-z\rightarrow0$, as $z\rightarrow\infty$. The preceding normalization condition is called the \textit{hydrodynamic condition}. Furthermore, an application of the Schwarz Reflection Principle, shows that $g_t$ has a Laurent expansion at infinity
 \begin{equation*}
     g_t(z)=z+\dfrac{b(t)}{z}+\cdots,
 \end{equation*}
 for all $z\in\mathbb{C}$ that lie outside a disc containing $\gamma([0,t])\cup\gamma([0,t])^*$, where $K^*$ denotes the reflection of $K$ with respect to the real axis. In particular, the coefficient $b(t)$ is positive and it is called the \textit{half-plane capacity} of $\gamma([0,t])$, denoted by $\text{hcap}(\gamma([0,t]))$.
 Then, the \textit{chordal} Loewner differential equation reads for the initial value problem
 \begin{equation}\label{LE1}
    \dfrac{\partial g}{\partial t}(z,t)= \dfrac{b'(t)}{g_t(z)-\lambda(t)}, \quad g_0(z)=z,
 \end{equation}
for all $z\in H_t$ and $0\le t <T$, where $\lambda(t)$ is a continuous real-valued function of $t$, given as $\lambda(t)=g_t^{-1}(\gamma(t))$, called the \textit{driving function}. In the literature, the half-plane capacity mostly appears as $b(t)=2t$, which is possible by means of a time reparameterization. We refer to \cite{mon} for a detailed derivation of equation (\ref{LE1}).

Slightly generalizing single-slit Loewner chains described by equation (\ref{LE1}), one can consider its multiple---finite or infinite---slit version. For instance, as we see in \cite{star}, given $n$ many disjoint Jordan curves that emanate from the real line towards the upper half-plane, the corresponding Loewner chain is produced by the driving functions $\lambda_1,\dots,\lambda_n$. Furthermore, in \cite{tec}, the authors show that if we have infinitely many slits $\Gamma_j$ parameterized in $[0,1]$ such that each curve $\Gamma_n(t)$ can be separated from the closure of $\bigcup_{j\neq n}\Gamma_{j}(t)$ at each time $t$, by open sets, then the Loewner equation is written as 
 \begin{equation}\label{tecn}
     \dfrac{\partial g}{\partial t}
	 (z,t)=\sum_{n=1}^{+\infty}\dfrac{b_n(t)}{g(z,t)-\lambda_n(t)},
 \end{equation}
	for $z\in\mathbb{H}\setminus\bigcup_{n=1}^{+\infty}\Gamma_n(t)$ and a.e $0\le t\le1$, where $\sum_{n=1}^{+\infty}b_n(t)=\partial_t\textrm{hcap}(\bigcup_{n=1}^{+\infty}\Gamma_{n}(t))$.

In the opposite direction, if we consider the initial value problem (\ref{LE1}), we then let $T_z$ be the supremum of all $t$, such that the solution to the equation is well defined and $g_t(z)\in\mathbb{H}$ for all $t\le T_z$. Then, the domain $H_t:=\{z\in\mathbb{H}: T_z>t\}$ is simply connected and $K_t:=\mathbb{H}\setminus H_t$ is compact. Moreover, $g_t$ maps $H_t$ conformally onto $\mathbb{H}$, satisfying the hydrodynamic condition. We refer to $(K_t)_{t}$, as the \textit{compact hulls} generated by (\ref{LE1}); see \cite[Chapter 4]{lawl} for details. Also, considering the unique Riemann maps from $\mathbb{H}$ onto $H_t$, satisfying the hydrodynamic condition, that is $f(\cdot,t)=g^{-1}(\cdot,t)$, then they satisfy the PDE
\begin{equation}
    \dfrac{\partial f}{\partial t}(z,t)=-f'(z,t)\frac{b(t)}{z-\lambda(t)}
\end{equation}
for all $z\in\mathbb{H}$ and $t\ge0$.
We usually call \textit{Loewner chain} the family of mappings $(f(\cdot,t))_{t\ge0}$.
For simplicity, we may refer to the family of the corresponding domains $(H_t)_{t\ge0}$ as the Loewner chain as well. To avoid repetition, we also refer to Loewner chains as Loewner flow.

It is not true, in general, that an arbitrary continuous function $\lambda(t)$ produces hulls that are curves. Several authors have studied the relation between a driving function and the corresponding hulls, showing that \text{Lip}-$\frac{1}{2}$ driving functions, with sufficiently small $\text{Lip}(\frac{1}{2})$-norm, produce quasi-slit domains (see e.g. \cite{lind}, \cite{MR} and  \cite{slei1}).

In its most general setting, Loewner Theory provides a description, through a differential equation again, for a family of simply connected domains that are not slit domains. To put it into context, if $(H_t)_{t\ge0}$ is a  decreasing family of simply connected subdomains of $\mathbb{H}$, so that $K_t:=\mathbb{H}\setminus H_t$ is compact and the corresponding Riemann maps $g_t$ are continuous with respect to $t$, then there exists some function $Q(z,t)$, analytic with respect to $z$ with $\mathrm{Im}Q<0$, so that 
\begin{equation}\label{kufarev}
    \frac{\partial g}{\partial t}(z,t)=Q(g(z,t),t).
\end{equation}
This is called the \textit{Loewner-Kufarev equation} and $Q$ is called the \textit{driving function} of the chain. With regard to the driving function of the Loewner-Kufarev equation, by \eqref{tecn} we understand that
\begin{equation*}
    Q(z,t)=\sum_{n=1}^{+\infty}\frac{b_n(t)}{z-\lambda_n(t)}.
\end{equation*}
As we mentioned above, solutions to \eqref{kufarev} are not necessarily slit functions. However, if $(H_t)_{t\ge0}$ are slit domains, with the properties defined in \cite{tec}, then the driving function is necessarily of this form. In this work, we will exclusively deal with finitely many slits, so we only treat driving functions of the above form.

\subsection*{Purpose of the work} 
The motivation of this article lies in our concern to find a Loewner chain with \textit{multiple attraction points}, in the following sense. Assume that at each time $t\ge0$, the growing hulls $(K_t)_{t\ge0}$ are $N$ many disjoint Jordan arcs emanating from some prescribed boundary points of $\mathbb{H}$ up to some interior points, the so-called \textit{tip points} of the arcs. Thus, we may write
\begin{equation*}
K_t:=\bigcup_{j=1}^{N}\Gamma_j([0,t]),    
\end{equation*}
where $\Gamma_j([0,t])$ is a Jordan arc with $\Gamma_{j}(0)\in\mathbb{R}$ and endpoint  $\Gamma_j(t)\in\mathbb{H}$.
Then, as $t\rightarrow+\infty$, the tip points of the arcs converge to some boundary points of $\mathbb{H}$. These points will serve as the attraction points of the Loewner chain, as seen in Figure \ref{fig: Orbits of the flow }. As \textit{orbits} or \textit{trajectories} of the tip points we will interchangeably mean the curves $\Gamma_j:[0,+\infty)\rightarrow\mathbb{H}$ and their traces 
\begin{equation*}
  \{\Gamma_j(t)\vcentcolon t\ge0\}=\Gamma_j([0,+\infty)).  
\end{equation*}
Now, one could argue that in order to find a Loewner chain with some desired geometric behavior as described above, we have to come up with a proper driving function and solve equation \eqref{kufarev} accordingly. Generally, the geometric behavior of the chain $(H_t)_{t\ge0}$ is related to the analytic properties of the driving function, but the exact relation is not always clear. In other words, it is complicated to a priori find the driving function such that the corresponding Loewner equation may yield a solution with the desired geometric behavior. For this reason, we opt to start with a chain presenting the geometric behavior we described above, and then work towards computing the respective driving function.

\begin{figure}[ht]
\centering
\resizebox{0.95\linewidth}{!}{%
\begin{tikzpicture}[
    x=1cm,y=1cm,
    every node/.style={font=\small}
]

\definecolor{myblue}{RGB}{70,80,180}
\definecolor{myred}{RGB}{170,90,90}
\definecolor{mygray}{RGB}{150,150,150}

\draw[gray] (0,0) -- (12.2,0);

\coordinate (k1) at (0.4,0);
\coordinate (k2) at (5.8,0);
\coordinate (k3) at (10,0);
\coordinate (k4) at (11.0,0);
\tikzset{
  midarrowright/.style={
    postaction={
      decorate,
      decoration={
        markings,
        mark=at position 0.5 with {
          \arrow{Latex[length=2mm,width=1.4mm]}
        }
      }
    }
  },
  midarrowleft/.style={
    postaction={
      decorate,
      decoration={
        markings,
        mark=at position 0.5 with {
          \arrowreversed{Latex[length=2mm,width=1.4mm]}
        }
      }
    }
  }
}
\node at (3.3,0.2) {$\cdots$};

\draw[myred, line width=0.8pt,midarrowleft]
(k1) .. controls (2.6,2.8) and (10.8,4.3) .. (k4);

\draw[myblue, line width=0.8pt,midarrowright]
(1.8,0) .. controls (2.5,2.3) and (8.1,2.5) .. (k3);

\draw[myblue, line width=0.8pt,midarrowright]
(2.9,0) .. controls (3.4,1.9) and (8.5,2.1) .. (k3);

\draw[myblue, line width=0.8pt,midarrowright]
(4.4,0) .. controls (4.6,1.3) and (8.7,1.6) .. (k3);

\draw[myred, line width=0.8pt,midarrowleft]
(k2) .. controls (6.8,0.45) and (8.3,1.0) .. (8.6,0);

\node at (6.2,0.45) {$\Gamma_N$};
\node at (2.9,-0.3) {$\Gamma_j(0)$};
\node at (4.0,2.7) {$\Gamma_1$};
\node at (8.7,-0.3) {$\Gamma_N(0)$};
\node at (11.5,-0.3) {$\Gamma_1(0)$};
\node at (8.5,1.5) {$\Gamma_j$};

\node[below] at (k1) {$\Gamma_1(\infty)$};
\node[below] at (k2) {$\Gamma_N(\infty)$};
\node[below] at (k3) {$\Gamma_j(\infty)$};

\draw (0.4,0.08) -- (0.4,-0.08);
\draw (5.8,0.08) -- (5.8,-0.08);
\draw (11.0,0.08) -- (11.0,-0.08);

\end{tikzpicture} 
}
\caption{We denote by $\Gamma_j$ the trace of each curve $\Gamma_{j}:[0,\infty]\rightarrow\overline{\mathbb{H}}$, which emanate from $\Gamma_j(0)\in\mathbb{R}$ and land at some boundary point, say $\Gamma_j(\infty)\in\mathbb{R}$. Those curves, denote the orbits of the tip points of the Loewner chain, as time evolves.}
\label{fig: Orbits of the flow }
\end{figure}
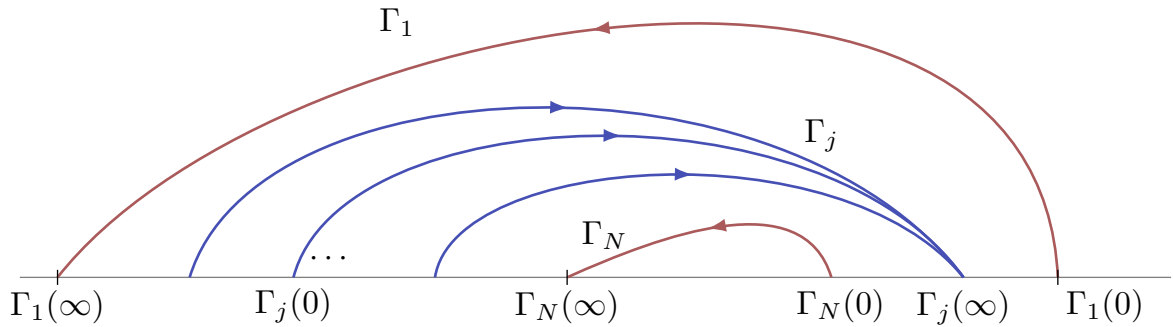

\subsection*{Semigroups of holomorphic self-maps} 
More in the context of complex dynamics, another instance of families of univalent functions, that have been extensively studied, is \textit{Semigroups of holomorphic self-maps}. Now, semigroups of holomorphic maps can be useful in providing some  fruitful ideas for our work, because as we shall see shortly, they yield one attraction point, the \textit{Denjoy--Wolff point}. Hence, according to the purpose of this work, if we are interested in one attraction point, then this could be possible by means of the theory of semigroups. For our purposes, we only view semigroups as a special case of Loewner chains, in a manner that we explain below. Suppose that the driving function is written only as a function of $\mathbb{H}$, thus $Q(z,t)=Q(z)$. Then, the solution to \eqref{kufarev} is some function $\phi_t(z)$, which by the uniqueness of the solution satisfies the \textit{semigroup property}
\[
    \phi_{t+s}=\phi_t\circ\phi_s,
    \qquad s,t\geq 0, \quad\phi_0(z)=z.
\]
Conversely, every continuous family $(\phi_t)_{t\ge0}$, satisfying the semigroup property above, is governed by an autonomous differential equation of the form
\[
    \frac{\partial}{\partial t}\phi_t(z)=G(\phi_t(z)),
    \qquad \phi_0(z)=z,
\]
where $G$ is called the \textit{infinitesimal generator} of $(\phi_t)$. For more details, we refer to \cite[Theorem 2.6]{BP} where the authors show that if $G:\mathbb{H}\to\mathbb{H}\cup\mathbb{R}$ is analytic, then it gives rise to a semigroup with Denjoy--Wolff point $\infty$. A connection between semigroups and Loewner chains is further discussed in
\cite{ABCD2010}, while explicit examples can be found in \cite{sola}. 

To demonstrate how semigroup theory ties to Loewner theory, suppose that $h\vcentcolon\mathbb{H}\to\mathbb{C}$ is a \textit{convex in the right direction} (also known as \textit{starlike at infinity} in the literature) univalent function. This means that its range $\Omega\vcentcolon=h(\mathbb{H})$ is a simply connected domain with the property $\Omega+t\subseteq\Omega$, for all $t\ge0$. The domain $\Omega$ is also said to be convex in the right direction. Then, the conjugate form
\begin{equation}\label{semigroups}
    f(z,t)=h^{-1}(h(z)+t),
\end{equation}
defines a family satisfying the semigroup property. Observe by a direct differentiation that the infinitesimal generator is given as 
\begin{equation}
    \label{infinitesimal generator}
    G(z)=\frac{1}{h'(z)}, \quad z\in\mathbb{H}.
\end{equation}For detailed information on the rich theory of semigroups and certain recent results, we refer the interested reader to \cite{BCDMGZ,KGR} and the monograph \cite{BCDM}. 

Given $z\in\mathbb{H}$, its \textit{orbit} under $(f(\cdot,t))_{t\ge0}$ is defined as the set $\Gamma_z\vcentcolon=\{f(z,t)\vcentcolon t\ge0\}$. Formula \eqref{semigroups} allows the study of the orbits to be conducted inside the domain $\Omega$, rendering their examination relatively easier because of the geometry of $\Omega$. Indeed, the conjugation in \eqref{semigroups} maps $\Gamma_z$ onto the horizontal half-line $\{h(z)+t\vcentcolon t\ge0\}$, thus yielding a ``linearization'' of the orbit. By the convexity of $\Omega$, there exists a unique prime end of $\Omega$ with impression $\infty$ (more information on prime ends follows in Section 2) and defined by crosscuts with increasing real parts. But the image of every orbit extends to $\infty$ moving towards the right as demonstrated in Figure \ref{fig:convex-right-domain}. Returning to $\mathbb{H}$, there exists a unique $\tau\in\mathbb{R}\cup\{\infty\}$ such that every orbit lands at $\tau$. In other words, $\lim_{t\to+\infty}f(z,t)=\tau$, for all $z\in\mathbb{H}$. Therefore, every family defined through a relation like \eqref{semigroups} has exactly one attraction point.

\begin{figure}[ht]
\centering
\resizebox{0.55\linewidth}{!}{%
\begin{tikzpicture}[
    x=1cm,
    y=0.51cm,
    every node/.style={font=\small}
]
\tikzset{
  midarrowright/.style={
    postaction={
      decorate,
      decoration={
        markings,
        mark=at position 0.5 with {
          \arrow{Latex[length=2mm,width=1.4mm]}
        }
      }
    }
  },
  midarrowleft/.style={
    postaction={
      decorate,
      decoration={
        markings,
        mark=at position 0.5 with {
          \arrowreversed{Latex[length=2mm,width=1.4mm]}
        }
      }
    }
  }
}
\definecolor{mygray}{RGB}{150,150,150}
\definecolor{myblack}{RGB}{40,40,40}


\draw[myblack, line width=0.6pt]
(1.8,4.2)
.. controls (3.3,3.9) and (4.4,3.2) .. (4.55,2.45)
.. controls (4.7,1.55) and (3.85,0.95) .. (1.0,0.70);

\draw[myblack, line width=0.7pt] (1.0,0.35) -- (2.35,0.35);

\draw[myblack, line width=0.6pt]
(1.0,0.00)
.. controls (2.0,-0.02) and (2.75,-0.55) .. (3.85,-0.85)
.. controls (5.3,-1.10) and (6.7,-1.25) .. (9.1,-1.3);

\draw[mygray, line width=0.7pt,midarrowright] (4.95,2.00) -- (8.4,2.00);

\fill[myblack] (4.95,2.00) circle (1.1pt);
\fill[myblack] (7.15,2.00) circle (1.1pt);

\node at (5,2.32) {$h(z)$};
\node at (7.85,2.32) {$h(z)+t$};
\node[anchor=west] at (8.5,1.5) {$h(\tau)$};

\end{tikzpicture}%
}
\caption{The domain of a convex in the right direction function \(h\). Fixing any \(z\in\mathbb H\), the point \(h(z)+t\) lies in \(\Omega\) for every \(t\ge 0\), marching to the point at infinity, whose preimage is some \(\tau\in\widehat{\mathbb R}\).}
\label{fig:convex-right-domain}
\end{figure}
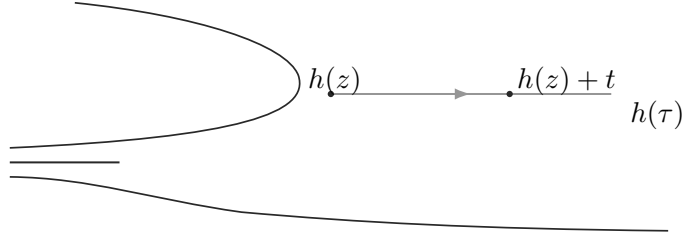

\subsection*{Close-to-convex functions} 
The class of the convex in the right direction functions may be generalized to a wider class of functions, the \textit{close-to-convex} functions. This is the class of univalent functions, whose range is complementary to the union of a family of non-intersecting half-lines. For instance, a convex in the right direction domain $\Omega$ has such a property. In a general context, close-to-convex functions are not systematically studied in Loewner theory. With our construction in this work, we present a slightly more general form of \eqref{semigroups}, by building a family $(h_t)_{t\ge0}$ of close-to-convex functions or equivalently a family of domains $(\Omega_t)_{t\ge0}$ with $\Omega_t\subset\Omega_s$ for $s<t$, so that 
\begin{equation}
    \label{Loewner}
    f(z,t)=h_0^{-1}(h_t(z))
\end{equation}
for all $z\in\mathbb{H}$ and $t\ge0$. 

In an idea to construct a Loewner chain $(f(\cdot,t))_{t\ge0}$ with multiple attraction points, we will make use of this wider class of functions. This would not be possible with the use of convex on the right-hand direction functions, as this leads to necessarily one only attraction point. In addition to that, looking at the driving function \eqref{kufarev} of the Loewner differential equation, there is no obvious way to choose the parameters to ensure that the solution is a family of maps with many attracting points. Hence, we choose a very intuitive geometric approach that resembles Semigroups. 

\subsection*{Ideas for the construction of the Loewner chain} 
Even though we will construct a Loewner chain with arbitrarily many attraction points, we motivate the reader by example. We initiate our discussion with the simplest case, the one with only two attraction points, to give the idea of the construction in an intuitive way. As a paradigm, consider a domain $\Omega_0$ to be horizontal strip minus a left and a right horizontal half-line, with different heights, as in Figure \ref{fig:two_slits_h_0}. This is a domain that corresponds to the range of a  close-to-convex function, since its complement consists of horizontal half-lines. As we shall see in Section 3, $h_0:\mathbb{H}\rightarrow\Omega_0$ is of the form
$$h_0(z)=\alpha\log(z-x_1)+b\log(z-x_2)-c\log(z-x_3)$$
for some $a,b,c>0$ and $x_1<x_2<x_3$ points on the real line. Note that these points correspond to the points at infinity, though $h_0$, as shown in Figure \ref{fig:two_slits_h_0}. One can see that the endpoints $w_1,w_2$ of the half-lines are the images of some $\xi_1\in(x_1,x_2)$ and $\xi_2\in(x_3,+\infty)$, respectively. Namely, there are the zeros of $h_0'$, as we shall se later on. Now, we wish to push $w_1,w_2$ to the points at infinity. Borrowing the idea from convex in the right direction functions, we see that $w_2-t$ extends to a point at infinity, which corresponds to a boundary point through $h_0$. However, $h_0^{-1}(w_1-t)$ is not defined, because at the same time $w_1-t$ marches on the left which means that $\Omega_0-t\nsubseteq\Omega_0$. So, although $h_0^{-1}(w_2-t)$ would tend to a boundary point in $\mathbb{R}$, $h_0^{-1}(\Omega_0-t)$ is not actually defined, thus \eqref{Loewner} does not make sense!
 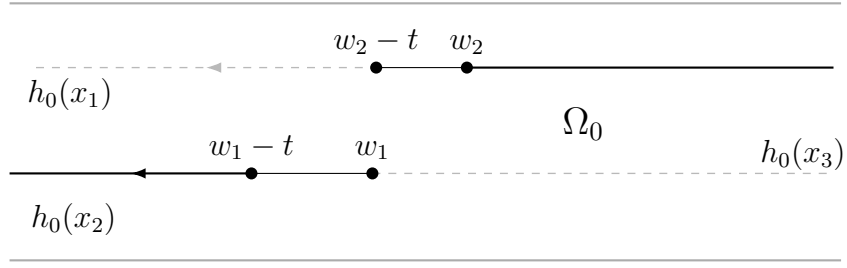
\begin{figure}
    \label{example1}
\centering
\begin{tikzpicture}[
    scale=1,
    every node/.style={font=\large},
    dot/.style={circle, fill=black, inner sep=1.6pt},
    boundary/.style={gray!65, line width=0.8pt},
    solidpart/.style={black, line width=0.8pt},
    dashedpart/.style={gray!55, line width=0.5pt, dashed}
]
\tikzset{
  midarrowright/.style={
    postaction={
      decorate,
      decoration={
        markings,
        mark=at position 0.5 with {
          \arrow{Latex[length=2mm,width=1.4mm]}
        }
      }
    }
  },
  midarrowleft/.style={
    postaction={
      decorate,
      decoration={
        markings,
        mark=at position 0.5 with {
          \arrowreversed{Latex[length=2mm,width=1.4mm]}
        }
      }
    }
  }
}
\def\xleft{0}
\def\xright{11}
\def\ytop{3.6}
\def\yupper{2.75}
\def\ylower{1.35}
\def\ybottom{0.2}

\def\xwtwot{4.85}   
\def\xwtwo{6.05}    

\def\xwonet{3.20}   
\def\xwone{4.8}    

\draw[boundary] (\xleft,\ytop) -- (\xright,\ytop);
\draw[boundary] (\xleft,\ybottom) -- (\xright,\ybottom);

\draw[dashedpart,midarrowleft] (0.35,\yupper) -- (\xwtwot,\yupper);
\draw  (\xwtwot,\yupper) -- (\xwtwo,\yupper);
\draw[solidpart] (\xwtwo,\yupper) -- (10.9,\yupper);

\draw[solidpart,midarrowleft]  (\xleft,\ylower) -- (\xwonet,\ylower);
\draw  (\xwonet,\ylower) -- (\xwone,\ylower);
\draw[dashedpart] (\xwone,\ylower) -- (10.9,\ylower);
\node[dot] at (\xwtwot,\yupper) {};
\node[dot] at (\xwtwo,\yupper) {};
\node[dot] at (\xwonet,\ylower) {};
\node[dot] at (\xwone,\ylower) {};

\node[left]  at (1.5,\yupper-0.35) {$h_0(x_1)$};
\node[above] at (\xwtwot,\yupper+0.05) {$w_2-t$};
\node[above] at (\xwtwo,\yupper+0.05) {$w_2$};

\node[left]  at (8,2) {\Large$\Omega_0$};
\node[above] at (\xwonet,\ylower+0.05) {$w_1-t$};
\node[above] at (0.85,\ylower-0.95) {$h_0(x_2)$};

\node[above] at (\xwone,\ylower+0.05) {$w_1$};
\node[right] at (9.8,\ylower+0.25) {$h_0(x_3)$};

\end{tikzpicture}
\caption{The range $\Omega_0$ is a strip $S$ minus two horizontal half-lines with tip points
$w_1,w_2$. Through $z\mapsto z-t$, the tip points travel to the left, so $\Omega_0-t\nsubseteq\Omega_0$.}
    \label{fig:two_slits_h_0}
\end{figure}

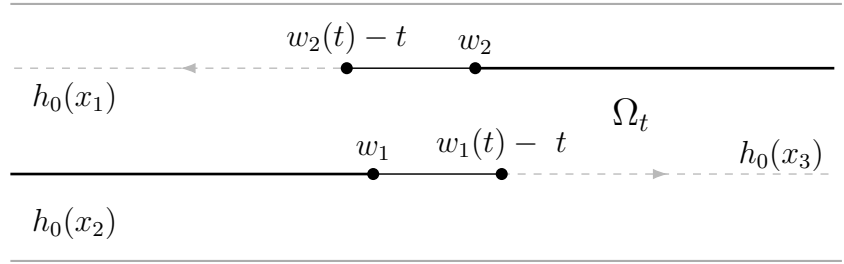
\begin{figure}
\centering
\begin{tikzpicture}[
    scale=1,
    every node/.style={font=\large},
    dot/.style={circle, fill=black, inner sep=1.6pt},
    boundary/.style={gray!65, line width=0.8pt},
    strongpart/.style={black, line width=1.0pt},
    thinpart/.style={black, line width=0.55pt},
    dashedpart/.style={gray!55, line width=0.5pt, dashed}
]
\tikzset{
  midarrowright/.style={
    postaction={
      decorate,
      decoration={
        markings,
        mark=at position 0.5 with {
          \arrow{Latex[length=2mm,width=1.4mm]}
        }
      }
    }
  },
  midarrowleft/.style={
    postaction={
      decorate,
      decoration={
        markings,
        mark=at position 0.5 with {
          \arrowreversed{Latex[length=2mm,width=1.4mm]}
        }
      }
    }
  }
}
\def\xleft{0}
\def\xright{11}
\def\ytop{3.6}
\def\yupper{2.75}
\def\ylower{1.35}
\def\ybottom{0.2}

\def\xwtwott{4.45}   
\def\xwtwo{6.15}     

\def\xwone{4.80}     
\def\xwonett{6.50}   

\draw[boundary] (\xleft,\ytop) -- (\xright,\ytop);
\draw[boundary] (\xleft,\ybottom) -- (\xright,\ybottom);

\draw[dashedpart,midarrowleft] (0.05,\yupper) -- (\xwtwott,\yupper);
\draw[thinpart]   (\xwtwott,\yupper) -- (\xwtwo,\yupper);
\draw[strongpart] (\xwtwo,\yupper) -- (10.9,\yupper);

\draw[strongpart] (\xleft,\ylower) -- (\xwone,\ylower);
\draw[thinpart]   (\xwone,\ylower) -- (\xwonett,\ylower);
\draw[dashedpart,midarrowright] (\xwonett,\ylower) -- (10.9,\ylower);

\node[dot] at (\xwtwott,\yupper) {};
\node[dot] at (\xwtwo,\yupper) {};
\node[dot] at (\xwone,\ylower) {};
\node[dot] at (\xwonett,\ylower) {};

\node[left]  at (1.55,\yupper-0.4) {$h_0(x_1)$};
\node[left]  at (1.55,0.7) {$h_0(x_2)$};
\node[above] at (\xwtwott,\yupper+0.05) {$w_2(t)-t$};
\node[above] at (\xwtwo,\yupper+0.05) {$w_2$};

\node[left]  at (8.6,2.13) {\Large$\Omega_t$};
\node[above] at (\xwone,\ylower+0.05) {$w_1$};
\node[above] at (\xwonett,\ylower+0.05) {$w_1(t)-\ t$};
\node[right] at (9.5,\ylower+0.25) {$h_0(x_3)$};

\end{tikzpicture}
\caption{Through an eligible process, the tip points travel to the points at infinity, in such a way that $\Omega_t\subset\Omega_s$, for $s<t$.}
\label{fig:two_slits_h}
\end{figure}
   
To overcome this, we take a closer look at the formula of $h_0$, which depends on the parameters $x_1,x_2,x_3\in\mathbb{R}$. So, the idea now is to let those parameters vary with time $t$ in such a way that the point $w_1$, which will now depend on $t$ as well, moves towards the points at infinity on the right, so fast that $w_1-t$ still extends to the right. In other words, if $h_t$ is the function above with $x_1(t),x_2(t),x_3(t)$ being the new parameters and $w_1(t),w_2(t)$ the tip points at each time $t\ge0$, we want to obtain $\mathrm{Re}w_1(t)-t\nearrow+\infty$ and $\mathrm{Re}w_1(t)-t\searrow-\infty$, as $t\rightarrow+\infty$, through this process.

In such a case, we understand that $(h_t)_{t\ge0}$ provides a family of decreasing domains $(\Omega_t)_{t\ge0}$ and therefore, the function $f(z,t)=h_0^{-1}(h_t(z)-t)$ is well defined. Recall that the tip points $w_j(t)-t$ correspond to some $\xi_j(t)\in\mathbb{R}$, which provide the driving function in equation \eqref{LE1}. Studying the orbits of the tip points $f(\xi_j(t),t)$, as $t\rightarrow+\infty$, we deduce by Figure \ref{fig:two_slits_h}, that $f(\xi_1(t),t)=h_0^{-1}(w_1(t)-t)\rightarrow x_3=x_3(0)\in\mathbb{R}$ and $f(\xi_2(t),t)=h_0^{-1}(w_2(t)-t)\rightarrow x_1=x_1(0)\in\mathbb{R}$ and hence we get two attraction points.  Note by this figure, that the initial points $x_1,x_3\in\mathbb{R}$ are the attraction points of the flow, regardless of how $x_j(t)$ are parameterized. The evolution of the family of domain $(H_t)_{t\ge0}$ is seen in Figure \ref{fig:loewner_flow_with_two_attraction_points}.

\begin{figure}[ht]
\centering
\begin{tikzpicture}[
    scale=1,
    every node/.style={font=\large},
    dot/.style={circle, fill=black, inner sep=1pt},
    boundary/.style={gray!65, line width=0.8pt},
    orbit/.style={gray!65, line width=0.8pt},
    dashedpart/.style={gray!55, line width=0.5pt, dashed}
]

\def\xleft{0.5}
\def\xright{10.7}
\def\yaxis{1.2}

\coordinate (x1)  at (1.9,\yaxis);
\coordinate (xi1) at (3.1,\yaxis);
\coordinate (x2)  at (4.8,\yaxis);
\coordinate (x3)  at (6.8,\yaxis);
\coordinate (xi2) at (8.8,\yaxis);

\coordinate (p1) at (3.9,3.45);   
\coordinate (p2) at (5.95,2.25);  

\draw[boundary] (\xleft,\yaxis) -- (\xright,\yaxis);

\draw[dashedpart]
    (x1) .. controls (2.35,1.95) and (3.10,2.85) .. (p1);

\draw[orbit]
    (p1) .. controls (4.90,4.70) and (8.65,4.95) .. (xi2);

\draw[orbit]
    (xi1) .. controls (2.75,2.75) and (4.65,2.95) .. (p2);

\draw[dashedpart]
    (p2) .. controls (6.7,1.75) and (6.80,1.35) .. (x3);

\node[dot] at (p1) {};
\node[dot] at (p2) {};

\node[below] at (x1)  {$x_1$};
\node[below] at (xi1) {$\xi_1$};
\node[below] at (x2)  {$x_2$};
\node[below] at (x3)  {$x_3$};
\node[below] at (xi2) {$\xi_2$};

\node[above left]  at ($(p1)+(-0.05,0.00)$) {$f(\xi_1(t),t)$};
\node[above right] at ($(p2)+(0.10,0.05)$) {$f(\xi_2(t),t)$};

\end{tikzpicture}
\caption{The evolution of the Loewner chain $(H_t)_{t\ge0}$ is $\mathbb{H}$ except two slits
$\{f(\xi_j(s),s):0\le s\le t\}$, emanating from $\xi_j=\xi_j(0)$, for $j=1,2$.
As $t\to+\infty$, the orbits of the tip points approach the boundary, forming
chains with two attraction points.}
\label{fig:loewner_flow_with_two_attraction_points}
\end{figure}
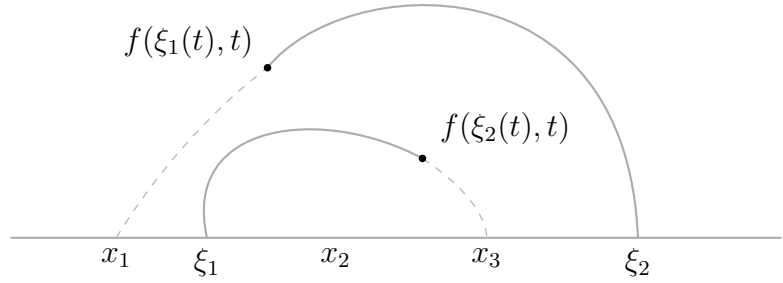

As we shall see later on, a possible choice of varying the parameters would be to consider $x_1,x_2$ steady with time and let $x_3(t)\searrow x_2$, for an appropriate decreasing function. Such a choice may be parameterized as $x_3(t)=x_2+(x_3-x_2)e^{-\theta t}$, for some $\theta>0$, which needs to be properly chosen as well. For this to take place, we count how fast $\mathrm{Re}w_1(t)$ and $\mathrm{Re}w_2(t)$ tend to infinity, as $t\rightarrow+\infty$ in order to take $\theta$ small enough so that $\mathrm{Re}w_1(t)-t\nearrow+\infty$ and $\mathrm{Re}w_1(t)-t\searrow-\infty$, as $t\rightarrow+\infty$.

Once this idea is established, it is relatively easy to perform the appropriate calculations, in order to take a model of two slits. However, generalizing to more than two slits becomes a long procedure, because it would require a configuration of numerous parameters which describe the location of the attraction points and more importantly, the speed of the tip points. In addition to the computational difficulty, one has to come up with a correct choice of parameterization and to understand the mechanism behind the movement of the tip points in time.

\subsection*{Organization of the article} Before ending the Introduction, we outline the main steps of this work. Our principal objective is to construct the machinery for an arbitrary number of attraction points. After providing the necessary background for all the tools that will be of use in Section \ref{sec:tools}, we commence our construction. First, in Section \ref{sec:construction}, we build the model domain of the Loewner chain using a close-to-convex function of the form (see Proposition \ref{h for multiple points})
\begin{equation*}
  h(z)=\sum_{j=1}^Nb_j^+\log(z-k_j^+)-\sum_{j=1}^Mb_j^-\log(z-k_j^-),  
\end{equation*}
for a set of parameters $b_1^+,\dots,b_N^+,b_1^-,\dots,b_M^->0$ and $k_1^+,\dots,k_N^+,k_1^-,\dots,k_M^-\in\mathbb{R}$. These parameters can be chosen almost randomly, and we only impose very loose restrictions which do not harm generality. Our first goal is to study the analytical properties of $h$ in terms of those parameters. Then, in Section \ref{sec:multiple points}, we will introduce a suitable time parameterization for the parameters $k_1^+,\dots,k_N^+,k_1^-,\dots,k_M^-$, which pose as the possible attraction points of our subsequent Loewner chain. Thus, this part is the most important step and captures the essence of our construction. Through the time parameterization, we obtain a family of univalent functions $h_t$, $t\ge0$. To continue with in Propositions \ref{prop:rates of left tips} and \ref{prop:rates of right tips}, we determine the rates at which the tip points converge to the corresponding point at infinity. Then, in view of Corollary \ref{existance of thetas}, we apply a drift $z\mapsto z-\gamma t$, for some constant $\gamma$ depending on the aforementioned rates, and thus we consider $h_t-\gamma t$.

Afterwards, Section \ref{sec:flow} contains the formal construction of the Loewner chain. We build the flow $f(\cdot,t)=h_0^{-1}(h_t(\cdot)-\gamma t)$ and we normalize it so as to satisfy the hydrodynamic condition. In this way we get the desired Loewner chain. For the sake of completeness, we explicitly compute the driving function of the Loewner chain. The whole construction is summarized in Theorem \ref{Loenwer Chain}. Finally, in Section \ref{sec:angles}, we show that all orbits of the Loewner chain launch orthogonally (Theorem \ref{thm:initial angles}) and we also calculate their angles of convergence to their respective attraction points (Theorem \ref{thm:convergence angles}). For this last part, we heavily rely on potential theory, and specifically harmonic measure.


\addtocontents{toc}{\protect\setcounter{tocdepth}{2}}
\section{Basic tools and preliminaries}\label{sec:tools}

\subsection{Theory of conformal mappings}
We initiate our discussion with some analytic properties of a special class of univalent functions, the \textit{close-to-convex} functions. A domain $D\subset\mathbb{C}$ is said to be \textit{convex}, if for any $z,w\in D$, the line segment $[z,w]$ lies in $D$. Obviously, convex domains are simply connected and without loss of generality we may assume that $0\in D$. Then, the Riemann mapping theorem guarantees the existence of a unique conformal mapping $f\in H(\mathbb{D})$ of the unit disc onto $D$, so that $f(0)=0$ and $f'(0)>0$. A univalent function $f\in H(\mathbb{D})$ is said to be \textit{convex} if it maps the unit disc onto a convex domain, thus $f(\mathbb{D})$ is convex. We denote by $\mathcal{C}$ the set of all convex functions. Detailed characterization of convex functions of the unit disc are given in \cite[Chapter 2.6]{dur}. For our purposes where we study univalent function of the upper half-plane $\mathbb{H}$, we can naturally consider convex function of $\mathbb{H}$.

\begin{definition}
A univalent function $f\in H(\mathbb{\mathbb{H}})$ is said to be \textit{convex} if it maps the upper half-plane onto a convex domain, thus $f(\mathbb{H})$ is convex. We denote by $\mathcal{C}_{\mathbb{H}}$ the set of all convex functions of $\mathbb{H}$.

\end{definition}

The next class of univalent functions is called the class of \textit{close-to-convex} functions, also referred to as \textit{linearly accessible} functions. This is the set of all univalent functions of $\mathbb{D}$, whose range is complementary to a family of straight lines.

\begin{definition}\label{close-to-convex}
 A function $f\in H(\mathbb{D})$ is said to be \textit{close-to-convex} if there exists a function $l\in\mathcal{C}$, so that 
 \begin{equation*}
    \mathrm{Re}\left(\frac{f'(z)}{l'(z)}\right)>0, \quad \text{for all }z\in\mathbb{D}.
 \end{equation*}
We use the notation $\mathcal{K}$ for the set of all close-to-convex functions of the unit disc.

\end{definition}
In the literature, the class $\mathcal{K}$ is defined primarily in the unit disc, but it is straightforward to consider functions of the upper half-plane as well. Indeed, we proceed with the following definition.
\begin{definition}\label{def:close to convex}
    
A function $f\in H(\mathbb{H})$ is said to be \textit{close-to-convex} if there exists some $l\in\mathcal{C}_{\mathbb{H}}$, so that 
\begin{equation*}
    \mathrm{Im}\left(\frac{f'(z)}{l'(z)}\right)>0, \quad\text{for all }z\in\mathbb{H}.
\end{equation*}

\end{definition}
Note that we do not assume $f$ to be a priori univalent. Nevertheless, the following theorem gives a sufficient condition for the univalence of close-to-convex functions.
\begin{theorem} \label{spirallikeinH}
     Let $f\in H(\mathbb{H})$ be close-to-convex. Then, $f$ is univalent.
\end{theorem}
\begin{proof}
    By definition, there exists some $l\in\mathcal{C}_{\mathbb{H}}$ such that
    \begin{equation*}
        \mathrm{Im}\left(\frac{f'(z)}{l'(z)}\right)>0, \quad \text{for all }z\in\mathbb{H}.
    \end{equation*}
    By \cite[Theorem 2.19]{dur}, if $D$ is a convex domain and $g\in H(D)$ is a function, such that $\mathrm{Re}g'$ has constant sign in $D$, then it is univalent in $D$. We utilize this by considering the function $g=-if\circ l^{-1}$, which is analytic in the convex domain $l(\mathbb{H})$ and by assumption 
    $$\mathrm{Re}(g'(z))=\mathrm{Re}\left(-i\frac{f'(l^{-1}(z))}{l'(l^{-1}(z))}\right)=\mathrm{Im}\left(\frac{f'(w)}{l'(w)}\right)>0$$
    for all $z\in l(\mathbb{H})$, where $w=l^{-1}(z)\in\mathbb{H}$. Consequently, $g$ is univalent and hence, so is $f$.
\end{proof}

\subsection{Prime end theory}\label{sub:prime ends}

Throughout the present work, we will heavily rely on the prime ends of certain simply connected domains. For this reason, in this subsection, we proceed to a concise rundown of the fundamentals of prime ends and relative notions. This will also lay the groundwork for various constructions that will be required in the sequel. For a comprehensive presentation of the deep theory of prime ends, we refer the interested reader to \cite[Chapter 9]{pom}.

Let $\Omega\subsetneq\mathbb{C}$ be a simply connected domain. Let $\gamma\vcentcolon[0,1]\to\hat{\mathbb{C}}$, where $\hat{\mathbb{C}}$ denotes the extended complex plane, be a Jordan arc with $\gamma((0,1))\subset\Omega$ and $\gamma(0),\gamma(1)\in\partial_\infty\Omega$, where $\partial_\infty\Omega$ is the boundary of $\Omega$ with respect to $\hat{\mathbb{C}}$. Then, the trace $C\vcentcolon=\gamma([0,1])$ is called a \textit{crosscut} of $\Omega$. A crosscut $C$ separates $\Omega$ into two open connected components $A$ and $B$ such that $\partial A\cap\Omega=\partial B\cap\Omega=C\cap\Omega=\gamma((0,1))$. A sequence $\{C_n\}$ of crosscuts is called a \textit{null chain} of $\Omega$ whenever the following three conditions hold:
\begin{enumerate}
    \item[(i)] $C_n\cap C_m=\emptyset$ for all $n,m\in\mathbb{N}$ with $n\ne m$;
    \item[(ii)] for each $n\ge2$, the sets $C_1\cap\Omega$ and $C_{n+1}\cap\Omega$ lie in different connected components of $\Omega\setminus C_n$;
    \item[(iii)] the spherical diameter of $C_n$ (or the Euclidean if $\Omega$ is bounded) converges to $0$, as $n\to+\infty$.
\end{enumerate}
Given a null chain $\{C_n\}$, the \textit{interior part} $\mathrm{Int}C_n$ of $C_n$, $n\ge2$, is defined to be the unique connected component of $\Omega\setminus C_n$ not contanining $C_1\cap\Omega$. Under this notation, two null chains $\{C_n\}$ and $\{C_n'\}$ are said to be \textit{equivalent} if for each $n\ge2$ there exists $m\in\mathbb{N}$ so that
\begin{equation}\label{eq:equivalence relation}
    \mathrm{Int}C_m'\subseteq\mathrm{Int}C_n \quad\text{and} \quad \mathrm{Int}C_m\subseteq\mathrm{Int}C_n'.
\end{equation}
The relation described above can be readily verified that defines an equivalence relation in the set of null chains. Each equivalence class is called a \textit{prime end} of $\Omega$. From now on, the set of all prime ends of $\Omega$ will be denoted by $\partial_C\Omega$ (also known as \textit{Carath\'{e}odory boundary}). For a prime end $\xi$ of $\Omega$ with a representative null chain $\{C_n\}$, its \textit{impression} is given by the formula
\begin{equation*}
    I(\zeta)\vcentcolon=\bigcap_{n=2}^{+\infty}\overline{\mathrm{Int}C_n},
\end{equation*}
where the closure is taken with respect to $\hat{\mathbb{C}}$. Note that due to \eqref{eq:equivalence relation}, the definition of impression does not depend on the representative of $\xi$.

It turns out that prime ends constitute the cornerstone in the attempts of extending a Riemann map to the boundary. Indeed, given a Riemann mapping $f\vcentcolon\mathbb{D}\to\Omega$, Carath\'{e}odory's Theorem asserts that $f$ can be extended to a homeomorphism between $\mathbb{D}\cup\partial\mathbb{D}$ and $\Omega\cup\partial_C\Omega$. In other words, the points of the unit circle come into a one to one correspondence with the prime ends of $\Omega$. The same correspondence does not hold, in general, for the sets $\partial\mathbb{D}$ and $\partial\Omega$. Of course, Carath\'{e}odory's Theorem can be generalized for any mapping $f\vcentcolon\Omega_1\to\Omega_2$, where $\Omega_1,\Omega_2\subsetneq\mathbb{C}$ are two simply connected domains and $f$ maps $\Omega_1$ conformally onto $\Omega_2$. Indeed, this $f$ induces the homeomorphism $f\vcentcolon\Omega_1\cup\partial_C\Omega_1\to\Omega_2\cup\partial_C\Omega_2$. Note that in similar circumstances, we are going to use the same symbol for both the initial mapping and its extension to the boundary.

This extension allows us to express certain notions which are usual for the unit circle in the setting of abstract boundaries. For example, if $f$ is a Riemann mapping of a simply connected domain $\Omega$, then a set of prime ends $E\subseteq\partial_C\Omega$ is said to be \textit{Borel, connected or compact}, if the set $f^{-1}(E)\subseteq\partial\mathbb{D}$ is Borel, connected or compact, respectively, in the usual topology of the unit circle.

Finally, we need a way of distinguishing the convergence to a prime end from that to a boundary point. More specifically, let $\{z_n\}\subset\mathbb{D}$ be a sequence converging to $\zeta\in\partial\mathbb{D}$. By Carath\'{e}odory's Theorem, $\zeta$ corresponds to the prime end $f(\zeta)\in\partial_C\Omega$. Therefore, the sequence $\{f(z_n)\}\in\Omega$ clusters on the impression $I(f(\zeta))$. It is quite possible that this impression is not a singleton and that $\{f(z_n)\}$ might land on a certain point $p\in I(f(\zeta))$. In such a scenario, we will say that $\{f(z_n\}$ converges to $p$ (in the Euclidean sense), whereas $\{f(z_n)\}$ converges to $f(\zeta)$ in the Carath\'{e}odory topology of $\Omega$. The same holds for a curve in the place of a sequence. 

Before ending the section, we provide an example with some terminology which seems useful in the sequel.
\begin{example}\label{ex:prime ends}
    Let $\Omega$ be the horizontal strip $S\vcentcolon=\{w\in\mathbb{C}\vcentcolon 0<\mathrm{Im}w<3\}$ minus the horizontal half-lines $L^+\vcentcolon=\{w\in\mathbb{C}\vcentcolon \mathrm{Re}w\le0, \, \mathrm{Im}w=1\}$ and $L^-\vcentcolon=\{w\in\mathbb{C}\vcentcolon \mathrm{Re}w\ge0, \, \mathrm{Im}w=2\}$. So $\Omega=S\setminus(L^+\cup L^-)$ and we have a shape similar to that described in the Introduction; see Figure \ref{fig:two_slits_h_0}. We will call the set $L^+$ a \textit{left slit} or a \textit{left half-line} and the set $L^-$ a \textit{right slit} or a \textit{right half-line}. In addition, the points $i$ and $2i$ will be called the \textit{tip points} of the left and right slit, respectively. These notions extend naturally in case we had any other horizontal strip and we excluded any number of half-lines. Furthermore, it may be checked that we can create four pairwise non-equivalent null chains in $\Omega$ with impression $\infty$. As a result, $\Omega$ has four distinct prime ends with impression $\infty$. We denote by $\infty_1^+$ the one defined through crosscuts with imaginary parts larger than $1$ and real parts decreasing to $-\infty$, by $\infty_2^+$ the one defined by imaginary parts less than $1$ and real parts decreasing to $-\infty$, by $\infty_1^-$ the one defined by imaginary parts less than $2$ and real parts increasing to $+\infty$, and by $\infty_2^-$ the remaining one. The prime ends $\infty_1^+,\infty_2^+$ are the \textit{left prime ends of} $\Omega$ (with impression infinity), whereas $\infty_1^-,\infty_2^-$ are the \textit{right prime ends of} $\Omega$ (with impression infinity). Again, left and right prime ends extend naturally for any horizontal strip and any number of slits. Moreover, each point of $L^+$ excluding its tip point corresponds to two distinct prime ends of $\Omega$. The same remark is true for the points of $L^-$ minus its tip point. The construction dictates that the curve $\gamma\vcentcolon(0,+\infty)\to\Omega$ with $\gamma(t)=i+t$ converges to the prime end $\infty_1^-$ in the Carath\'{e}odory topology of $\Omega$, as $t\to+\infty$. Since the impression of this prime end is the singleton $\{\infty\}$, we may say that $\gamma$ converges to $\infty_1^-$ or to a right prime end, without explicitly mentioning the Carath\'{e}odory topology. Finally, due to the existence of $\gamma$, we will write that the tip point $i$ \textit{accesses} the prime end $\infty_1^-$ since there exists a (curve with trace a) horizontal half-line ``joining'' $i$ to $\infty_1^-$. Similarly, the tip point $2i$ accesses $\infty_2^+$. The definition of access may be also used for any point $w\in\Omega$. Note that each $w\in\Omega$ may access at most two prime ends of $\Omega$ with impression infinity.
\end{example}

\subsection{Harmonic measure}\label{sub:harmonic measure}

During the course of our last proof, we are going to utilize one conformal invariant, the \textit{harmonic measure}. The harmonic measure is a potential theoretic tool with impressive capabilities. An introductory presentation of its rich theory may be found in \cite{harmonic}. For the purposes of the present article, we will only review some basic facts. Let $\Omega\subsetneq\mathbb{C}$ be a domain with non-polar boundary (i.e. of positive logarithmic capacity). Let $E$ be a Borel subset of $\partial\Omega$. Then, the harmonic measure of $E$ with respect to $\Omega$ is exactly the solution of the generalized Dirichlet problem for the Laplacian in $\Omega$ with boundary function equal to $1$ on $E$ and to $0$ on $\partial\Omega\setminus E$. For the harmonic measure of $E$ with respect to $\Omega$ and for $z\in\Omega$, we use the notation $\omega(z,E,\Omega)$. It is known that for a fixed $z\in\Omega$, $\omega(z,\cdot,\Omega)$ is a Borel probability measure on $\partial\Omega$.
\begin{remark}\label{rem:Brownian}
  An important piece of information with regard to the harmonic measure is its probabilistic interpretation. Although we do not need it in this work, it serves nicely as an intuition when utilizing this machinery. In particular, the quantity $\omega(z,E,\Omega)$ represents the probability of a Brownian motion starting at $z$ to exit the domain $\Omega$ for the first time passing through $E$. This thought, combined with subsequent examples justifies intuitively our usage of harmonic measure in our study during Section \ref{sec:angles}. 
\end{remark}

As we already mentioned, the harmonic measure is conformally invariant. This allows us to use sets of prime ends. More specifically, if $\Omega\subsetneq\mathbb{C}$ is a simply connected domain, $E\subseteq\partial_C\Omega$ is Borel, and $f\vcentcolon\mathbb{D}\to\Omega$ is a Riemann mapping, then 
\begin{equation*}
    \omega(z,E,\Omega)\vcentcolon=\omega(f^{-1}(z),f^{-1}(E),\mathbb{D}), \quad \text{for all }z\in\Omega,
\end{equation*}
where $f^{-1}(E)$ is well-defined due to the extension of $f$ as a homeomorphism by Carath\'{e}odory's Theorem. In addition, the harmonic measure has a very useful monotonicity property. More specifically, let $\Omega_1\subset\Omega_2\subsetneq\mathbb{C}$ be two domains with non-polar boundaries and let $E\subset\partial\Omega_1\cap\partial\Omega_2$ (or $E\subset\partial_C\Omega_1\cap\partial_C\Omega_2$) be Borel. Then,
\begin{equation*}
    \omega(z,E,\Omega_1)\le\omega(z,E,\Omega_2), \quad\text{for all }z\in\Omega_1.
\end{equation*}
In particular, this monotonicity can be made more precise by means of the so-called \textit{Strong Markov Property}. Indeed, if $\Omega_1,\Omega_2,E$ are as above, then by \cite[p.88]{PS} we know that
\begin{equation}\label{eq:Markov}
    \omega(z,E,\Omega_2)=\omega(z,E,\Omega_1)+\int_{\partial\Omega_1\cap\Omega_2}\omega(\zeta,E,\Omega_2)\omega(z,d\zeta,\Omega_2),
\end{equation}
for all $z\in\Omega_1$. 

Next, we provide an example that will be useful in subsequent sections.
\begin{example}\label{ex:zero harmonic measure}
 Let $[a,b]\subset\mathbb{R}$ and let $z\in\mathbb{H}$. Then, by \cite[p.100]{harmonic}, 
\begin{equation}\label{eq:harmonic measure half-plane}
 \omega(z,[a,b],\mathbb{H})=\frac{1}{\pi}\Arg\left(\frac{z-b}{z-a}\right). 
 \end{equation}
As a consequence, given any curve $\gamma:[0,+\infty)\to\mathbb{H}$ whose cluster set, as $t\to+\infty$, is contained in $(-\infty,a)\cup(b,+\infty)\cup\{\infty\}$, it can be easily calculated that $\lim_{t\to+\infty}\omega(\gamma(t),[a,b],\mathbb{H})=0$. Through this easy example we may extract a very useful property via the conformal invariance of the harmonic measure. Let $\Omega\subsetneq\mathbb{C}$ be a simply connected domain and $E\subseteq\partial_C\Omega$ be Borel. If $\eta\vcentcolon[0,+\infty)\to\Omega$ is a curve which clusters in the Carath\'{e}odory topology of $\Omega$ on a set $\Lambda$ compactly contained in $\partial_C\Omega\setminus E$, then 
\begin{equation*}
    \lim_{t\to+\infty}\omega(\eta(t),E,\Omega)=0.
\end{equation*}
This can be achieved by using a conformal mapping of $\mathbb{H}$ onto $\Omega$ and realizing that by construction this conformal mapping can be chosen so that the preimage of $E$ is contained in an interval $[a,b]\subset\mathbb{R}$, while the cluster set of $\eta$ has preimage in $(-\infty,a)\cup(b,+\infty)\cup\{\infty\}$. Then, \eqref{eq:harmonic measure half-plane} yields the desired limit. Observe that in \eqref{eq:harmonic measure half-plane}, the closed interval $[a,b]$ may be replaced by the open interval $(a,b)$ (or any of the intervals $(a,b]$, $[a,b)$) and the result still holds because the singletons $\{a\}$ and $\{b\}$ have zero logarithmic capacity.
   
\end{example}

\begin{remark}\label{rem:harmonic measure angles}
    Suppose that $E\subset\partial\mathbb{D}$, where $\mathbb{D}$ is the unit disk, is a circular arc with endpoints $a,b$. Then, we know (see e.g. \cite[p.155]{carath}) that the level set
    \begin{equation*}
       D_k=\left\{z\in\mathbb{D}:\omega(z,E,\mathbb{D})=k\right\}, \quad k\in(0,1), 
    \end{equation*}
    is a circular arc (or a diameter in case $k=\frac{1}{2}$ and $E$ is a half-circle) inside $\mathbb{D}$ with endpoints $a,b$ that intersects $\partial\mathbb{D}$ with angles $k\pi$ and $(1-k)\pi$. Consider $\gamma\vcentcolon[0,+\infty)\to\mathbb{D}$ be a curve converging, as $t\to+\infty$, to $\zeta\in\partial\mathbb{D}$. We say that $\gamma$ converges to $\zeta$ \textit{by angle} $\theta\in[0,\pi]$ provided $\lim_{t\to+\infty}\arg(1-\overline{\zeta}\gamma(t))=\frac{\pi}{2}-\theta$. Let $E$ be the half-circle with endpoints $\pm\zeta$ lying to the left of $\gamma([0,+\infty))$ if we consider this trace oriented towards $\zeta$. Then, by the level set we described above, we get that $\gamma$ converges to $\zeta$ by angle $\theta$ if and only if 
    \begin{equation}\label{eq:harmonic measure angle}
        \lim_{t\to+\infty}\omega(\gamma(t),E,\mathbb{D})=\frac{\theta}{\pi}.
    \end{equation}
    Next, we want to transcend this remark to the setting of the upper half-plane. Consider a M\"{o}bius automorphism mapping the unit disk conformally onto $\mathbb{H}$ and specifically the point $\zeta$ to $x\in\mathbb{R}$ and the point $-\zeta$ to $\infty$. Let $\eta\vcentcolon[0,+\infty)\to\mathbb{H}$ denote the image of $\gamma$. Hence $\lim_{t\to+\infty}\eta(t)=x$. Since conformal mappings preserve orientations, the half-circle $E$ is mapped onto the interval $[x,+\infty)$ (or $(x,+\infty)$ depending on whether we consider $E$ to contain $\zeta$ or not, something that does not affect the harmonic measure since singletons are polar sets). So, in principle, $\eta$ converges to $x$ by angle $\theta$ if and only if
    \begin{equation*}
        \lim_{t\to+\infty}\omega(\eta(t),[x,+\infty),\mathbb{H})=\frac{\theta}{\pi},
    \end{equation*}
    in view of \eqref{eq:harmonic measure angle}. In this scenario, the convergence by angle $\theta$ yields $\lim_{t\to+\infty}\arg(\eta(t))=\pi-\theta$. Nevertheless, the usual considerations of angles in the upper half-plane is the reverse, since it makes more sense to define convergence by angle $\theta$ as $\lim_{t\to+\infty}\arg(\eta(t))=\theta$. For this reason, we will follow this convention and say that $\eta$ converges by angle $\theta$ if and only if
    \begin{equation}\label{eq:angles in half-plane}
        \lim_{t\to+\infty}\omega(\eta(t),[x,+\infty),\mathbb{H})=\frac{\pi-\theta}{\pi}.
    \end{equation}
    As a result, thinking reversely, if we want to compute the angle of convergence of $\gamma$, we just need to compute the quantity $\pi(1-\lim_{t\to+\infty}\omega(\eta(t),[x,+\infty),\mathbb{H}))$. Of course, if the actual limit does not exist, then $\eta$ does not converge by a certain angle, but has a slope equal to a continuum in $[0,\pi]$. But for the purposes of this article, we will not deal with such pathological cases.
\end{remark}

\begin{example}\label{ex:strip}
    We end the section with a final example that is going to be essential. Let $S=\{w\in\mathbb{C}\vcentcolon a<\mathrm{Im}w<b\}$ be a horizontal strip and denote by $\partial S^+$ and $\partial S^-$ the upper and lower, respectively, horizontal lines bounding $S$. Then, through a conformal mapping that sends $S$ to the upper half-plane and $\partial S^+$ to some interval $(a,b)\in\mathbb{R}$, we may use \eqref{eq:harmonic measure half-plane} and the conformal invariance of the harmonic measure (also cf. \cite[p.100]{harmonic}) to compute 
    \begin{equation}\label{eq:harmonic measure strip upper}
        \omega(z,\partial S^+,S)=\frac{\mathrm{Im}z-a}{b-a}, \quad\text{for all }z\in S.
    \end{equation}
     Since the harmonic measure is a Borel probability measure, we have $\omega(z,\partial S^-,S)=1-\omega(z,\partial S^+,S)$ which leads to
    \begin{equation}\label{eq:harmonic measure strip lower}
        \omega(z,\partial S^-,S)=\frac{b-\mathrm{Im}z}{b-a}, \quad\text{for all }z\in S.
    \end{equation}
\end{example}

\section{Construction of the model domain}\label{sec:construction}
We start off with our model domain. As we mentioned earlier, we wish to construct a chain with multiple attraction points on the boundary. Intuitively, we build it based on the following simple idea. 

In the classical setting of non-elliptic semigroups, one has domains $\Omega=h(\mathbb{H})$, for some $h$ (the Koenigs map), that are convex in the right direction. This means that for every $w\in\Omega$, the horizontal half-line $\{w+t:t\ge0\}$ lies in $\Omega$. Clearly, such a domain $\Omega$ only has one prime end with impression infinity defined by crosscuts with strictly increasing real parts. Then, for each $w\in\Omega$, $w+t$ converges to this right prime end of infinity, as $t\to+\infty$, and hence $h^{-1}(w+t)$ converges at some point $\tau\in\mathbb{R}\cup\{{\infty}\}$, as $t\to+\infty$. For instance, consider a strip minus horizontal half-lines extending to the left. Keeping this in mind, in order to construct a chain with multiple attraction points, we need multiple prime ends of infinity on the right. In this way, the various half lines $\{w+t:t\ge0\}$, $w\in\Omega$, extend to different prime ends, depending on their respective heights. As a consequence, we examine a domain with multiple prime ends of infinity on the left and multiple prime ends of infinity on the right. Such a domain is examined below.

\begin{proposition}\label{h for multiple points proof}
    Let $(k_j^+)_{j=1}^N$, $(k_j^-)_{j=1}^M$ be points on the real line in increasing order with $k_N^+<k_1^-$ and let $(b_j^+)_{j=1}^N$, $(b_j^-)_{j=1}^M$ be some positive numbers such that $\sum_{j=1}^N b_j^+>\sum_{j=1}^M b_j^-$. For the principal branch of the logarithm, consider the function
    \begin{equation}\label{h for multiple points}
        h(z)\vcentcolon=\sum_{j=1}^Nb_j^+\log(z-k_j^+)-\sum_{j=1}^Mb_j^-\log(z-k_j^-), \quad z\in\mathbb{H}.
    \end{equation}
    Then, $h$ is a close-to-convex function of the upper half-plane.
\end{proposition}

\begin{proof}
    Taking into account that $k_j^+<k_i^-$, for any pair of indices $i,j$, we consider two points $s,\tau\in\mathbb{R}$ satisfying $\tau<k_j^+<s<k_i^-$, for all $i,j$. Evidently, 
    \begin{equation}\label{eq:close to convex 1}
        (k_j^+-\tau)(k_j^+-s)<0 \quad \text{and} \quad (k_j^--\tau)(k_j^--s)>0,
    \end{equation}
    for any choice of indices $i,j$. Define the function $l(z)\vcentcolon=\log(\frac{s-z}{\tau-z})$ which maps the upper half-plane conformally onto a strip. Hence $l\in\mathcal{C}_{\mathbb{H}}$, since $l(\mathbb{H})$ is convex. Our objective is to show that the imaginary part of $\frac{h'}{l'}$ has constant sign in $\mathbb{H}$, which in turn will imply the desired result. Differentiating $h$ and $l$, we get that

    \begin{align}\label{eq:close to convex 2}
    \notag    &\frac{h'(z)}{l'(z)}(s-\tau)=\sum_{j=1}^Nb_j^+\frac{(z-\tau)(z-s)}{z-k_j^+}-\sum_{j=1}^Mb_j^-\frac{(z-\tau)(z-s)}{z-k_j^-}\\
     \notag     &=\sum_{j=1}^Nb_j^+\frac{(z-k_j^++k_j^+-\tau)(z-k_j^++k_j^+-s)}{z-k_j^+}-\sum_{j=1}^Mb_j^-\frac{(z-k_j^-+k_j^--\tau)(z-k_j^-+k_j^--s)}{z-k_j^-}\\
     \notag    &=\sum_{j=1}^Nb_j^+(z-k_j^+)+\sum_{j=1}^Nb_j^+(2k_j^+-\tau-s)+\sum_{j=1}^Nb_j^+\frac{(k_j^+-\tau)(k_j^+-s)}{z-k_j^+}\\
         &\hspace{30mm} -\sum_{j=1}^Mb_j^-(z-k_j^-)-\sum_{j=1}^Mb_j^-(2k_j^--\tau-s)-\sum_{j=1}^Mb_j^-\frac{(k_j^--\tau)(k_j^--s)}{z-k_j^-}.
    \end{align}
Equating the imaginary parts in \eqref{eq:close to convex 2}, we get 
\begin{equation*}
   \frac{s-\tau}{\mathrm{Im}z}\mathrm{Im}\left(\frac{h'(z)}{l'(z)}\right)=\sum_{j=1}^Nb_j^+-\sum_{j=1}^Mb_j^--\sum_{j=1}^Nb_j^+\frac{(k_j^+-\tau)(k_j^+-s)}{|z-k_j^+|^2}+\sum_{j=1}^Nb_j^-\frac{(k_j^--\tau)(k_j^--s)}{|z-k_j^-|^2}.
\end{equation*}
Given that $s-\tau>0$, our hypothesis on the sets $(b_j^+)_{j=1}^N, (b_j^-)_{j=1}^M$ along with relation \eqref{eq:close to convex 1} yield that $\frac{h'}{l'}$ has positive imaginary part in $\mathbb{H}$. In view of Definition \ref{def:close to convex}, we see that $h$ is close-to-convex. 
\end{proof}

Since $h$ is univalent and continuous up to the boundary (except for the points $k_j^+,k_i^-$), we can describe its image by looking at its boundary behaviour. To achieve this, it is necessary to inspect the location of the roots of $h'$. It can be seen directly that $h'$ is a rational function with numerator the polynomial
\begin{equation}\label{polynomial}
   P(z)=\sum_{j=1}^Nb_j^+\prod_{i\neq j}(z-k_i^+)\prod_{i=1}^M(z-k_i^-)-\sum_{j=1}^Mb_j^-\prod_{i\neq j}(z-k_i^-)\prod_{i=1}^N(z-k_i^+).
\end{equation}
This is a polynomial of degree $N+M-1$, meaning it has exactly this many roots counting multiplicity. Observe that each interval $(k_j^+,k_{j+1}^+)$, $1\le j\le N-1$, and $(k_j^-,k_{j+1}^-)$, $1\le j\le M-1$, contains at least one root of $P$, and hence of $h'$. This is because $P$ alternates sign at the endpoints of each of the above intervals. This means that there exist $N+M-2$ many roots $\rho_j^+\in(k_j^+,k_{j+1}^+)$ and $\rho_j^-\in(k_j^-,k_{j+1}^-)$ of odd multiplicity. Therefore, if even one of them is multiple, then we would have at least $N+M$ roots counting multiplicity, exceeding the degree of $P$. Hence, all of the above $N+M-2$ many distinct roots are simple. Due to the continuity, we understand that there are no other roots in these intervals. Indeed, since $P$ alternates sign at the endpoints, if there exists another root $\rho_j'$, it is either a multiple root, or there exists at least another $\rho_j''$, violating again the fact that $P$ has degree $N+M-1$. For the same reason, observe that the interval $(k_N^+,k_1^-)$ cannot contain a root either. Finally, since $P$ has real coefficients, the remaining root lies in one of the intervals $(-\infty,k_1^+)$, $(k_M^-,+\infty)$. Assuming that this root lies in $(-\infty,k_1^+)$, we deduce that $P(x)>0$ for all $x>k_M^-$ and as a result $h$ is increasing in $(k_M^-,+\infty)$. However, we have that $\lim_{x\rightarrow (k_M^-)^+}h(x)=+\infty$ and because $\sum_{j=1}^N b_j^+>\sum_{j=1}^M b_j^-$, we have that 
\begin{equation*}
  \lim_{x\rightarrow+\infty}h(x)=\lim_{x\rightarrow+\infty}\log\left(\frac{\prod_{j=1}^N(x-k_j^+)^{b_j^+}}{\prod_{j=1}^M(x-k_j^-)^{b_j^-}}\right)=+\infty,  
\end{equation*}
contradicting the monotonicity of $h$. Thus, the remaining root of $P$, say $\rho_M^-$, lies in $(k_M^-,+\infty)$.

We, now, work on the boundary values of $h$. Firstly, note that the imaginary part of $h(x)$ is constant in each of the aforementioned intervals, with 
\begin{equation}
    \label{imaginary part of h in the positives}
    \mathrm{Im}(h(x))=\sum_{j=\nu+1}^{N}b_j^+\pi-\sum_{j=1}^{M}
    b_j^-\pi, \quad x\in(k_{\nu}^+,k_{\nu+1}^+),\quad 1\le\nu\le N,
\end{equation}
\begin{equation}
    \label{imaginary part of h in the negatives}
    \mathrm{Im}(h(x))=-\sum_{j=\mu+1}^M
    b_j^-\pi, \quad x\in(k_{\mu}^-,k_{\mu+1}^-),\quad 1\le\mu\le N,\quad \mathrm{Im}(h(x))=0,\quad x>k_M^-,
\end{equation}
\begin{equation}
    \label{imaginary part of h in between}
    \mathrm{Im}(h(x))=\sum_{j=1}^Nb_j^+\pi-\sum_{j=1}^M
    b_j^-\pi, \quad x<k_1^+, \quad \mathrm{Im}(h(x))=-\sum_{j=1}^M
    b_j^-\pi, \quad x\in(k_N^+,k_1^-).
\end{equation}
In addition, we easily see that $\lim_{x\rightarrow k_j^+}\mathrm{Re}(h(x))=-\infty$, $\lim_{x\rightarrow k_i^-}\mathrm{Re}(h(x))=+\infty$, for all indices $i,j$. Moreover, $\mathrm{Re}(h(x))\rightarrow+\infty$, as $x\rightarrow\pm\infty$.. Combining everything together, we deduce that $h$ maps the upper half-plane onto a strip minus $N-1$ half-lines extending to the point at infinity from the left and $M$ half-lines extending to the point at infinity from the right. Now, since the zeroes of $h'$ are exactly the points $\rho_{\nu}^+,\rho_{\mu}^-$, each half-line has its tip point at $h(\rho_{\nu}^+)$ and  $h(\rho_{\mu}^-)$, for $\nu=1,\dots,N-1$ and $j=1,\dots,M$ as we see in Figure \ref{fig:image of h}. For the sake of simplicity, we sometimes use the notation as
$b^+=\sum_{j=1}^Nb_j^+$ and 
    $b^-=\sum_{i=1}^M b_i^-$.

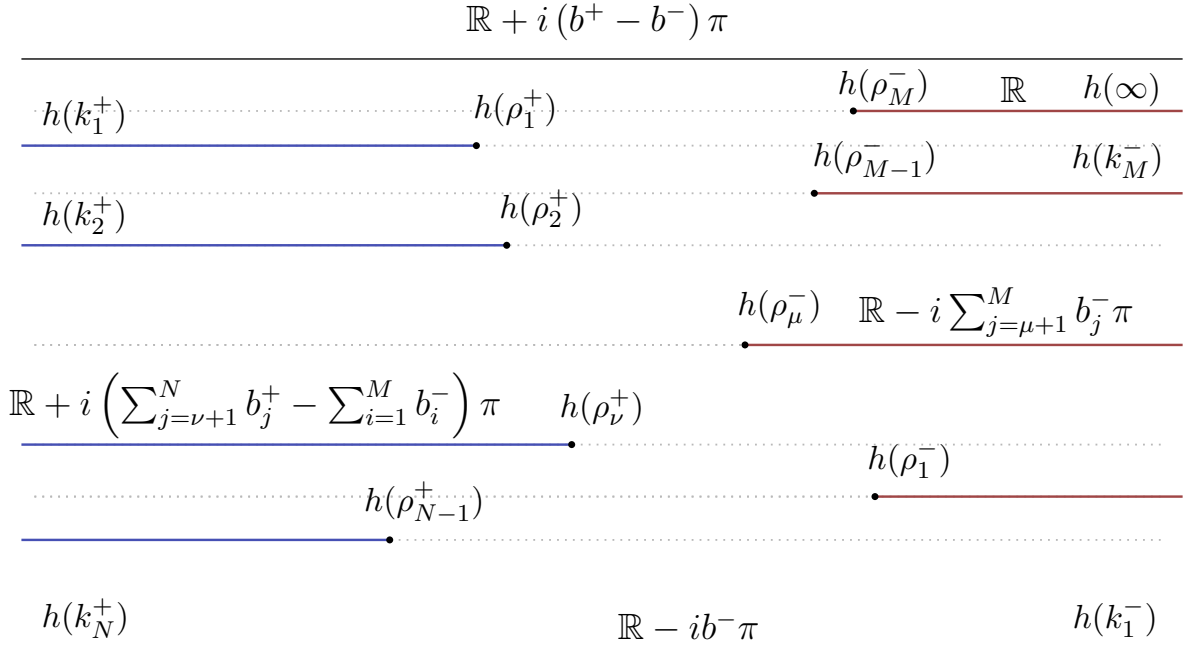
\begin{figure}[ht]
\centering
\resizebox{0.95\linewidth}{!}{%
\begin{tikzpicture}[
    x=1cm,y=1cm,
    every node/.style={font=\normalsize}
]

\definecolor{myblue}{RGB}{70,80,180}
\definecolor{myred}{RGB}{160,70,70}
\definecolor{mygray}{RGB}{185,185,185}

\def\ytop{7.25}
\def\ybottom{0.35}

\def\yA{6.25} 
\def\yB{5.10} 
\def\yC{3.95} 
\def\yD{2.80} 
\def\yE{2.20} 
\def\yF{1.70} 

\draw[black, line width=0.45pt] (0.5,\ytop) -- (13.9,\ytop);
\draw[black, line width=0.45pt] (0.5,\ybottom) -- (13.9,\ybottom);

\node[font=\large] at (7.15,7.68)
{$\mathbb R+i\left(b^+-b^-\right)\pi$};

\node[font=\large] at (8.20,0.70)
{$\mathbb R-ib^-\pi$};

\foreach \Y in {\yA,\yB,\yC,\yD,\yE,\yF}{
    \draw[mygray, dotted, line width=0.6pt] (0.65,\Y) -- (13.65,\Y);
}
\draw[mygray, dotted, line width=0.6pt] (0.65,6.65) -- (10.10,6.65);
\draw[mygray, dotted, line width=0.6pt] (0.65,5.70) -- (9.65,5.70);
\draw[mygray, dotted, line width=0.6pt] (0.65,\yC) -- (8.85,\yC);
\draw[mygray, dotted, line width=0.6pt] (0.65,\yE) -- (10.35,\yE);

\draw[myblue, line width=0.8pt] (0.5,\yA) -- (5.75,\yA);
\fill[black] (5.75,\yA) circle (1.15pt);

\draw[myblue, line width=0.8pt] (0.5,\yB) -- (6.10,\yB);
\fill[black] (6.10,\yB) circle (1.15pt);

\draw[myblue, line width=0.8pt] (0.5,\yD) -- (6.85,\yD);
\fill[black] (6.85,\yD) circle (1.15pt);

\draw[myblue, line width=0.8pt] (0.5,\yF) -- (4.75,\yF);
\fill[black] (4.75,\yF) circle (1.15pt);

\draw[myred, line width=0.8pt] (10.10,6.65) -- (13.90,6.65);
\fill[black] (10.10,6.65) circle (1.15pt);

\draw[myred, line width=0.8pt] (9.65,5.70) -- (13.90,5.70);
\fill[black] (9.65,5.70) circle (1.15pt);

\draw[myred, line width=0.8pt] (8.85,\yC) -- (13.90,\yC);
\fill[black] (8.85,\yC) circle (1.15pt);

\draw[myred, line width=0.8pt] (10.35,\yE) -- (13.90,\yE);
\fill[black] (10.35,\yE) circle (1.15pt);

\node[anchor=west] at (0.60,\yA+0.35) {$h(k_1^+)$};
\node[anchor=west] at (0.60,\yB+0.35) {$h(k_2^+)$};

\node[font=\large] at (3.20,\yD+0.45)
{$\mathbb R+i\left(\sum_{j=\nu+1}^{N} b_j^{+}-\sum_{i=1}^{M} b_i^{-}\right)\pi$};

\node[anchor=west] at (0.60,\ybottom+0.45) {$h(k_N^+)$};

\node at (6.20,\yA+0.43) {$h(\rho_1^+)$};
\node at (6.50,\yB+0.43) {$h(\rho_2^+)$};
\node at (7.20,\yD+0.43) {$h(\rho_\nu^+)$};
\node at (5.15,\yF+0.43) {$h(\rho_{N-1}^+)$};

\node at (10.45,6.95) {$h(\rho_M^-)$};
\node[font=\large] at (11.95,6.9) {$\mathbb R$};
\node at (13.20,6.9) {$h(\infty)$};

\node at (10.35,6.10) {$h(\rho_{M-1}^-)$};
\node at (13.15,6.10) {$h(k_M^-)$};

\node at (9.25,\yC+0.43) {$h(\rho_\mu^-)$};
\node[font=\large] at (11.75,\yC+0.38)
{$\mathbb R-i\sum_{j=\mu+1}^{M} b_j^{-}\pi$};

\node at (10.75,\yE+0.43) {$h(\rho_1^-)$};
\node[anchor=east] at (13.75,\ybottom+0.45) {$h(k_1^-)$};

\end{tikzpicture}%
}
\caption{The image of \(h\). Each point at infinity is denoted as $h(k_{\nu}^+)$ on the left and $h(k_{\mu}^-)$ on the right since the points $k_{\nu}^+,k_{\mu}^-$ correspond to the above prime ends.}
\label{fig:image of h}
\end{figure}

It is, now, evident that the point at infinity is the impression of $N+M+1$ prime ends of $h(\mathbb{H})$ which correspond through $h$ to the points $k_{\nu}^+$, $k_{\mu}^-$ and the point at infinity, as shown in Figure \ref{fig:image of h}. 

Note that each point $w$ can access at most two such prime ends of $\Omega$; see Example \ref{ex:prime ends}. On the other hand, we understand that a point $w$ accesses no prime ends with infinity as its impression, if there exist $\rho_{\nu}^+$ and $\rho_{{\mu}}^-$ so that $\mathrm{Im}w=\mathrm{Im}h(\rho_{\nu}^+)=\mathrm{Im}h(\rho_{\mu}^-)$. In other words, there exist two horizontal half-lines in $\partial\Omega$ with the same height and $w$ lies between them. Note that the height of each half-line depends exclusively on the choice of the positive numbers $b_j^+$, $b_i^-$, for $j=1,\dots N$ and $i=1,\dots, M$. Other than that, the heights of the half-lines depend uniquely on the numbers $b_j^+,b_i^-$, in the sense that there exists only one combination that corresponds to prescribed heights.

Keeping all the above information in mind, we make the following convention: the parameters $b_j^+$ and $b_i^-$ are chosen is such a way that every point of $\Omega$ can access at least one prime end with impression infinity. Geometrically, this means that each tip point of the half-lines can be extended to the point at infinity via a horizontal orbit. Note, also, that a prime end might be potentially accessed by more than one tip point, or even by none. 

As for the construction of the Loewner chain, it is based on the following idea.  Since each tip point can access a prime end, we wish to discover a way to move each tip point to the corresponding prime end it accesses. For example, applying the mapping $z\mapsto z+t$, $t\ge0$, the tip points $h(\rho_{\nu}^+)$ are progressively translated to the point at infinity to the right (recall that non-elliptic semigroups are written as a conjugation $h^{-1}(t+h(z))$. But in this case, the points $h(\rho_{\mu}^-)$ would be also moved to the point at infinity to the right, preventing the creation of a Loewner chain. For this reason, we need to let the points $k_j^+,k_i^-$ vary in such a way that the tip points $h(\rho_{\mu}^-)$ move to the point at infinity to the left in such a rate that overrules the push to the right, due to the translation $z\mapsto z+t$.

In the sequel, our aim is to define a proper variation for the parameters, so that the idea above can indeed be applied. Apparently, in order to consider the best possible way to `move' the parameters, we should be able to understand the dynamics of $h$, with respect to these parameters. Thus, for $\mathbf{k}\vcentcolon=(k_1^+,\dots,k_N^+,k_1^-,\dots,k_M^-)\in\mathbb{R}^{N+M}$, we are going to write $h=h_{\mathbf{k}}$ for the function defined via the formula in \eqref{h for multiple points}. This signifies that the roots of $h'=h'_{\mathbf{k}}$ are also functions of $\mathbf{k}$, and thus $\rho_{\nu}^+=\rho_{\nu}^+(\mathbf{k})$, $\rho_{\mu}^-=\rho_{\mu}^-(\mathbf{k})$. Next, recall that the tip points are the points $h_{\mathbf{k}}(\rho_{\nu}^+(\mathbf{k})),h_{\mathbf{k}}(\rho_{\mu}^-(\mathbf{k}))$, and that we want them to access their corresponding prime ends as $\mathbf{k}$ changes. Even though the roots cannot be found explicitly, we still need to understand how they behave with respect to $\mathbf{k}$. Towards this goal, we present certain properties of the roots. The following theorem can be proven relatively easily using the Implicit Function Theorem.
 \begin{theorem}[{\cite[Theorem 1.18]{Volklein}}]\label{continuity of the roots}
 For every index $\nu\in\{1,\dots,N-1\}$ and $\mu\in\{1,\dots,M\}$, the roots $\rho_{\nu}^+$ and $\rho_{\mu}^-$ are continuously differentiable functions of $\mathbf{k}$. In fact, the roots are of the same differentiability class, as the coefficients of the polynomial.
 \end{theorem}
 Consider the configuration of Proposition \ref{h for multiple points proof}. We then have the following result on the convergence of the roots.
 \begin{proposition}
     \label{prop:convergence of the roots}
     Suppose that $k_j^+\rightarrow k_N^+$, $2\le j\le N-1$ and $k_j^-\rightarrow k_N^+$, $1\le j\le M$. Then, $\rho_{\nu}^+(\mathbf{k}),\rho_{\mu}^-(\mathbf{k})\rightarrow k_N^+$, for all $\nu=2,\dots,N-1$ and $\mu=1,\dots,M-1$. In addition, 
     \begin{enumerate}
         \item[\textup{(1)}] if $\sum_{j=2}^Nb_j^+<\sum_{j=1}^{M}b_j^-$, then $\rho_{1}^+(\mathbf{k})\rightarrow k_N^+$ and $\lim_{\mathbf{k}}\rho_M^-(\mathbf{k})>k_N^+$
         \item[\textup{(2)}] if $\sum_{j=2}^Nb_j^+>\sum_{j=1}^{M}b_j^-$, then $\lim_{\mathbf{k}}\rho_1^+(\mathbf{k})<k_N^+$ and $\rho_{1}^+(\mathbf{k})\rightarrow k_N^+$. 
     \end{enumerate}
 \end{proposition}
\begin{proof}
    The first part of the statement follows at once, since for every $\mathbf{k}$, each root $\rho_{\nu}^+$ lies in the interval $(k_{\nu}^+,k_{\nu+1}^+)$, for all $\nu=2,\dots,N-1$, while $\rho_{\mu}^-\in(k_{\mu}^-,k_{\mu+1}^-)$, for all $\mu=1,\dots,M-1$. Hence, all that remains is to show the convergence of the remaining roots. Writing $P=P_{\mathbf{k}}$ for the polynomial in \eqref{polynomial} and setting $\mathbf{k_0}=(k_1^+,k_N^+,\dots k_N^+)$, we infer that 

    \begin{align}\label{eq:convergence of roots 1}
    \notag    P_{\mathbf{k_0}}(z)&=b_1^+(z-k_N^+)^{N+M-1}+\left(\sum_{j=2}^{N}b_j^+-\sum_{j=1}^Mb_j^-\right)(z-k_1^+)(z-k_N^+)^{N+M-2}\\
    \notag   &=(z-k_N^+)^{N+M-2}\left[b_1^+(z-k_N^+)+\left(\sum_{j=2}^{N}b_j^+-\sum_{j=1}^Mb_j^-\right)(z-k_1^+)\right]\\
         &=(z-k_N^+)^{N+M-2}\left[\left(\sum_{j=1}^Nb_j^+-\sum_{j=1}^Mb_j^-\right)z-b_1^+k_N^+-\left(\sum_{j=2}^Nb_j^+-\sum_{j=1}^Mb_j^-\right)k_1^+\right]
    \end{align}
    Therefore, the roots of $P_{\mathbf{k}_0}$ are $k_N^+$ (of multiplicity $N+M-2$, of which the multiplicity $N+M-3$ occurs due to the limits of the roots $\rho_\nu^+(\mathbf{k})$, $\nu\in\{2,\dots,N-1\}$, and $\rho_\mu^-(\mathbf{k})$, $\mu\in\{1,\dots,M-1\}$), and one more which can be computed through the bracket in \eqref{eq:convergence of roots 1}. So there are two distinct cases. Either the roots $\rho_1^+(\mathbf{k})$ converge to $k_N^+$ and the roots $\rho_{\mu}^-$ converge to the root defined by the bracket in \eqref{eq:convergence of roots 1}, or vice versa. By continuity, $\rho_1^+(\mathbf{k_0})=\lim_{\mathbf{k}}\rho_1^+(\mathbf{k})\le k_N^+$, whereas $\rho_M^-(\mathbf{k_0})=\lim_{\mathbf{k}}\rho_M^-(\mathbf{k})\ge k_N^+$. It is then easy to see that $\rho_1^+(\mathbf{k_0})<k_N^+$ (and thus $\rho_\mu^-(\mathbf{k_0})=k_N^+$) if and only if
    \begin{equation}\label{eq:convergence of roots 2}
     b_1^+k_N^++\left(\sum_{j=2}^{N}b_j^+-\sum_{j=1}^Mb_j^-\right)k_1^+<\left(\sum_{j=1}^Nb_j^+-\sum_{j=1}^Mb_j^-\right)k_N^+,  
    \end{equation}
    where we have also used the initial assumption $\sum_{j=1}^Nb_j^+>\sum_{j=1}^Mb_j^-$. Through algebraic considerations and recalling that $k_1^+<k_N^+$, we see that \eqref{eq:convergence of roots 2} is equivalent to 
    \begin{equation*}
    \sum_{j=2}^{N}b_j^+-\sum_{j=1}^Mb_j^->0,   
    \end{equation*}
    which yields the desired outcome.    
\end{proof}

 In what follows, we will always assume that the first case of the preceding proposition holds. Geometrically, this implies that the $M$-th right half-line is the highest of all the half-lines, as we see in Figure \ref{fig:image of h}. In addition, we have the following proposition regarding the monotonicity of the roots.

 \begin{proposition}\label{prop:roots with respect to kj}
     For each $\nu\in\{1,\dots,N-1\}$, we have
     \begin{equation*}
         \frac{\partial\rho_{\nu}^+}{\partial k_j^+}>0,\quad j=1,\dots,N-1\quad\text{and}\quad \frac{\partial\rho_{\nu}^+}{\partial k_j^-}<0,\quad j=1,\dots,M. 
     \end{equation*}
    Likewise, for each $\mu\in\{1,\dots,M\}$, we have 
     \begin{equation*}
     \frac{\partial\rho_{\mu}^-}{\partial k_j^+}<0,\quad j=1,\dots,N-1\quad\text{and}\quad \frac{\partial\rho_{\mu}^-}{\partial k_j^-}>0,\quad j=1,\dots,M.    
     \end{equation*}
\end{proposition}
\begin{proof}
     Fix $\nu\in\{1,\dots,N-1\}$ and an index $j=1,\dots,N-1$. Recall that for any $\mathbf{k}$, $h_{\mathbf{k}}'(\rho_{\nu}^+(\mathbf{k}))=0$. Hence, differentiating with respect to $k_j^+$ and using the chain rule, we obtain
     \begin{equation}\label{eq:chain rule}
       h''_{\mathbf{k}}(\rho_{\nu}^+(\mathbf{k}))\frac{\partial\rho_{\nu}^+}{\partial k_j^+}+\frac{\partial h'_{\mathbf{k}}}{\partial k_j^+}(\rho_{\nu}^+(\mathbf{k}))=0.    
     \end{equation}
   Due to the facts that $\rho_{\nu}^+$ is the unique and simple root of $h'_{\mathbf{k}}$ in the interval $(k_{\nu}^+,k_{\nu+1}^+)$ while $\lim_{x\rightarrow (k_{\nu}^+)^+}h'_{\mathbf{k}}(x)=+\infty=-\lim_{x\rightarrow (k_{\nu+1}^+)^-}h'_{\mathbf{k}}(x)$, we comprehend that $h''_{\mathbf{k}}(\rho_{\nu}^+(\mathbf{k}))<0$. On the other, by \eqref{h for multiple points}, we see that
   \begin{equation*}
       h'_{\mathbf{k}}(z)=\sum_{\nu=1}^N\frac{b_\nu^+}{z-k_\nu^+}-\sum_{\mu=1}^M\frac{b_\mu^-}{z-k_\mu^-}.
   \end{equation*}
   Differentiating with respect to $k_j^+$ we find at once that
   \begin{equation*}
       \frac{\partial h'_{\mathbf{k}}}{\partial k_j^+}(x)=\frac{b_j^+}{(x-k_j^+)^2}>0, \quad \text{for all }x\in\mathbb{R}\setminus\{k_j^+\}.
   \end{equation*}
   Applying all the aforementioned information on \eqref{eq:chain rule}, we get $\frac{\partial\rho_{\nu}^+}{\partial k_j^+}>0$. In similar fashion, for any index $j=1,\dots, M$, we get that $\frac{\partial\rho_{\nu}^+}{\partial k_j^-}<0$. The second part of the statement is deduced by arguments akin to those above, albeit this time $h''_{\mathbf{k}}(\rho_{\mu}^-(\mathbf{k}))>0$. We omit the details for the sake of avoiding repetition.
 \end{proof}
By Proposition \ref{prop:roots with respect to kj} and the chain rule, we directly deduce the following corollary.
\begin{corollary}
    Suppose that the parameters $k^+_j$ and $k_i^-$ are differentiable functions of $t\in[0,+\infty)$, so that $(k^+_j)'(t)>0$ and $(k^-_j)'(t)<0$,for all $t\in[0,+\infty)$. Then, for any $t\in[0,+\infty)$, $(\rho_{\nu}^+)'(t)>0$, for every $\nu=1,\dots,N-1$ and $(\rho_{\mu}^-)'(t)<0$, for every $\mu=1,\dots,M$.
\end{corollary}

\section{Variation of the parameters}\label{sec:multiple points}
We are ready to configure our model in its full generality. As described in Section \ref{sec:construction}, we make an initial choice of points of $N+M$ points $k_1^+<\dots<k_N^+<k_1^-<\dots k_M^-$. We also make an arbitrary choice of positive numbers $(b_j^+)_{j=1}^N$ and $(b_j^-)_{j=1}^M$, so that $\sum_{j=1}^Nb_j^+>\sum_{j=1}^Mb_j^-$. 
\subsection{Setup of the parameters and related quantities}
Mimicking the model for two attraction points, as we described in the introductory section, we establish the process so that $k_j^+\rightarrow k_N^+$, for $j=2,\dots,N-1$ and $k_j^-\rightarrow k_N^+$, for $j=1,\dots,M$. To do so, we pick some positive exponents 
\begin{equation}\label{eq:thetas}
   \theta_2^+\le\dots\le\theta_{N-1}^+\quad \text{and} \quad \theta_1^-\ge\dots\ge\theta_M^-
\end{equation}
and we define the functions 
\begin{equation}\label{k(t)}
   k_1^+(t)\equiv k_1^+,\quad k_j^+(t)\vcentcolon=k_N^++(k_j^+-k_N^+)e^{-\theta_j^+t}\quad\text{and}\quad k_j^-(t)\vcentcolon=k_N^++(k_j^--k_N^+)e^{-\theta_j^-t},
\end{equation}
      for $t\ge0$. Thus, for each $t\ge0$, we obtain a vector 
      \begin{equation*}
       \mathbf{k}(t)\vcentcolon=(k_1^+,k_2^+(t),\dots,k_{N-1}^+(t),k_N^+,k_1^-(t),\dots,k_M^-(t))\in\mathbb{R}^{N+M}.   
      \end{equation*}
 The ordering of the exponents is essential because we have to maintain the initial ordering of the $k_j^{\pm}$, that is 
      \begin{equation}\label{orderings of k}
          k_1^+(t)<\dots<k_{N-1}^+(t)<k_N^+<k_1^-(t)<\dots<k_M^-(t),
      \end{equation}
for all $t\ge0$. Note that $k_j^+$ increases to $k_N^+$, while $k_j^-$ decreases to $k_N^+$, justifying the $\pm$ superscripts. Under these notations and recalling \eqref{h for multiple points}, we write 
    \begin{equation}
           \label{h_t(N,M)}
           h_t(z)\vcentcolon=h(z,t)\vcentcolon=h_{\mathbf{k}(t)}(z)=\sum_{j=1}^Nb_j^+\log(z-k_j^+(t))-\sum_{j=1}^Mb_j^-\log(z-k_j^-(t)).
     \end{equation}
Furthemore, following the previously established procedure, the roots of the z-derivative of $h$ are some functions $\rho_j^+(t)\in(k_j^+(t),k_{j+1}^+(t))$, for $1\le j\le N-1$, and $\rho_j^-(t)\in(k_j^-(t),k_{j+1}^-(t))$,  for $1\le j\le M-1$, whereas $\rho_M^-(t)>k_M^-(t)$. In other terms,
       \begin{equation}
           \label{root+-}
           h'(\rho_j^+(t),t)=0,\ 2\le j\le N-1 \quad\text{and}\quad h'(\rho_j^-(t),t)=0,\ 1\le j\le M,
       \end{equation}
for all $t\ge0$. Note that $\rho_j^{\pm}$ are continuous functions of $t$. Later on, our intention will be to measure the rates by which the tip points $h(\rho_j^{\pm}(t),t)$ diverge to infinity. For the sake of simplifying that process, we define the following quantities:
     \begin{definition}\label{rates of the tips}
         Given the above choices of parameters, the numbers
         \begin{equation}
         \label{convergence of the left tips}
          r(\rho_j^+):=\lim_{t\rightarrow+\infty}\frac{1}{t}\mathrm{Re}h(\rho_j^+(t),t), \quad \ 1\le j\le N-1,
         \end{equation}
         and 
          \begin{equation}
           \label{convergence of the right tips}r(\rho_j^-):=\lim_{t\rightarrow+\infty}\frac{1}{t}\mathrm{Re}h(\rho_j^-(t),t), \quad\ 1\le j\le M,    
          \end{equation}
         will be called \textit{rates} of the tip points.
     \end{definition}
   Recall that from Proposition \ref{prop:convergence of the roots} $\lim_{t\to+\infty}\rho_M^-(t)>k_N^+$. Therefore, $h(\rho_M^-(t),t)$ accumulates, as $t\to+\infty$, to an interior point $\Omega_0=h_0(\mathbb{H})$ which leads to $r(\rho_M^-)=0$. In what follows, we compute the rest of the rates. This is the crucial step for our model to work. Indeed, our intention is to find an appropriate factor $\gamma$ so that
   \begin{equation}
       \label{limits of the tips+}
       \lim_{t\rightarrow+\infty}(\mathrm{Re}h(\rho_j^+(t),t)-\gamma t)=+\infty, \quad 1\le j\le N-1
   \end{equation}
   and
 \begin{equation}
       \label{limits of the tips-}
       \lim_{t\rightarrow+\infty}(\mathrm{Re}h(\rho_j^-(t),t)-\gamma t)=-\infty, \quad 1\le j\le M.
   \end{equation}
In order to determine the rates of the tip points, we apparently need to determine the limiting behavior of the terms $|\rho_i^{\pm}(t)-k_j^{\pm}(t)|$, as $t\rightarrow+\infty$, for all indices of our configuration. Obviously, since $\rho_j^{\pm}$ are the roots of $h'$, no useful computational information can be deduced, to calculate these limits directly. Therefore, we need to stratagem accordingly. For instance, we might be able to determine the limiting behavior of the ratios of the form $\frac{\rho_j^{\pm}(t)-k_N^+}{\rho_j^{\pm}(t)-k_j^{\pm}(t)}$. Indeed, if for example, we fix some $\nu\in\{1,\dots,N-1\}$ , then looking at (\ref{root+-}) we get that
\begin{equation} \label{sums}
    b _1^+\frac{\rho_{\nu}^+(t)-k_N^+}{\rho_{\nu}^+(t)-k_1^+(t)}+\dots+b_{N-1}^+\frac{\rho_{\nu}^+(t)-k_N^+}{\rho_{\nu}^+(t)-k_{N-1}^+(t)}+b_N^+=b_1^-\frac{\rho_{\nu}^+(t)-k_N^+}{\rho_{\nu}^+(t)-k_1^-(t)}+\dots+b_M^-\frac{\rho_{\nu}^+(t)-k_N^+}{\rho_{\nu}^+(t)-k_M^-(t)}
\end{equation}
which provides more realistic expectations for calculating the rates of the tip points. Hence, it would be useful to put these quantities under some definitions.

It turns out that the rates themselves are hard to calculate. To bypass this obstacle, we will follow a slight detour. We are going to examine the behavior of the following quantities:
\begin{definition}\label{ratios +-}
    For each $\nu\in\{1,\dots N-1\}$, we define
    \begin{equation*}
           l_{j,\nu}^+(t):=\frac{\rho_{\nu}^+(t)-k_N^+}{\rho_{\nu}^+(t)-k_j^+(t)},\quad1\le j\le N,\quad\text{and}\quad l_{j,\nu}^-(t):=\frac{\rho_{\nu}^+(t)-k_N^+}{\rho_{\nu}^+(t)-k_j^-(t)},\quad1\le j\le M, 
    \end{equation*}
    for $t\ge0$. In an analogous manner, for each $\mu\in\{1,\dots M-1\}$ we define
    \begin{equation*}
         a_{j,\nu}^+(t):=\frac{\rho_{\mu}^-(t)-k_N^+}{\rho_{\mu}^-(t)-k_j^+(t)},\quad1\le j\le N,\quad\text{and}\quad \alpha_{j,\mu}^-(t):=\frac{\rho_{\mu}^-(t)-k_N^+}{\rho_{\mu}^-(t)-k_j^-(t)},\quad1\le j\le M,  
    \end{equation*}
    for $t\ge0$. Note that we do not define the latter quantities for $\mu=M$, since the limiting behavior of $\rho_M^-(t)$, as $t\to+\infty$, is already dealt with.
\end{definition}
  
   Under the notation of the previous definition, equation $(\ref{sums})$ becomes
   \begin{equation}
       \label{sums for l}
       \sum_{j=1}^{N-1}b_j^+l_{j,\nu}^+(t)+b_N^+= \sum_{j=1}^{M}b_j^-l_{j,\nu}^-(t), \quad t\ge0.
   \end{equation}
   For a fixed $\mu\in\{1,\dots,M-1\}$, following similar steps, but this time multiplying by $\rho_\mu^-(t)-k_N^+$, we get
   \begin{equation}
       \label{sums for a}
       \sum_{j=1}^{N-1}b_j^+a_{j,\mu}^+(t)+b_N^+= \sum_{j=1}^{M}b_j^-a_{j,\mu}^-(t), \quad t\ge0.
   \end{equation}
Before moving further, we write down the orderings of the quantities in Definition \ref{ratios +-}. Taking into account the ordering of $k_{j}^{\pm}(t)$ and $\rho_j^{\pm}(t)$, we have that for any $\nu\in\{1,\dots,N-2\}$,
   
  \begin{align}\label{eq:orderings of l_n}
  \notag    &0<l_{M,\nu}^-(t)<\dots<l_{1,\nu}^-(t)<1<l_{N-1,\nu}^+(t)<\dots<l_{\nu+1,\nu}^+(t),\\
    &l_{\nu,\nu}^+(t)<\dots<l_{2,\nu}^+(t)<l_{1,\nu}^+(t)<0,
  \end{align}
for all $t\ge0$. In the special case when $\nu=N-1$, we have
\begin{equation}
\label{orderings of lN}
      0<l_{M,N-1}^-(t)<\dots<l_{1,N-1}^-(t)<1 \quad\text{and}\quad l_{N-1,N-1}^+(t)<\dots<l_{1,N-1}^+(t)<0,
\end{equation}
for all $t\ge0$. On the other side, for any $\mu\in\{1,\dots,M-1\}$, we get
   \begin{align}\label{orderings of a_m}
  \notag    0<a_{1,\mu}^+(t)<\dots<&a_{N-1,\mu}^+(t)<1<a_{1,\mu}^-(t)<\dots<a_{\mu,\mu}^-(t),\\
    &a_{\mu+1,\mu}^-(t)<\dots<a_{M,\mu}^-(t)<0,
  \end{align}
  for all $t\ge0$.

The functions $l_{j,\nu}^{\pm}$ and $\alpha_{j,\mu}^{\pm}$ play the most important role in our work and hence, we need to understand their behaviour.
A first lemma concerns the continuity and boundedness of these functions. 
\begin{lemma}\label{finiteness}
    Each function introduced in Definition \ref{ratios +-} is continuous, and all of its limit points, as $t\rightarrow+\infty$, are finite. 
\end{lemma}
\begin{proof}
    Returning to \eqref{polynomial}, the numbers $\rho_j^{\pm}$ are roots of a polynomial whose coefficients depend continuously on the functions $k_j^{\pm}$. Theorem \ref{continuity of the roots} asserts that the roots of a polynomial depend continuously on its coefficients. As a result, the roots $\rho_j^{\pm}$ are continuous functions of $t$ and, by extension, the same applies to the functions $l_j^{\pm}$, $a_j^{\pm}$.

    Fix, now, some $\nu\in\{1,\dots,N-2\}$. By $(\ref{eq:orderings of l_n})$, we have that for any index $j=1,\dots,M$, every limit point of $l_{j,\nu}^-$ is finite. Assume, on the other hand, that for some index $i\ge\nu+1$, a limit number of $l_{i,\nu}^+(t)$ is $+\infty$. Thus, there exists a sequence $(t_n)_{n\ge1}$, such that $l_{i,\nu}^+(t_n)\rightarrow+\infty$. Looking at $(\ref{sums for l})$, the right-hand side is evidently bounded. Thus, it is necessary that for some other index $i'\le\nu$, the sequence $(l_{i',\nu}^+(t_n))_{n\ge1}$ has a subsequence that tends to $-\infty$. Write $(t_{m_n})_{n\ge1}$ for the corresponding points. After some elementary computations, observe that
    \begin{equation*}
        \frac{l_{i',\nu}^+(t_{m_n})}{1-l_{i',\nu}^+(t_{m_n})}(k_{i'}^+(t_{m_n})-k_N^+)=k_N^+-\rho_{\nu}^+(t_{m_n})=\frac{l_{i,\nu}^+(t_{m_n})}{1-l_{i,\nu}^+(t_{m_n})}(k_{i}^+(t_{m_n})-k_N^+).  
    \end{equation*}
    But utilising $(\ref{k(t)})$, we see that 
    \begin{equation*}
          \frac{l_{i,\nu}^+(t_{m_n})}{1-l_{i,\nu}^+(t_{m_n})}\frac{1-l_{i',\nu}^+(t_{m_n})}{l_{i',\nu}^+(t_{m_n})}=\frac{k_{i'}^+(0)-k_N^+}{k_{i}^+(0)-k_N^+}e^{(\theta_i^+-\theta_{i'}^+)t_{m_n}}.  
    \end{equation*}
    By our assumptions, the left-hand side of the last equation converges to $1$. Nonetheless, since $i>i'$, the ordering in \eqref{eq:thetas} implies that $\theta_i^+\ge\theta_{i'}^+$, which in turn shows that the right-hand side either diverges to $+\infty$ or is constantly equal to $(k_{i'}^+(0)-k_N^+)/(k_i^+(0)-k_N^+)\ne1$. Contradiction! So, for any index $i\ge\nu+1$, all the limit numbers of $l_{i,\nu}^+(t)$, as $t\to+\infty$, are finite. The reverse procedure provides the desired outcome for the indices $i\le\nu$. The case $\nu=N-1$ is fairly simpler, since this time, all the functions $l_{j,N-1}^+$ are negative and \eqref{sums} leads directly to the finiteness of all the possible limit numbers.
    
    Fixing $\mu\in\{1,\dots,M\}$, the functions $a_{j,\mu}^+$ are clearly bounded due to \eqref{orderings of a_m}. Finally, the process for the functions $a_{j,\mu}^-$ follows similar steps as before and uses \eqref{sums for a}. We omit the details for the sake of not being redundant.
\end{proof}

At this stage, we understand that relations \eqref{k(t)}-\eqref{orderings of a_m} are independent of how the parameters $b_j^{\pm}$ and $\theta_j^{\pm}$ relate with each other. But from now on, we do take into account the way we intertwine them. For example, it is important to prescribe the $b_j^{\pm}$'s because the heights of the half-lines in $\partial\Omega_0$ depend only on the choice of the $b_j^{\pm}$'s, as we showed in Section \ref{sec:construction}. Looking at Figure \ref{fig:mu's} it is possible that several left half-lines are contained in the strip defined by the same two consecutive right half-lines. This means that even though the parameters $b_j^{\pm}$ are chosen arbitrarily, there exist some relations between them with respect to the preceding heights. In particular, we can extract indices $(\mu_{\nu})_{\nu=1}^{N-1}$ with
\begin{equation}\label{ordering of m}
    M\ge\mu_1\ge\dots\ge\mu_{N-1}\ge1
\end{equation}
so that for each $\nu\in\{1,\dots,N-1\}$, we have $\mathrm{Im}(h_t(\rho_{\mu_{\nu}-1}^-))<\mathrm{Im}(h_t(\rho_\nu^+))<\mathrm{Im}(h_t(\rho_{\mu_{\nu}}^-))$ and this implies (with the aid of \eqref{imaginary part of h in the positives}-\eqref{imaginary part of h in between}) that 
\begin{equation}
    \label{height with b}
    \sum_{j=1}^{\mu_{\nu}-1}b_j^-<\sum_{j=\nu+1}^Nb_j^+<\sum_{j=1}^{\mu_{\nu}}b_j^-.
\end{equation}

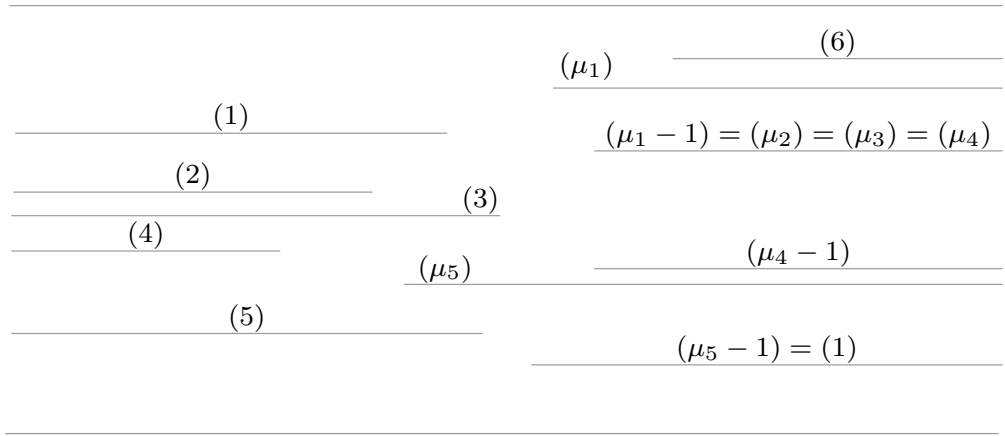
\begin{figure}[ht]
\centering
\resizebox{0.8\linewidth}{!}{%
\begin{tikzpicture}[
  x=0.02cm,
  y=-0.02cm, 
  every node/.style={font=\scriptsize, inner sep=0.5pt}
]

\definecolor{linegray}{RGB}{155,155,155}
\definecolor{framegray}{RGB}{110,110,110}

%
%
%
%

\def\Lift{8} 


\draw[linegray, line width=0.35pt] (18,19) -- (524,19);
\draw[linegray, line width=0.35pt] (16,237) -- (522,237);

\draw[linegray, line width=0.35pt] (21,84) -- (241,84);
\node at (131,{84-\Lift}) {(1)};

\draw[linegray, line width=0.35pt] (20,114) -- (203,114);
\node at ({(20+203)/2},{114-\Lift}) {(2)};

\draw[linegray, line width=0.35pt] (19,126) -- (268,126);
\node at (258,118) {(3)}; 

\draw[linegray, line width=0.35pt] (19,144) -- (156,144);
\node at ({(19+156)/2},{144-\Lift}) {(4)};

\draw[linegray, line width=0.35pt] (19,186) -- (259,186);
\node at ({(19+259)/2},{186-\Lift}) {(5)};

\draw[linegray, line width=0.35pt] (356,46) -- (524,46);
\node at ({(356+524)/2},{46-\Lift}) {(6)};

\draw[linegray, line width=0.35pt] (295,61) -- (524,61);
\node at (312,50) {$ (\mu_1) $};

\draw[linegray, line width=0.35pt] (316,93) -- (524,93);
\node at ({(316+524)/2},{93-\Lift}) {$ (\mu_1-1)=(\mu_2)=(\mu_3)=(\mu_4) $};

\draw[linegray, line width=0.35pt] (316,153) -- (524,153);
\node at ({(316+524)/2},{153-\Lift}) {$ (\mu_4-1) $};

\draw[linegray, line width=0.35pt] (219,161) -- (524,161);
\node at (240,153) {$ (\mu_5) $};

\draw[linegray, line width=0.35pt] (284,202) -- (524,202);
\node at ({(284+524)/2},{202-\Lift}) {$ (\mu_5-1)=(1) $};

\end{tikzpicture}%
}
\caption{The indices $\mu_\nu$ allow to classify the right half-lines so as to distinguish those left half-lines that lie on the same strip. For example, the left half-lines $(2),(3),(4)$ lie on one common strip and so $\mu_2=\mu_3=\mu_4$.}
\label{fig:mu's}
\end{figure}

During the proof of Lemma \ref{finiteness}, we deduced a very useful formula that relates the ratios in Definition \ref{ratios +-}. Indeed, for every $\nu\in\{1,\dots,N-1\}$, we have that for any $j=1,\dots,N-1$ and $i=1,\dots,M$,
\begin{equation}\label{connections of l's}
    \frac{l_{j,\nu}^+(t)}{1-l_{j,\nu}^+(t)}(k_j^+(t)-k_N^+)=k_N^+-\rho_{\nu}^+(t)=\frac{l_{i,\nu}^-(t)}{1-l_{i,\nu}^-(t)}(k_i^-(t)-k_N^+),
\end{equation}
for all $t\ge0$. Identically, for every $\mu\in\{1,\dots,M\}$, we have that for any $j=1,\dots,N-1$ and $i=1,\dots,M$,
\begin{equation}\label{connections of a's}
    \frac{a_{j,\mu}^+(t)}{1-a_{j,\mu}^+(t)}(k_j^+(t)-k_N^+)=k_N^+-\rho_{\mu}^-(t)=\frac{a_{i,\mu}^-(t)}{1-a_{i,\mu}^-(t)}(k_i^-(t)-k_N^+),
\end{equation}
for all $t\ge0$.

\subsection{Choice of exponents}As it turns out, our assumption in \eqref{eq:thetas} implies that for a fixed $\nu\in\{1,\dots,N-1\}$, only for the indices $j$ that correspond to equal $\theta^{\pm}_j$'s, can the functions $l_{j,\nu}^{\pm}$ have a limiting number different than $0$ or $1$. The question is which. This depends on the exponents $\theta_j^{\pm}$ and the height of the half-lines in the boundary of $\Omega_0$. For the rest of this work, we will make a further assumption correlating the exponents $\theta_j^{\pm}$ with the heights of the half-lines, that will provide the desired answer. Naturally, a different extra assumption could yield a different answer. The correct choice for the exponents is crucial for our model to work.
\par 

Before explaining the choice of the exponents, we fix a notation for the relevant ends of the image domain. Let
\[
\Omega_0 \vcentcolon= h_0(\mathbb H)
\]
be the domain corresponding to the initial configuration $\mathbf{k}(0)$. Working as in Example \ref{ex:prime ends}, for each $\nu\in\{1,\dots,N\}$, the boundary point $k_\nu^+$ determines a distinguished {\it left} prime end of $\Omega_0$, and for each $\mu\in\{1,\dots,M\}$, the boundary point $k_\mu^-$ determines a distinguished {\it right} prime end of $\Omega_0$. We denote these prime ends by
\[
\infty_\nu^+ \qquad\text{and}\qquad \infty_\mu^-,
\]
respectively. Thus, $\infty_\nu^+$ and $\infty_\mu^-$ are not points of $\mathbb C$, but symbols for the corresponding prime ends of $\Omega_0$. Geometrically, $\infty_\nu^+$ is the horizontal end at $-\infty$ of the strip adjacent to $k_{\nu}^+$, while $\infty_\mu^-$ is the horizontal end at $+\infty$ of the strip adjacent to $k_\mu^-$. 

We now describe the horizontal levels which will be used to organize these
prime ends. For the left prime ends we define
\[
S_1^+
=
\{w\in\mathbb C:\pi(\sum_{j=1}^Nb_j^+-\sum_{j=1}^Mb_j^-)>\operatorname{Im}w>\operatorname{Im}h(\rho_1^+)\},
\]
\[
S_\nu^+
=
\{w\in\mathbb C:
\operatorname{Im}h(\rho_\nu^+)<\operatorname{Im}w
<
\operatorname{Im}h(\rho_{\nu-1}^+)\},
\qquad 2\leq \nu\leq N-1,
\]
and
\[
S_N^+
=
\{w\in\mathbb C:-\pi \sum_{j=1}^Mb_j^-<\operatorname{Im}w<\operatorname{Im}h(\rho_{N-1}^+)\}.
\]

Similarly, for the right prime ends we define
\[
S_1^-
=
\{w\in\mathbb C:-\pi\sum_{j=1}^Mb_j^-<\operatorname{Im}w<\operatorname{Im}h(\rho_1^-)\},
\]
and
\[
S_\mu^-
=
\{w\in\mathbb C:
\operatorname{Im}h(\rho_{\mu-1}^-)
<
\operatorname{Im}w
<
\operatorname{Im}h(\rho_\mu^-)\},
\qquad 2\leq \mu\leq M.
\]

By the definition of the heights in \eqref{height with b}, for each $\nu\in\{1,\dots,N-1\}$ there exists an (not necessarily unique) index $\mu_\nu\in\{1,\dots,M\}$ such that
\[
h(\rho_\nu^+)\in S_{\mu_\nu}^-.
\]

The exponents are then chosen so that their ordering reflects the vertical ordering of the prime ends. More precisely, if the strip $S_\nu^+$ is contained in the strip $S_\mu^-$, then the corresponding horizontal ends occur at the same "height", and we set
\[
\theta_\nu^+=\theta_\mu^-.
\]
If, on the other hand, $h(\rho_\nu^+)\in S_{\mu_\nu}^-$ but $S_\nu^+\not\subset S_{\mu_\nu}^-$, then the corresponding ends are vertically separated, and we impose
\[
\theta_\nu^+<\theta_{\mu_\nu}^-.
\]
Similarly, whenever $S_\mu^-\subset S_\nu^+$, we set
\[
\theta_\mu^-=\theta_\nu^+.
\]
Thus the ordering of the exponents is read off from the ordering of the prime ends $\infty_j^\pm$ from bottom to top in a zig-zag manner as we see in Figure \ref{fig:thetas}: the exponents increase as we move upward, and the inequalities are strict unless a $+$-strip is contained in a $-$-strip.

 A careful inspection of Figure \ref{fig:thetas} and the above configuration implies that \eqref{eq:thetas} still holds.

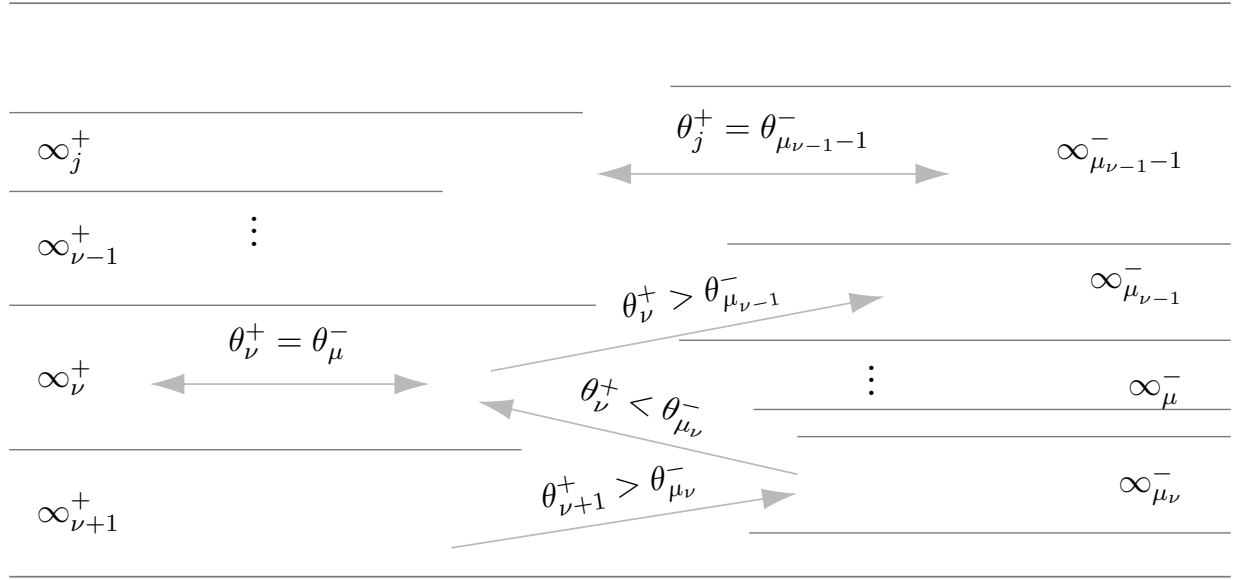
\begin{figure}[ht]
\centering
\resizebox{0.98\linewidth}{!}{%
\begin{tikzpicture}[
    x=1cm,y=1cm,
    every node/.style={font=\normalsize},
    >=Latex
]

\definecolor{linegray}{RGB}{120,120,120}
\definecolor{arrowgray}{RGB}{185,185,185}

\tikzset{
    halfline/.style={draw=linegray, line width=0.45pt},
    boundary/.style={draw=linegray, line width=0.55pt},
    comp/.style={
        draw=arrowgray,
        line width=0.55pt,
        -{Latex[length=4mm,width=2.4mm]}
    },
    compboth/.style={
        draw=arrowgray,
        line width=0.55pt,
        {Latex[length=4mm,width=2.4mm]}-{Latex[length=4mm,width=2.4mm]}
    },
    lab/.style={
        inner sep=1.5pt,
        outer sep=0pt
    }
}

\def\yTop{7.10}
\def\yBot{0.55}

\def\yLj{5.85}
\def\yLvmone{4.95}
\def\yLv{3.65}
\def\yLvpone{2.00}

\def\yRtop{6.15}
\def\yRmunumone{4.35}
\def\yRmu{3.25}
\def\yRmunu{2.15}
\def\yRbot{1.05}

\pgfmathsetmacro{\yLabLj}{(\yTop+\yLj)/2}
\pgfmathsetmacro{\yLabLvmone}{(\yLj+\yLvmone)/2}
\pgfmathsetmacro{\yLabLv}{(\yLv+\yLvpone)/2}
\pgfmathsetmacro{\yLabLvpone}{(\yLvpone+\yBot)/2}

\pgfmathsetmacro{\yLabRtop}{(\yRtop+\yRmunumone)/2}
\pgfmathsetmacro{\yLabRmu}{(\yRmu+\yRmunu)/2}
\pgfmathsetmacro{\yLabRmunu}{(\yRmunu+\yRbot)/2}

\pgfmathsetmacro{\yDotsRight}{(\yRmunu+\yRbot)/2}

\draw[boundary] (0.00,\yTop) -- (13.95,\yTop);
\draw[boundary] (0.00,\yBot) -- (13.95,\yBot);

\draw[halfline] (0.00,\yLj)      -- (6.55,\yLj);
\draw[halfline] (0.00,\yLvmone)  -- (4.95,\yLvmone);
\draw[halfline] (0.00,\yLv)      -- (6.7,\yLv);
\draw[halfline] (0.00,\yLvpone)  -- (5.85,\yLvpone);

\draw[halfline] (7.55,\yRtop)       -- (13.95,\yRtop);
\draw[halfline] (8.20,\yRmunumone)  -- (13.95,\yRmunumone);
\draw[halfline] (7.65,\yRmu)        -- (13.95,\yRmu);
\draw[halfline] (9,\yRmunu)      -- (13.95,\yRmunu);
\draw[halfline] (8.5,\yRmunu+0.31)      -- (13.95,\yRmunu+0.31);
\draw[halfline] (8.45,\yRbot)       -- (13.95,\yRbot);

\node[anchor=west] at (0.18,\yLabLvmone) {$\infty_j^{+}$};
\node[anchor=west] at (0.18,\yLabLv+1.5) {$\infty_{\nu-1}^{+}$};
\node[anchor=west] at (0.18,\yLabLv) {$\infty_{\nu}^{+}$};
\node[anchor=west] at (0.18,\yLabLvpone) {$\infty_{\nu+1}^{+}$};

\node[anchor=east] at (13.55,\yLabRmunu+3.8) {$\infty_{\mu_{\nu-1}-1}^{-}$};
\node[anchor=east] at (13.55,{\yLabRtop-1.35}) {$\infty_{\mu_{\nu-1}^{ }}^{-}$};
\node[anchor=east] at (13.55,\yLabRmu) {$\infty_{\mu}^{-}$};
\node[anchor=east] at (13.55,\yLabRmunu) {$\infty_{\mu_\nu}^{-}$};

 \node[anchor=east] at (10.05,\yDotsRight+1.3) {\Large$\vdots$};

 \node[anchor=east] at (3,\yDotsRight+3) {\Large$\vdots$};


\draw[compboth] (6.70,5.15) -- (10.75,5.15)
    node[midway, above=0.18cm, lab]
    {$\theta_j^{+}=\theta_{\mu_{\nu-1}-1}^{-}$};

\draw[compboth] (1.60,2.75) -- (4.80,2.75)
    node[midway, above=0.18cm, lab]
    {$\theta_\nu^{+}=\theta_{\mu}^{-}$};

\draw[comp] (5.5,2.90) -- (10,3.75)
    node[pos=0.55, above=0.05cm, sloped, lab]
    {$\theta_\nu^{+}>\theta_{\mu_{\nu-1}^{}}^{-}$};
\draw[comp] (9,1.72) -- (5.35,2.55)
    node[midway, above=0.08cm, sloped, lab]
    {$\theta_\nu^{+}<\theta_{\mu_\nu}^{-}$};

\draw[comp] (5.05,0.88) -- (9,1.5)
    node[midway, above=0.08cm, sloped, lab]
    {$\theta_{\nu+1}^{+}>\theta_{\mu_\nu}^{-}$};

\end{tikzpicture}%
}
\caption{For the ordering of the $\theta$'s we look at the ``ordering'' of the points at $\infty$, in a zig-zag manner. They increase as we go up and the inequalities are strict, unless a $+$-strip is contained in a $-$-strip.}
\label{fig:thetas}
\end{figure}

\subsection{Rates of convergence} We will, now, measure the rate of convergence of the tip points to the points at infinity they access, as $t\rightarrow+\infty$. As we mentioned several times previously, we need to to measure the 'speeds" in Definition \ref{rates of the tips} and then compare them. We first deal with the left tips and then we move on with the right ones.

\addtocontents{toc}{\protect\setcounter{tocdepth}{0}}
\subsection*{Convention} 
In what follows we will repeatedly take limit points for the ratios of Definition \ref{ratios +-}. Fix $\nu\in\{1,\dots,N-1\}$ and take, for instance, an index $j\in\{1,\dots,N-1\}$ and let $l_{j,\nu}^+$ be a limit number of $l_{j,\nu}^+(t)$, as $t\to+\infty$. Thus, there exists a sequence $(t_n)_{n\ge1}$, with $t_n\rightarrow+\infty$, so that $\lim_{n\to+\infty}l_{j,\nu}^+(t_n)=l_{j,\nu}^+$. Then, because $(l_{1.\nu}^-(t_n))_{n\ge1}$ is bounded, we may extract a convergent subsequence for some sequence $(t_{n_m})_{m\ge1}\subseteq(t_n)_{n\ge1}$. Taking a further subsequence, we extract a limit point for $(l_{2,\nu}^-(t_{n_m}))_{n\ge1}$ and we continue repeatedly until we reach the index $M$. In addition, by Lemma $\ref{finiteness}$ we  also extract limit points $l_{i,\nu}^-$ of $(l_{i,\nu}^-(t))_{t\ge0}$, so that all those limit points correspond to the same sequence $(t_n)_{n\ge1}$. 

So, when we start with a limit point of $(l_{j,\nu}^{\pm}(t))_{t\ge0}$, then we get a limit point of $(l_{i,\nu}^{\pm}(t))_{t\ge0}$ for any other index $i$, that corresponds to the same sequence $(t_n)_{n\ge1}$. To avoid describing this process every time we implement it, we just say that if $l_{j,\nu}^{\pm}$ is a limit point of $(l_{j,\nu}^{\pm}(t))_{t\ge0}$, then $l_{i,\nu}^{\pm}$ is a corresponding limit number of $(l_{i,\nu}^{\pm}(t))_{t\ge0}$. The same applies to the ratios $a_{j,\mu}^{\pm}(t)$.
\medskip

In the lemma below, we distinguish all those ratios whose limiting numbers are neither $0$ nor $1$. The proof requires the geometric configuration of the slits as seen in Figure \ref{fig:thetas}.
\begin{lemma}
    \label{non zero-one l's}
    For every $\nu\in\{1,\dots,N-1\}$
   \begin{equation*}
      0<\liminf_{t\rightarrow+\infty}{l^-_{\mu_{\nu},\nu}(t)}\le\limsup_{t\rightarrow+\infty}{l^-_{\mu_{\nu},\nu}(t)}<1.    
   \end{equation*}

\end{lemma}
 \begin{proof} 
We start with the case $\nu\in\{2,\dots,N-1\}$. By the definition of $\mu_\nu$, the tip point $h(\rho_\nu^+)$ lies in the strip $S_{\mu_\nu}^-$. Therefore, the choice of the exponents implies that
\begin{equation}\label{eq:theta-order-proof}
\theta_j^- \le \theta_\nu^+ \le \theta_{\mu_\nu}^- < \theta_i^-,
\qquad i<\mu_\nu<j.
\end{equation}
If, in addition, $S_\nu^+\subset S_{\mu_\nu}^-$, then by construction
\begin{equation}\label{eq:theta-order-proof-equality}
\theta_j^-<\theta_\nu^+=\theta_{\mu_\nu}^-<\theta_i^-,
\qquad i<\mu_\nu<j.
\end{equation}

We first consider the case $\nu\in\{2,\dots,N-2\}$ and we distinguish two further subcases with regard to whether the strip $S_\nu^+$ is contained in $S_{\mu_\nu}^-$ or not.

\textit{Case I.} Assume that $S_{\nu}^+\subset S_{\mu_{\nu}}^-$ and hence \eqref{eq:theta-order-proof-equality} prevails. Applying the fact that $\theta_{\nu}^+=\theta_{\mn}^-$ in \eqref{k(t)} and using \eqref{connections of l's}, we deduce that 
     \begin{equation}\label{eq:non zero lambda 0}
     \frac{l_{\mn,\nu}^-(t)}{1-l_{\mn,\nu}^-(t)}(k_{\mu_\nu}^-(0)-k_N^+)=\frac{l_{\nu,\nu}^+(t)}{1-l_{\nu,\nu}^+(t)}(k_{\nu}^+(0)-k_N^+).    
     \end{equation}
    Suppose that $l_{\mn,\nu}^-=0$. Then, the preceding equality signifies that $l_{\nu,\nu}^+=0$ as well. But in such a case, due to the ordering in \eqref{eq:orderings of l_n}, we have $l_{M,\nu}^-=l_{M-1,\nu}^-=\dots=l_{\mu_\nu,\nu}^-=0=l_{1,\nu}^+=l_{2,\nu}^+=\dots=l_{\nu,\nu}^+$. Therefore, taking limits, relation \eqref{sums for l} becomes
    \begin{equation}\label{eq:non zero lambda 1}
        \sum_{j=\nu+1}^{N-1}b_j^+l_{j,\nu}^++b_N^+=\sum_{j=1}^{\mu_\nu-1}b_j^-l_{j,\nu}^-.
    \end{equation}
    Then, keeping \eqref{eq:orderings of l_n} in mind and using \eqref{eq:non zero lambda 1}, we have
     \begin{equation*}
      \sum_{j=\nu+1}^Nb_j^+\le\sum_{j=\nu+1}^{N-1}b_j^+l_{j,\nu}^++b_N^+=\sum_{j=1}^{\mn-1}b_j^-l_{j,\nu}^-\le\sum_{j=1}^{\mn-1}b_j^-<\sum_{j=\nu+1}^Nb_j^+,    
     \end{equation*}
    where the last inequality is derived from \eqref{height with b}. Contradiction! Since we picked an arbitrary limit number, we deduce that $\liminf_{t\rightarrow+\infty}l_{\mn,\nu}^-(t)>0$. On the other hand, if we assume that $l_{\mu_\nu,\nu}^-=1$, we understand that $l_{\nu,\nu}^+=1$ too via \eqref{eq:non zero lambda 0}. But this contradicts the fact that $l_{\nu,\nu}^+\le0$. Thus,  $\limsup_{t\rightarrow+\infty}l_{\mn,\nu}^-(t)<1$.

     \textit{Case II.} Suppose, now, that $S_{\nu}^+\nsubseteq S_{\mn}^-$. This time, \eqref{eq:theta-order-proof} becomes $\theta_i^-\le\theta_{\nu}^+<\theta_{\mu_{\nu}}^-<\theta_j^-$, for all $j<\mu_\nu<i$. If we assume that $l_{\mn,\nu}^-=0$, \eqref{connections of l's} yields 
`   \begin{equation*}
     \lim_{n\to+\infty}\frac{l_{\nu,\nu}^+(t_n)}{1-l_{\nu,\nu}^+(t_n)}=\lim_{t\to+\infty}\left(\frac{l_{\mn,\nu}^-(t_n)}{1-l_{\mn,\nu}^-(t_n)}\frac{k_{\mu_\nu}^-(0)-k_N^+}{k_{\nu}^+(0)-k_N^+}e^{(\theta_{\nu}^+-\theta_{\mn}^-)t_n}\right)=0,
     \end{equation*}
     for some sequence $\{t_n\}$ increasing to $+\infty$. Hence $l_{\nu,\nu}^+=0$, for the corresponding limit number and we reach exactly the same contradiction as in the previous case. Therefore, $\liminf_{t\rightarrow+\infty}{l_{\mn,\nu}^-(t)>0}$. If, on the other end, assume that $l_{\mn,\nu}^-=1$, we obtain
     \begin{equation*}
      \lim_{n\rightarrow+\infty}\frac{l_{\nu+1,\nu}^+(t_n)}{1-l_{\nu+1,\nu}^+(t_n)}(k_{\nu+1}^+(0)-k_N^+)=\lim_{n\rightarrow+\infty}\left(\frac{l_{\mn,\nu}^-(t_n)}{1-l_{\mn,\nu}^-(t_n)}(k_{\mu_\nu}^-(0)-k_N^+)e^{(\theta_{\nu+1}^+-\theta_{\mn}^-)t_n}\right)=-\infty,   
     \end{equation*}
     for a suitable $\{t_n\}$, because $\theta_{\nu+1}^+\ge\theta_{\mn}^-$ (see Figure \ref{fig:thetas}). As a result, $l_{\nu+1,\nu}^+=1$ for the corresponding limit number. Through \eqref{eq:orderings of l_n}, we find $l_{\mu_\nu,\nu}^-=\dots=l_{2,\nu}^-=l_{1,\nu}^-=1=l_{\nu+1,\nu}^+=\dots=l_{N-1,\nu}^+$. Consequently, after taking limits, \eqref{sums for l} turns into
     \begin{equation}\label{eq:non zero lambda 2}
         \sum_{j=2}^\nu b_j^+l_{j,\nu}^++\sum_{j=\nu+1}^{N}b_j^+=\sum_{j=1}^{\mu_\nu}b_j^-+\sum_{j=\mu_\nu+1}^M b_j^-l_{j,\nu}^-.
     \end{equation}
     Thus, combining the ordering in \eqref{eq:orderings of l_n} with \eqref{eq:non zero lambda 2}, we see that
     \begin{equation*}
     \sum_{j=\nu+1}^Nb_j^+\overset{\eqref{eq:orderings of l_n}}{\ge}\sum_{j=2}^{\nu}b_j^+l_{j,\nu}^++\sum_{j=\nu+1}^Nb_j^+=\sum_{j=1}^{\mn}b_j^-+\sum_{j=\mn+1}^Mb_j^-l_{j,\nu}^-\overset{\eqref{eq:orderings of l_n}}{\ge}\sum_{j=1}^{\mn}b_j^-\overset{\eqref{height with b}}{>}\sum_{j=\nu+1}^Nb_j^+,    
     \end{equation*}
     which is absurd. This means that every limit number of $l_{\mn,\nu}^-$ has to be less than $1$, and so we get the desired result for $\nu\in\{2,\dots,N-2\}$.

    We move on the case $\nu=N-1$. Then, the ordering in \eqref{eq:orderings of l_n} simplifies to
\begin{equation}
    \label{orderings of lN new}
      0<l_{M,N-1}^-(t)<\dots<l_{1,N-1}^-(t)<1 \quad\text{and}\quad l_{N-1,N-1}^+(t)<\dots<l_{2,N-1}^+(t)<0.
\end{equation}
Once more, we distinguish two subcases.

\textit{Case I.} Suppose that $\theta_{N-1}^+=\theta_{\mu_{N-1}}^-$, or equivalently $S_{N}^+\subset S_{\mu_{N-1}^-}$. As usual,
\begin{equation}\label{eq:non zero lambda 3}
  \frac{l^+_{N-1,N-1}(t)}{1-l_{N-1,N-1}^+(t)}(k^+_{N-1}(0)-k_N^+)=\frac{l^-_{\mu_{N-1},N-1}(t)}{1-l_{\mu_{N-1},N-1}^-(t)}(k^-_{\mu_{N-1}}(0)-k_N^+).  
\end{equation}
Assume that $l_{\mu_{N-1},N-1}^-=1$. By the preceding relation, we also get $l_{N-1,N-1}^+=1$, contradicting \eqref{orderings of lN new}. Thus, $\limsup_{t\rightarrow+\infty}{l_{\mu_{N-1},N-1}^-}(t)<1$. On the other side, if $l_{\mu_{N-1},N-1}^-=0$, \eqref{orderings of lN new} dictates that $l_{\mu_{N-1},N-1}^-=\dots=l_{M-1,N-1}^-=l_{M,N-1}^-=0$. But from \eqref{eq:non zero lambda 3} and the fact that $l_{N-1,N-1}^+$ has no infinite limit points (see Lemma \ref{finiteness}, we get $l_{N-1,N-1}^+=0$. Together with \eqref{orderings of lN new}, we comprehend that $l_{2,N-1}^+=l_{3,N-1}^+=\dots=l_{N-1,N-1}^+=0$. Applying on \eqref{sums for l} gives rise to
\begin{equation}\label{eq:non zero lambda 4}
   b_N^+=\sum_{j=1}^{\mu_{N-1}-1}b_j^-l_{j,N-1}^-. 
\end{equation}
Nevertheless, \eqref{eq:non zero lambda 4} implies
\begin{equation*}
  \sum_{j=1}^{\mu_{N-1}-1}b_j^-\overset{\eqref{orderings of lN new}}{\ge}\sum_{j=1}^{\mu_{N-1}-1}b_j^- l_{j,\mu_{N-1}}^-=b_N^+\overset{\eqref{height with b}}{>}\sum_{j=1}^{\mu_{N-1}-1}b_j^-. 
\end{equation*}
Contradiction! This means that $\liminf_{t\rightarrow+\infty}{l_{\mu_{N-1},N-1}^-}(t)>0$.

Assume, finally, that $\theta_{N-1}^+<\theta_{\mu_{N-1}^-}$ or equivalently $S_{N-1}^+\nsubseteq S_{\mu_{N-1}}^-$. We write
\begin{equation*}
\frac{l^+_{N-1,N-1}(t)}{1-l_{N-1,N-1}^+(t)}(k^+_{N-1}(0)-k_N^+)=\frac{l^-_{\mu_{N-1},N-1}(t)}{1-l_{\mu_{N-1},N-1}^-(t)}(k^-_{\mu_{N-1}}(0)-k_N^+)e^{(\theta_{N-1}^+-\theta_{\mu_{N-1}}^-)t}.    
\end{equation*}
If $l_{\mu_{N-1},N-1}^-=0$, then $l_{N-1,N-1}^+=0$ by the relation above and we reach a contradiction exactly as before since \eqref{eq:non zero lambda 4} will still hold. So, $\liminf{l_{\mu_{N-1},N-1}^-(t)}>0$. Lastly, if we assume that $l_{\mu_{N-1},N-1}^-=1$, \eqref{orderings of lN new} dictates that  we will have that $l_{1,N-1}^-=l_{2,N-1}^-=\dots=l_{\mu_{N-1},N-1}^-=1$. Ergo, \eqref{sums for l} transforms to
\begin{equation}\label{eq:non zero lambda 5}
    \sum_{j=2}^{N-1}b_j^+l_{j,N-1}^++b_N^+=\sum_{j=1}^{\mu_{N-1}}b_j^-+\sum_{j=\mu_{N-1}+1}^Mb_j^-l_{j,N-1}^-.
\end{equation}
As a consequence, working on \eqref{eq:non zero lambda 5}, we find
\begin{equation*}
    b_N^+\overset{\eqref{orderings of lN new}}{\ge}\sum_{j=2}^{N-1}b_j^+l_{j,N-1}^++b_N^+=\sum_{j=1}^{\mu_{N-1}}b_j^-+\sum_{j=\mu_{N-1}+1}^{M}b_j^-l_{j,N-1}^-\overset{\eqref{orderings of lN new}}{\ge}\sum_{j=1}^{\mu_{N-1}}b_j^- \overset{\eqref{height with b}}{>}b_N^+,
\end{equation*}
which is once again absurd. Hence, $\limsup{l_{\mu_{N-1},N-1}^-(t)}<1$
and the desired outcome holds for $\nu=N-1$.

To conclude the proof, we treat the case $\nu=1$. Observe that in this instance, \eqref{eq:orderings of l_n} may be written as 
     \begin{equation}\label{eq:ordering for l_1}
      0<l_{M,1}^-(t)<\dots<l_{1,1}^-(t)<1<l_{N-1,1}^+(t)<\dots<l_{2,1}^+(t).   
     \end{equation}
Since $\theta_2^+\ge\theta_{\mu_1}^-$ by our configuration, using \eqref{connections of l's} and assuming that $l_{\mu_1,1}^-=1$, we obtain
\begin{equation*}
 \lim_{n\rightarrow+\infty}\frac{l_{2,1}^+(t_n)}{1-l_{2,1}^+(t_n)}(k_{2}^+(0)-k_N^+)=\lim_{n\rightarrow+\infty}\left(\frac{l_{\mu_1,1}^-(t_n)}{1-l_{\mu_1,1}^-(t_n)}(k_{\mu_1}^-(0)-k_N^+)e^{(\theta_2^+-\theta_{\mu_1}^-)t_n}\right)=+\infty. 
\end{equation*}
Therefore, it is necessary that $l_{2,1}^+=1$ and by extension $l_{3,1}^+=\dots=l_{N-1,1}^+=1$ due to \eqref{eq:ordering for l_1}. Furthermore, for the same reason, $l_{\mu_1,1}^-=1$ forces $l_{1,1}^-=l_{2,1}^-=\dots=l_{\mu_1-1,1}^-=1$ as well. Then,
\begin{equation*}
 \sum_{j=2}^Nb_j^+\overset{\eqref{sums for l}}{=}\sum_{j=1}^{\mu_1}b_j^-+\sum_{j=\mu_1+1}^Mb_j^-l_{j,1}^-\overset{\eqref{eq:ordering for l_1}}{\ge}\sum_{j=1}^{\mu_1}b_j^-\overset{\eqref{height with b}}{>}\sum_{j=2}^Nb_j^+.
\end{equation*}
Contradiction! Thus, $\limsup_{t\rightarrow+\infty}{l_{\mu_1,1}^-}(t)<1$. On the other hand, if $l_{\mu_1,1}^-=0$ (and therefore $l_{\mu_1+1,1}^-=\dots=l_{M,1}^-=0$), a similar procedure leads to
\begin{equation*}
 \sum_{j=2}^Nb_j^+\overset{\eqref{eq:ordering for l_1}}{\le}\sum_{j=2}^Nb_j^+l_{j,1}^+\overset{\eqref{sums for l}}{=}\sum_{j=1}^{\mu_1-1}b_j^-l_{j,1}^-\overset{\eqref{eq:ordering for l_1}}{\le}\sum_{j=1}^{\mu_1-1}b_j^-\overset{\eqref{height with b}}{<}\sum_{j=2}^Nb_j^+, 
\end{equation*}
which is out of the question! As a result, $\liminf_{t\rightarrow+\infty}{l_{\mu_1,1}^-}(t)>0$ and the proof is complete.

 \end{proof}

Note that the same result is true for the limiting numbers of $l_{j,\nu}^+$, for any index $j$ (if they exist), such that $S_{j}^+\subset S^-_{\mu_{\nu}}$. The following proposition is a direct consequence of the preceding lemma and sheds light on the matter of how fast the tip points $h(\rho_{\nu}^+(t),t)$ diverge. Before that, we specify a piece of notation that we are going to be using frequently. Given two functions $f,g\vcentcolon[0,+\infty)\to[0,+\infty)$, we will write $f(t)\asymp g(t)$ (as $t\to+\infty$) provided there exist $t_0\ge0$ and an absolute positive constant $c\ge1$ such that $\frac{1}{c}g(t)<f(t)<cg(t)$, for all $t>t_0$. In other words, $f(t)\asymp g(t)$ if and only if both $f(t)/g(t)$ and $g(t)/f(t)$ are $O(1)$, as $t\to+\infty$.

\begin{proposition}\label{prop:rates of left tips}
For every $\nu\in\{2,\dots,N-1\}$,
        \begin{equation*}
         r(\rho_{\nu}^+)=-\sum_{j=2}^{\nu}b_j^+\theta_j^++\left(\sum_{j=1}^{\mn}b_j^--\sum_{j=\nu+1}^{N}b_j^+\right)\theta_{\mn}^-+\sum_{j=\mn+1}^Mb_j^-\theta_j^-  
        \end{equation*}
        while,
        \begin{equation*}
         r(\rho_1^+)=\left(\sum_{j=1}^{\mu_1}b_j^--\sum_{j=2}^{N}b_j^+\right)\theta_{\mu_1}^-+\sum_{j=\mu_1+1}^Mb_j^-\theta_j^->0\,.    \end{equation*}
\end{proposition}

\begin{proof}
Fix \(\nu\in\{1,\dots,N-1\}\). By Lemmata~\ref{finiteness} and
\ref{non zero-one l's}, the quantity
\[
\frac{l_{\mu_\nu,\nu}^-(t)}{1-l_{\mu_\nu,\nu}^-(t)}
\]
remains bounded above and below by positive constants for \(t\ge 0\). Hence,
combining \eqref{k(t)} with the left-hand side of
\eqref{connections of l's}, we obtain
\begin{equation}\label{rate of rho-N+}
|\rho_\nu^+(t)-k_N^+|
=
\frac{l_{\mu_\nu,\nu}^-(t)}{1-l_{\mu_\nu,\nu}^-(t)}
\,|k_{\mu_\nu}^-(t)-k_N^+|
\asymp e^{-\theta_{\mu_\nu}^- t}.
\end{equation}
With \eqref{rate of rho-N+} at our disposal, we may estimate the asymptotic behavior, as $t\to+\infty$, of all the differences $\rho_\nu^+(t)-k_j^\pm(t)$. As a matter of fact, using \eqref{eq:thetas}, for \(j=1,\dots,\mu_\nu\), we have \(\theta_j^-\ge \theta_{\mu_\nu}^-\), and therefore
\begin{equation}\label{rate of rho-j-1}
|\rho_\nu^+(t)-k_j^-(t)|
=
\bigl|(\rho_\nu^+(t)-k_N^+)+(k_N^+-k_j^-(t))\bigr|
\asymp e^{-\theta_{\mu_\nu}^-t}+e^{-\theta_j^-t} 
\asymp e^{-\theta_{\mu_\nu}^- t}.
\end{equation}
Similarly, for \(j=\mu_\nu+1,\dots,M\), we obtain
\begin{equation}\label{rate of rho-j-2}
|\rho_\nu^+(t)-k_j^-(t)|
=
\bigl|(\rho_\nu^+(t)-k_N^+)+(k_N^+-k_j^-(t))\bigr|
\asymp e^{-\theta_j^- t},
\end{equation}
because for these indices $\theta_j^-<\theta_{\mu_\nu}$. We now compare the distance of \(\rho_\nu^+(t)\) with the points \(k_j^+(t)\). To start with, we treat the case \(\nu\in\{2,\dots,N-2\}\). By our choice of the exponents,
\[
\theta_{\nu+1}^+\ge \theta_{\mu_\nu}^- \ge \theta_\nu^+,
\qquad \nu\in\{2,\dots,N-2\}\,.\]
Thus, if \(j=2,\dots,\nu\), then
\(\theta_j^+\le \theta_\nu^+\le \theta_{\mu_\nu}^-\), and so
\begin{equation}\label{rate of rho-j+1}
|\rho_\nu^+(t)-k_j^+(t)|
=
\bigl|(\rho_\nu^+(t)-k_N^+)+(k_N^+-k_j^+(t))\bigr|
\asymp e^{-\theta_j^+ t}.
\end{equation}
On the other hand, if \(j=\nu+1,\dots,N-1\), then
\(\theta_j^+\ge \theta_{\nu+1}^+\ge \theta_{\mu_\nu}^-\), and hence
\begin{equation}\label{rate of rho-j+2}
|\rho_\nu^+(t)-k_j^+(t)|
=
\bigl|(\rho_\nu^+(t)-k_N^+)+(k_N^+-k_j^+(t))\bigr|
\asymp e^{-\theta_{\mu_\nu}^- t}.
\end{equation}
All that remains is to consider the extremal cases \(\nu=1\) and \(\nu=N-1\). In these cases, our assumptions give
\[
\theta_2^+\ge \theta_{\mu_1}^-,
\qquad
\theta_{\mu_{N-1}}^-\ge \theta_{N-1}^+.
\]
Consequently, if \(\nu=1\), for every \(j=2,\dots,N-1\) we have \(\theta_j^+\ge \theta_2^+\ge \theta_{\mu_1}^-\), and hence
\begin{equation}\label{rate of rho-1-j+}
|\rho_1^+(t)-k_j^+(t)|
=
\bigl|(\rho_1^+(t)-k_N^+)+(k_N^+-k_j^+(t))\bigr|
\asymp e^{-\theta_{\mu_1}^- t}.
\end{equation}
Similarly, in the case \(\nu=N-1\), \(\theta_{\mu_{N-1}}^- \ge \theta_j^+\) for every \(j=2,\dots,N-1\), and therefore
\begin{equation}\label{rate of rho-N-1-j+}
|\rho_{N-1}^+(t)-k_j^+(t)|
=
\bigl|(\rho_{N-1}^+(t)-k_N^+)+(k_N^+-k_j^+(t))\bigr|
\asymp e^{-\theta_j^+ t}.
\end{equation}
Collecting all the estimates above and returning to \eqref{h_t(N,M)}, for \(\nu=2,\dots,N-1\) we find
\begin{align*}
\mathrm{Re}\, h(\rho_\nu^+(t),t)
&=
\sum_{j=1}^N b_j^+\log|\rho_\nu^+(t)-k_j^+(t)|
-
\sum_{j=1}^M b_j^-\log|\rho_\nu^+(t)-k_j^-(t)| \\
&=
b_1^+\log|\rho_\nu^+(t)-k_1^+|
+\sum_{j=2}^{\nu} b_j^+(-\theta_j^+ t)
+\sum_{j=\nu+1}^{N} b_j^+(-\theta_{\mu_\nu}^- t)\\
&\quad
-\sum_{j=1}^{\mu_\nu} b_j^-(-\theta_{\mu_\nu}^- t)
-\sum_{j=\mu_\nu+1}^{M} b_j^-(-\theta_j^- t) +O(1),
\end{align*}
as \(t\to+\infty\). Since \(\rho_\nu^+(t)\to k_N^+\) and $k_1^+\ne k_N^+$, the first term is bounded. Dividing by \(t>0\)
and letting \(t\to+\infty\), we obtain the desired conclusion. 

For \(\nu=1\), using \eqref{rate of rho-N+}, \eqref{rate of rho-j-1},
\eqref{rate of rho-j-2}, and \eqref{rate of rho-1-j+}, we obtain
\begin{align*}
\mathrm{Re}\, h(\rho_1^+(t),t)
&=
\sum_{j=1}^N b_j^+\log|\rho_1^+(t)-k_j^+(t)|
-
\sum_{j=1}^M b_j^-\log|\rho_1^+(t)-k_j^-(t)| \\
&=
b_1^+\log|\rho_1^+(t)-k_1^+|
+\sum_{j=2}^{N} b_j^+(-\theta_{\mu_1}^- t) \\
&\quad
-\sum_{j=1}^{\mu_1} b_j^-(-\theta_{\mu_1}^- t)
-\sum_{j=\mu_1+1}^{M} b_j^-(-\theta_j^- t)+ O(1),
\end{align*}
as \(t\to+\infty\). As before, the first term is bounded. Dividing by \(t>0\) and letting \(t\to+\infty\), we conclude that
\[
r(\rho_1^+)
=
\left(\sum_{j=1}^{\mu_1} b_j^-
-
\sum_{j=2}^{N} b_j^+\right)\theta_{\mu_1}^-
+
\sum_{j=\mu_1+1}^{M} b_j^-\,\theta_j^- .
\]
This proves the desired formula in the case \(\nu=1\). By \eqref{height with b}, we also deduce that the parenthesis in the latter equality is positive and thus $r(\rho_1^+)>0$.
\end{proof}

We now turn our attention to the right tip points as we attempt to figure out analogous properties. Arguing as above, we will see that the relative heights of the right half-lines affect the rates of the right-handed tip points.
For convenience matters, before stating the results, we describe the
geometric terminology that will be used. The indices
\(\mu_\nu\), \(\nu=1,\dots,N-1\), divide the family of right slits into
three types, according to their position relative to the left tip points.
We call a right slit \textit{upper} if its index satisfies
\[
    \mu\in\{\mu_1,\dots,M-1\};
\]
these are the right slits whose tips lie above the level determined by
\(h(\rho_1^+)\). We call a right slit \textit{middle} if, for some
\(\nu\in\{2,\dots,N-2\}\),
\[
    \mu\in\{\mu_\nu,\dots,\mu_{\nu-1}-1\};
\]
geometrically, these are the right slits lying between the levels
determined by two consecutive left tip points \(h(\rho_\nu^+)\) and
\(h(\rho_{\nu-1}^+)\). Finally, we call a right slit \textit{lower} if
\[
    \mu\in\{1,\dots,\mu_{N-1}-1\};
\]
these are the right slits lying below the level determined by
\(h(\rho_{N-1}^+)\). See Figure \ref{fig: relative heights}.

\begin{figure}[ht]
\centering
\resizebox{0.8\linewidth}{!}{%
\begin{tikzpicture}[
  x=0.02cm,
  y=-0.02cm,
  every node/.style={font=\scriptsize, inner sep=0.6pt}
]

\definecolor{framegray}{RGB}{95,95,95}
\definecolor{rayblack}{RGB}{35,35,35}      
\definecolor{timeviolet}{RGB}{210,180,210} 
\definecolor{timeblue}{RGB}{180,190,235}   

%
%
%

\draw[rayblack, line width=0.5pt] (22,14) -- (520,14);
\draw[rayblack, line width=0.5pt] (22,250) -- (520,250);


\draw[timeviolet, line width=0.7pt] (340,72) -- (60,72);
\draw[timeviolet, line width=0.7pt] (300,40)-- (60,40);
\draw[rayblack,   line width=0.6pt] (340,72) -- (520,72);
\draw[rayblack,   line width=0.6pt] (300,40) -- (520,40);

\draw[rayblack, line width=0.6pt] (22,94) -- (245,94);
\draw[timeblue, line width=0.7pt] (245,94) -- (480,94);

\draw[timeviolet, line width=0.7pt] (65,118) -- (392,118);
\draw[rayblack,   line width=0.6pt] (392,118) -- (520,118);

\draw[rayblack, line width=0.6pt] (150,140) -- (520,140);
\draw[timeviolet, line width=0.7pt] (55,140)--(150,140);



\draw[timeviolet, line width=0.7pt] (46,178) -- (250,178);
\draw[rayblack,   line width=0.6pt] (250,178) -- (520,178);

\draw[rayblack, line width=0.6pt] (22,198) -- (160,198);
\draw[timeblue, line width=0.7pt] (160,198) -- (455,198);

\draw[rayblack, line width=0.6pt] (345,216) -- (520,216);
\draw[timeviolet, line width=0.7pt] (345,216)--(50,216);
\draw[rayblack, line width=0.6pt] (345,240) -- (520,240);
\draw[timeviolet, line width=0.7pt] (345,240)--(60,240);


\node at (460,50) {$\vdots$};
\node at (468,156) {$\vdots$};
\node at (468,222) {$\vdots$};

\node at (468,125) {$\vdots$};
\node at (40,140) {$\vdots$};

\node[anchor=west] at (12,220) {$\infty_{N}^{+}$};
\node[anchor=west] at (12,60) {$\infty_{1}^{+}$};

\node[anchor=west] at (480,93) {$\infty_{\mu_{1}}^{-}$};
\node[anchor=west] at (460,197) {$\infty_{\mu_{N-1}}^{-}$};

\node at (275,30)  {$h(\rho_{M}^{-})$};
\node at (340,58)  {$h(\rho_{\mu_{1}}^{-})$};
\node at (255,80)  {$h(\rho_{1}^{+})$};
\node at (400,106) {$h(\rho_{\mu_{1}-1}^{-})$};
\node at (275,168) {$h(\rho_{\mu_{N-1}}^{-})$};
\node at (300,206.1) {$h(\rho_{\mu_{N-1}-1}^{-})$};
\node at (325,230) {$h(\rho_{1}^{-})$};
\node at (165,188) {$h(\rho_{N-1}^{+})$};

\end{tikzpicture}%
}
\caption{The upper slits access $\infty_1^+$ and the lower slits access $\infty_N^+$. The rest are denoted as middle slits.}
\label{fig: relative heights}
\end{figure}
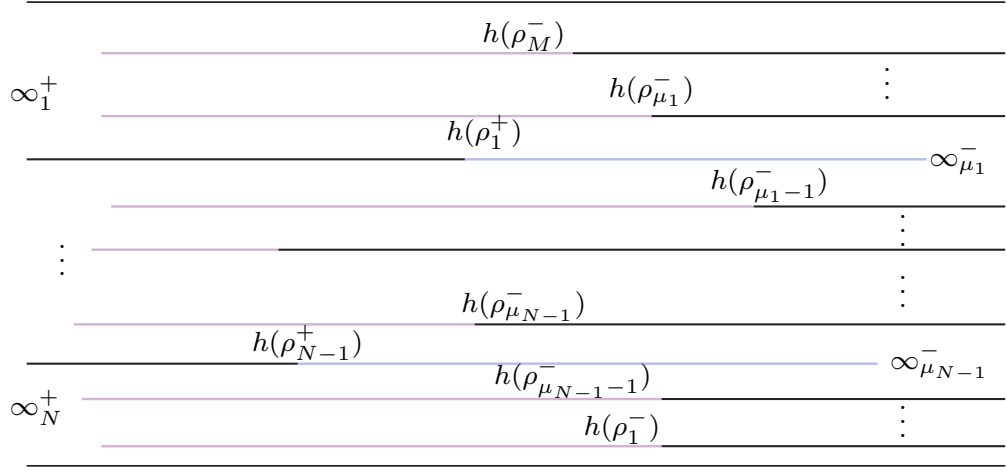

\newpage
 \begin{lemma}
   \label{non zero-one a's}
    For $\mu\in\{1,\dots,M-1\}$, the following hold:
    \begin{enumerate}
    \item[\textup{(1)}] (Upper slits) If $\mu\in\{\mu_1,\dots,M-1\}$, then $\limsup_{t\to+\infty}{a_{\mu+1,\mu}^-(t)}<0$.
    \item[\textup{(2)}](Middle slits) If there exists $\nu\in\{2,\dots,N-2\}$ so that $\mu\in\{\mu_{\nu},\dots,\mu_{\nu-1}-1\}$, then $0<\liminf_{t\to+\infty}{a_{\nu,\mu}^+(t)}\le\limsup_{t\to+\infty}{a_{\nu,\mu}^+(t)}<1$.
    \item[\textup{(3)}](Lower slits) If $\mu\in\{1,\dots,\mu_{N-1}-1\}$, then $\liminf_{t\to+\infty}{a_{\mu,\mu}^-(t)}>1$.
    \end{enumerate}
 \end{lemma}
 \begin{proof}
(1) We commence with the upper right half-lines, namely those with indices $\mu\in\{\mu_1,\dots,M-1\}$ (whenever such indices exist). Recall that we disregard the case $\mu=M$. Note that by a consecutive use of \eqref{sums for a} and \eqref{orderings of a_m},
\begin{equation}\label{eq:upper slits 1}
  \sum_{j=2}^Nb_j^+\ge\sum_{j=2}^{N-1}b_j^+\alpha_{j,\mu}^++b_N^+=\sum_{j=1}^{\mu}b_j^-\alpha_{j,\mu}^-+\sum_{j=\mu+1}^Mb_j^-\alpha_{j,\mu}^-\ge\sum_{j=1}^{\mu}b_j^-+\sum_{j=\mu+1}^Mb_j^-\alpha_{j,\mu}^-.  
\end{equation}
Assuming that $\alpha_{\mu+1,\mu}^-=0$, \eqref{orderings of a_m} implies $\alpha_{\mu+2,\mu}^-=\dots=\alpha_{M-1,\mu}^-=\alpha_{M,\mu}^-=0$. Going back to \eqref{eq:upper slits 1} and taking into account that $\mu\ge\mu_1$, we see that
\begin{equation*}
    \sum_{j=2}^Nb_j^+\ge\sum_{j=1}^{\mu}b_j^-\ge \sum_{j=1}^{\mu_1}b_j^-
\end{equation*}
which cannot hold due to \eqref{height with b}. Thus, it is necessary that $\limsup{\alpha_{\mu+1,\mu}^-(t)}<0$. 

(2) Let us continue with the middle right half-lines. This means that we consider $\mu\in\{1,\dots,M-1\}$ such that $h(\rho_{\mu}^-)\in S_{\nu}^+$ for some $\nu\in\{2,\dots,N-2\}$. Thus $\mu\in\{\mu_{\nu},\dots,\mu_{\nu-1}-1\}$. 

\textit{Case I.} Suppose, first, that $\mu>\mu_{\nu}$. In this case, our configuration forces $S_\mu^-\subset S_\nu^+$ and $\theta_{\mu_\nu}^->\theta_\mu^-=\theta_\nu^+>\theta_{\mu_\nu-1}^-$. By \eqref{connections of a's}, we see that for each limit number $\alpha_{\nu,\mu}^+$, with corresponding number $a_{\mu,\mu}^-$,
     \begin{equation}\label{eq:middle slits 1}
     \frac{a_{\nu,\mu}^+}{1-a_{\nu,\mu}^+}(k_\nu^+(0)-k_N^+)=\frac{a_{\mu,\mu}^-}{1-a_{\mu,\mu}^-}(k_{\mu}^-(0)-k_N^+).  
     \end{equation}
Thus, assuming $\alpha_{\nu,\mu}^+=0$, we also get $\alpha_{\mu,\mu}^-=0$. But $\alpha_{\mu,\mu}^-\ge1$ by \eqref{orderings of a_m} and we immediately get $\liminf_{t\to+\infty}a_{\nu,\mu}^+(t)>0$. On the other hand, assuming $a_{\nu,\mu}^+=1$, \eqref{eq:middle slits 1} provides $\alpha_{\mu,\mu}^-=1$. But then, via \eqref{orderings of a_m}, we find $\alpha_{\nu+1,\mu}^+=\dots=a_{N-1,\mu}^+=1=a_{\mu-1,\mu}^-=\dots=a_{1,\mu}^-$. Combining everything and taking limits, \eqref{sums for a} becomes
    \begin{equation}\label{eq:middle slits 2}
        \sum_{j=2}^{\nu-1}b_j^+\alpha_{j,\mu}^+ +\sum_{j=\nu}^Nb_j^+=\sum_{j=1}^{\mu}b_j^- +\sum_{j=\mu+1}^Mb_j^-\alpha_{j,\mu}^-.
    \end{equation}
Through quick algebraic considerations, we infer that
    \begin{equation*}
     \sum_{j=\nu}^Nb_j^+\overset{\eqref{orderings of a_m}}{\le}\sum_{j=2}^{\nu-1}b_j^+a_{j,\mu}^++\sum_{j=\nu}^Nb_j^+\overset{\eqref{eq:middle slits 2}}{=}\sum_{j=1}^{\mu}b_j^-+\sum_{j=\mu+1}^Mb_j^-a_{j,\mu}^-\overset{\eqref{orderings of a_m}}{\le}\sum_{j=1}^{\mu}b_j^-\overset{\eqref{orderings of a_m}}{\le}\sum_{j=1}^{\mu_{\nu-1}}b_j^-\overset{\eqref{height with b}}{<}\sum_{j=\nu}^Nb_j^+.
    \end{equation*}
  We reached a contradiction, and thus $\limsup_{t\to+\infty}a_{\nu,\mu}^+(t)<1$ which proves the desired inequality.

\textit{Case II.} This time suppose $\mu=\mn$. Our configuration shows that $S_\mu^-=S_{\mu_\nu}^-\nsubseteq S_\nu^+$ and for the exponents we have $\theta_{\nu+1}^+\ge\theta_\mu^-=\theta_{\mu_\nu}^->\theta_\nu^+\ge\theta_{\mu_\nu+1}^-$. As usual, first assume that $\alpha_{\nu,\mn}^+=0$. In such a circumstance, \eqref{connections of a's} dictates 
      \begin{equation*}
       \lim_{n\to+\infty}\frac{\alpha_{\mn+1,\mn}^-(t_n)}{1-\alpha_{\mn+1,\mn}^-(t_n)}=\lim_{n\to+\infty}\left(\frac{\alpha_{\nu,\mn}^+(t_n)}{1-\alpha_{\nu,\mn}^+(t_n)}\frac{k_{\nu}^-(0)-k_N^+}{k_{\mn+1}^-(0)-k_N^+}e^{(\theta_{\mn+1}^--\theta_{\nu}^+)t_n}\right)=0,
     \end{equation*}
for a suitable sequence $\{t_n\}$ increasing to $+\infty$. So, necessarily, $\alpha_{\mn+1,\mn}^-=0$ as well. Returning to \eqref{orderings of a_m}, we find $\alpha_{\nu-1,\mu_\nu}^+=\dots=\alpha_{2,\mu_\nu}^+=0=\alpha_{\mu_\nu+2,\mu_\nu}^-=\dots=\alpha_{M,\mu_\nu}^-$. Consequently, \eqref{sums for a} turns into
\begin{equation}\label{eq:middle slits 3}
    \sum_{j=\nu+1}^{N-1}b_j^+\alpha_{j,\mu_\nu}^+ +b_N^+=\sum_{j=1}^{\mu_\nu}b_j^-\alpha_{j,\mu_\nu}^-,
\end{equation}
after taking limits. Therefore, through the usual steps,
\begin{equation*}
 \sum_{j=\nu+1}^Nb_j^+\overset{\eqref{orderings of a_m}}{\ge}\sum_{j=\nu+1}^{N-1}b_j^+\alpha_{j,\mu_\nu}^++b_N^+\overset{\eqref{eq:middle slits 3}}{=}\sum_{j=1}^{\mn}b_j^-\alpha_{j,\mu_\nu}^-\overset{\eqref{orderings of a_m}}{\ge}\sum_{j=1}^{\mu_\nu}b_j^-\overset{(\ref{height with b})}{>}\sum_{j=\nu+1}^Nb_j^+   
\end{equation*}
which is absurd. Ergo  $\liminf_{t\to+\infty}\alpha_{\nu,\mn}^+(t)>0$. On the opposing end, assume that $a_{\nu,\mu_\nu}^+=1$, for some limit number corresponding to a sequence $\{t_n\}$. Then, by \eqref{connections of a's},
 \begin{equation*}
       \lim_{n\to+\infty}\frac{\alpha_{\mn,\mn}^-(t_n)}{1-\alpha_{\mn,\mn}^-(t_n)}=\lim_{n\to+\infty}\left(\frac{\alpha_{\nu,\mn}^+(t_n)}{1-\alpha_{\nu,\mn}^+(t_n)}\frac{k_{\nu}^+(0)-k_N^+}{k_{\mn}^-(0)-k_N^+}e^{(\theta_{\mn}^--\theta_{\nu}^+)t_n}\right)=+\infty.
    \end{equation*}
Since every limit number of $a_{\mu_\nu,\mu_\nu}^-(t)$ is finite by Lemma \ref{finiteness}, we get $\alpha_{\mu_\nu,\mu_\nu}^-=1$. Then, \eqref{orderings of a_m} yields $a_{1,\mu_\nu}^-=\dots=a_{\mu_\nu-1,\mu_\nu}^-=1=\alpha_{\nu+1,\mu_\nu}^+=\dots=\alpha_{N-1,\mu_\nu}^+$ as well. As a result, \eqref{sums for a} is transformed into
\begin{equation}\label{eq:middle slits 4}
    \sum_{j=2}^{\nu-1}b_j^+\alpha_{j,\mu_\nu}^+ +\sum_{j=\nu}^{N}b_j^+=\sum_{j=1}^{\mu_\nu}b_j^- +\sum_{j=\mu_\nu+1}^{M}b_j^-\alpha_{j,\mu_\nu}^-,
\end{equation}
after taking limits. Then, in similar to before fashion,
\begin{equation*}
    \sum_{j=\nu}^Nb_j^+ \overset{\eqref{orderings of a_m}}{\le}\sum_{j=2}^{\nu-1}b_j^+\alpha_{j,\mu_\nu}^+ +\sum_{j=\nu}^{N}b_j^+\overset{\eqref{eq:middle slits 4}}{=}\sum_{j=1}^{\mu_\nu}b_j^- +\sum_{j=\mu_\nu+1}^{M}b_j^-\alpha_{j,\mu_\nu}^- \overset{\eqref{orderings of a_m}}{\le}\sum_{j=1}^{\mu_\nu}b_j^- \le \sum_{j=1}^{\mu_{\nu-1}-1}b_j^-\overset{\eqref{height with b}}{<}\sum_{j=\nu}^Nb_j^+,
\end{equation*}
where the second to last inequality is a byproduct of our hypothesis that we are treating a middle slit. Contradiction and whence $\limsup_{t\to+\infty}\alpha_{\nu,\mn}^+<1$. So, we have the desired outcome for the middle slits.

(3) Finally, for the lower right half-lines (whenever they exist), we deal with the indices $\mu\in\{1,\dots,\mu_{N-1}-1\}$. For any such index, we can see that $\theta_{\mu}^-=\theta_{\mu_{N-1}}^->\theta_{N-1}^+$. Then, taking into account that $\alpha_{\mu+1,\mu}^-(t)<0$, for all $t\ge0$,
\begin{equation*}
 \lim_{t\to+\infty}\left(\frac{\alpha_{N-1,\mu}^+(t)}{1-\alpha_{N-1,\mu}^+(t)}(k_{N-1}^+(0)-k_N^+)\right)=\lim_{t\to+\infty}\left(\frac{\alpha_{\mu+1,\mu}^-(t)}{1-\alpha_{\mu+1,\mu}^-(t)}(k_{\mu+1}^-(0)-k_N^+)e^{(\theta_{N-1}^+-\theta_{\mu+1}^-)t}\right)=0   
\end{equation*}
and hence $\alpha_{N-1,\mu}^+=0$. Therefore, \eqref{orderings of a_m} provides $\alpha_{1,\mu}^+=\alpha_{2,N-1}^+=\dots=\alpha_{N-2,\mu}^+=0$. Aiming towards a contradiction, assume that $\alpha_{\mu,\mu}^-=1$. Then, $\alpha_{1,\mu}^-=\alpha_{2,\mu}^-=\dots=\alpha_{\mu-1,\mu}^-=1$. Combining all the aforementioned information, taking limits in \eqref{sums for a} leads to
\begin{equation}\label{eq:middle slits 5}
    b_N^+=\sum_{j=1}^{\mu}b_j^-+\sum_{j=\mu+1}^Mb_j^-\alpha_{j,\mu}^-\overset{\eqref{orderings of a_m}}{\le}\sum_{j=1}^\mu b_j^-\le \sum_{j=1}^{\mu_{N-1}-1}b_j^-\overset{\eqref{height with b}}{<}b_N^+,
\end{equation}
where the second to last inequality is due to the choice of $\mu$. So, we have reached a contradiction, and it is necessary that $\liminf_{t\to+\infty}\alpha_{\mu,\mu}^-(t)>1$. This concludes the proof.
 \end{proof}

 With the aid of the preceding lemma, we are ready to find the rates of the right tip points. The following proof is similar to the one of Proposition \ref{prop:rates of left tips}, so we shall only provide a brief sketch.
\begin{proposition}\label{prop:rates of right tips}
    For $\mu\in\{1,\dots,M-1\}$, the following hold:
    \begin{enumerate}
        \item[\textup{(1)}] (Upper slits) If $\mu\in\{\mu_1,\dots,M-1\}$, then
        \begin{equation*}
          r(\rho_{\mu}^-)=\left(-\sum_{j=2}^{N}b_j^++\sum_{j=1}^{\mu+1}b_j^-\right)\theta_{\mu+1}^-+\sum_{j=\mu+2}^Mb_j^-\theta_j^-.
        \end{equation*}
        \item[\textup{(2)}] (Middle slits) If there exists $\nu\in\{2,\dots,N-2\}$ so that $\mu\in\{\mu_\nu,\dots,\mu_{\nu-1}-1\}$, then
         \begin{equation*}
        r(\rho_{\mu}^-)=-\sum_{j=2}^{\nu}b_j^+\theta_j^++\left(-\sum_{j=\nu+1}^{N}b_j^++\sum_{j=1}^{\mu}b_j^-\right)\theta_{\nu}^++\sum_{j=\mu+1}^Mb_j^-\theta_j^-.    
        \end{equation*}
        \item[\textup{(3)}] (Lower slits) If $\mu\in\{1,\dots,\mu_{N-1}-1\}$, then
        \begin{equation*}
        r(\rho_{\mu}^-)=-\sum_{j=2}^{N-1}b_j^+\theta_j^++\left(\sum_{j=1}^{\mu}b_j^--b_N^+\right)\theta_{\mu}^-+\sum_{j=\mu+1}^Mb_j^-\theta_j^-.    
        \end{equation*}
    \end{enumerate}
\end{proposition}
\begin{proof}
(1) To start with, we fix some $\mu\in\{\mu_1,\dots,M-1\}$. By Lemma \ref{finiteness}, we have that $\liminf_{t\to+\infty}\alpha_{\mu+1,\mu}^-(t)>-\infty$, while by Lemma \ref{non zero-one a's}, we have that $\limsup_{t\to+\infty}a_{\mu+1,\mu}^-(t)<0$. Hence, the fraction $\frac{\alpha_{\mu+1,\mu}^-(t)}{1-\alpha_{\mu+1,\mu}^-(t)}$, $t\ge0$, is compactly contained in $(-1,0)$. Then, through \eqref{k(t)} and \eqref{connections of a's},
\begin{equation*}
 |\rho_{\mu}^-(t)-k_N^+|=\left|\frac{a_{\mu+1,\mu}^-(t)}{1-a_{\mu+1,\mu}^-(t)}\right|(k_\mu^-(t)-k_N^+)\asymp e^{-\theta_{\mu+1}^-t}.
\end{equation*}
As a result,
\begin{equation*}
 |\rho_{\mu}^-(t)-k_j^-(t)|=|\rho_{\mu}^-(t)-k_N^++k_N^+-k_j^-(t)|\asymp 
e^{-\theta_{\mu+1}^-t}+e^{-\theta_j^-t}\asymp e^{-\theta_{\mu+1}^-t}\,,   
\end{equation*}
for every $j\le\mu+1$, since $\theta_j^-\ge\theta_{\mu+1}^-$ for these indices. Otherwise, an identical process shows that $|\rho_{\mu}^-(t)-k_j^-(t)|\asymp e^{-\theta_{j}^-t}$. In similar fashion, for $j\in\{2,\dots,N-1\}$, we have that $\theta_j^+\ge\theta_2^+\>>\theta_{\mu_1}^-\ge\theta_{\mu+1}^-$ due to our configuration. Consequently, $|\rho_{\mu}^-(t)-k_j^+(t)|\asymp e^{-\theta_{\mu+1}^-t}$, for each $j\in\{2,\dots,N-1\}$. Tracing back our steps to \eqref{h_t(N,M)}, we deduce 
\begin{equation*}
 \mathrm{Re}(h(\rho_{\mu}^-(t),t))=b_1^+\log|\rho_{\mu}^-(t)-k_1^+|+\sum_{j=2}^Nb_j^+(-\theta_{\mu+1}^-t)-\sum_{j=1}^{\mu+1}b_j^-(-\theta_{\mu+1}^-t)-\sum_{j=\mu+2}^{M}b_j^+(-\theta_{j}^-t)+O(1),  
\end{equation*}
as $t\rightarrow+\infty$. Keeping in mind that the first term of the right-hand side converges to $b_1^+\log(k_N^+-k_1^+)$, dividing by $t>0$ and taking limits as $t\to+\infty$, the desired outcome follows. 

(2) We proceed with the second statement and the indices $\mu\in\{\mu_{\nu},\dots,\mu_{\nu-1}-1\}$, for some $\nu\in\{2,\dots,N-1\}$. Noting that $a_{\nu,\mu}^+\in(0,1)$ for every limit number by Lemma \ref{non zero-one a's}, then the quantity $\frac{\alpha_{\nu,\mu^+(t)}}{1-\alpha_{\nu,\mu}^+(t)}$, $t\ge0$, is compactly contained in $(0,+\infty)$. As such,
\begin{equation*}
 |\rho_{\mu}^-(t)-k_N^+|=\frac{a_{\nu,\mu}^+(t)}{1-a_{\nu,\mu}^+(t)}|k_\nu^+(t)-k_N^+|\asymp e^{-\theta_{\nu}^+t}.
\end{equation*}
Next, for $j\in\{2,\dots,\nu\}$, we have that $\theta_j^+\le\theta_{\nu}^+$ and therefore 
\begin{equation*}
 |\rho_{\mu}^-(t)-k_j^+(t)|=|\rho_{\mu}^-(t)-k_N^++k_N^+-k_j^+(t)|\asymp e^{-\theta_{\nu}^+t}+e^{-\theta_j^+t}\asymp e^{-\theta_{j}^+t}.   
\end{equation*}
 Analogously, whenever $j\in\{\nu+1,\dots,N-1\}$, our ordering dictates $\theta_j^+\ge\theta_\nu^+$ which will in turn produce $|\rho_\mu^-(t)-k_j^+(t)|\asymp e^{-\theta_{\nu}^+t}$. In addition, for $j\in\{1,\dots,\mu\}$, our configuration shows that $\theta_j^-\ge\theta_{\mu}^-\ge\theta_{\mu_{\nu-1}-1}^-\ge\theta_{\nu}^+$, and so
 \begin{equation*}
   |\rho_{\mu}^-(t)-k_j^-(t)|=|\rho_{\mu}^-(t)-k_N^++k_N^+-k_j^-(t)|\asymp e^{-\theta_{\nu}^+t}+e^{-\theta_j^-t}\asymp e^{-\theta_{\nu}^+t}.
 \end{equation*}
On the other hand, for $j\in\{\mu+1,\dots,M\}$, it is $\theta_j^-\le\theta_{\mu+1}^-\le\theta_{\mu_{\nu}+1}^-\le\theta_\nu^+$ which, as before, signifies $|\rho_{\mu}^-(t)-k_j^-(t)|\asymp e^{-\theta_{j}^-t}$. In view of all the above differences,
\begin{equation*}
 \mathrm{Re}(h(\rho_{\mu}^-(t),t))=\sum_{j=2}^{\nu}b_j^+(-\theta_{j}^+t)+\sum_{j=\nu+1}^{N}b_j^+(-\theta_{\nu}^+t)-\sum_{j=1}^{\mu}b_j^-(-\theta_{\nu}^+t)-\sum_{j=\mu+1}^{M}b_j^+(-\theta_{j}^-t)+O(1),  
\end{equation*}
as $t\rightarrow+\infty$, and the second formula follows.

(3) The exact same procedure is valid for the third part. By the preceding lemma we get that, for $\mu\in\{1,\dots,\mu_{N-1}-1\}$, the ratio $\frac{\alpha_{\mu,\mu}^-(t)}{1-\alpha_{\mu,\mu}^-(t)}$, $t\ge0$, is compactly contained in $(-\infty,-1)$. Applying on \eqref{connections of a's}, we infer that
\begin{equation*}
    |\rho_\mu^-(t)-k_N^+|=\left|\frac{\alpha_{\mu,\mu}^-(t)}{1-\alpha_{\mu,\mu}^-(t)}\right|(k_\mu^-(t)-k_N^+)\asymp e^{-\theta_\mu^- t}.
\end{equation*}
Nevertheless, for $j\in\{2,\dots,N-1\}$, our configuration yields $\theta_j^+\le\theta_{N-1}^+<\theta_{\mu_{N-1}-1}^-\le\theta_\mu^-$, and so $|\rho_{\mu}^-(t)-k_j^+(t)|\asymp e^{-\theta_{j}^+t}$. Moreover, for $j\in\{1,\dots,\mu\}$, we have $\theta_j^-\ge\theta_{\mu}^-$, and thus $|\rho_{\mu}^-(t)-k_j^-(t)|\asymp e^{-\theta_{\mu}^-t}$ through the usual procedure. Finally, for $j\in\{\mu+1,\dots,M\}$, the ordering brings forth $\theta_j^-\le\theta_\mu^-$ which implies $|\rho_\mu^-(t)-k_j^-(t)|\asymp e^{-\theta_j^-t}$ for these indices. Taking everything into consideration, the result follows.
\end{proof}

\addtocontents{toc}{\protect\setcounter{tocdepth}{2}}
\subsection{Comparison of the rates}Recall that we want to find some parameter $\gamma$, so that \eqref{limits of the tips+} and \eqref{limits of the tips-} are satisfied. For this to hold, we request that $\gamma<r(\rho_{\nu}^+)$, for every $\nu=1,\dots,N-1$, and $\gamma>r(\rho_{\mu}^-)$, for every $\mu=1,\dots,M$. 
It is then evident that if 
\begin{equation}
    \label{condition on thetas}
    \min_{\nu}r(\rho_{\nu}^+)>\max_{\mu}r(\rho_{\mu}^-),
\end{equation}
picking any $\gamma$ in between, directly forces \eqref{limits of the tips+} and \eqref{limits of the tips-} to hold. For this reason, our next endeavor is to show that there is a choice of parameters $\theta_j^\pm$ so that \eqref{condition on thetas} holds. Note that as of yet, we have only made remarks about the ordering of the exponents $\theta_j^\pm$ and the interplay between them. We have never dealt with their actual values. Later in the sequel, we will see how their magnitude eventually influences the existence of the desired $\gamma$.

Since we need to determine the minimum and maximum rates of the left and right tip point respectively, we start our investigation by comparing their rates, using Propositions \ref{prop:rates of left tips} and \ref{prop:rates of right tips}. Inspecting Proposition \ref{prop:rates of left tips} more closely, it is possible to see that the rates of the tip points that lie in the same strip are equally given. \CB To see this, we first gather all those strips that indeed contain left half-lines. Towards this aim, we pick all the indices $\nu_1,\dots,\nu_{\lambda}$ so that $S_{\nu_{\omega}}^-$ contains left half-lines as follows.

In terms of the geometry of $
\Omega_0$, we proceed as follows. Due to Proposition \ref{prop:convergence of the roots}, the highest of all the half-lines belonging to $\partial\Omega_0$ is the half-line with tip point $h(\rho_M^-)$. As a consequence, the strip $S_1^+$ contains at least one right slit. Hence, we write $\nu_1=1$. If $\mu_1=1$, then we write $\lambda=1$, and the process is terminated. Otherwise, if $\mu_1-1\ge1$, we pick $\nu_2>\nu_1$ to be the largest index for which $h(\rho_{\nu_2-1}^+)\in S_{\mu_1}^-$. Equivalently, $\nu_2$ is the smallest index, for which the strip $S_{\mu_{\nu_2}}^-$ contains left half-lines, after $S_{\mu_1}^-$. If $\nu_2=N$, this means that all right half-lines are contained in $S_{\mu_1}^-$, hence $\lambda=1$. We continue the process inductively, to find indices 
    \begin{equation}\label{new indices}
        1=\nu_1<\nu_2<\dots<\nu_{\lambda}
    \end{equation}
   so that the only strip to contains left half-lines are the strips $S_{\mu_{\nu_{\omega}}}^-$. We stop the process, until $\nu_{\lambda+1}=N$, thus until we cover all left-hand tip points. See Figure \ref{fig: definition of nu blocks}. 
   
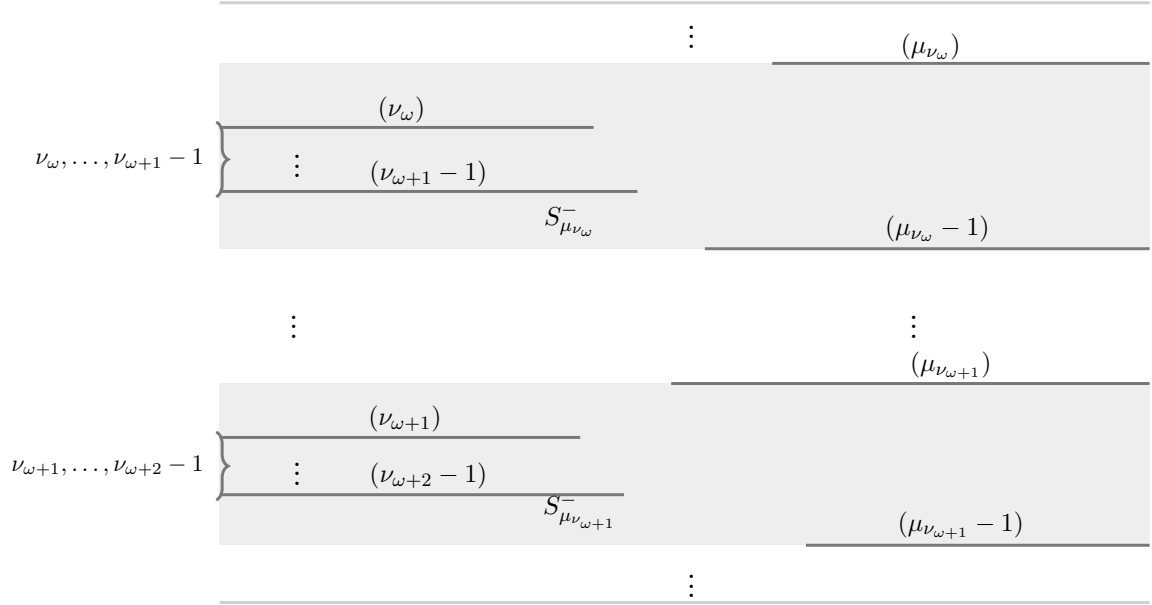
\begin{figure}[ht]
\centering
\resizebox{0.92\linewidth}{!}{%
\begin{tikzpicture}[
    x=1cm,y=1cm,
    every node/.style={font=\normalsize},
    >=Latex
]

\definecolor{linegray}{RGB}{120,120,120}
\definecolor{bandgray}{RGB}{238,238,238}

\tikzset{
  halfline/.style={draw=linegray, line width=1.25pt},
  brace/.style={
    decorate,
    decoration={brace, amplitude=5pt},
    draw=linegray,
    line width=1.25pt
  }
}

\def\xL{0}
\def\xLm{6.00}
\def\xR{13.80}
\def\xRs{8.20}

\def\yTop{8.00}

\def\yRoneTop{7.10}
\def\yLoneA{6.15}
\def\yLoneB{5.20}
\def\yRoneBot{4.35}

\def\yGap{3.30}

\def\yRtwoTop{2.35}
\def\yLtwoA{1.55}
\def\yLtwoB{0.70}
\def\yRtwoBot{-0.05}

\def\yBot{-0.90}

\draw[halfline, opacity=.35] (\xL,\yTop) -- (\xR,\yTop);
\draw[halfline, opacity=.35] (\xL,\yBot) -- (\xR,\yBot);

\fill[bandgray] (0.00,\yRoneTop) rectangle (\xR,\yRoneBot);

\draw[halfline] (\xRs,\yRoneTop) -- (\xR,\yRoneTop);
\node at (10.55,7.35) {$ (\mu_{\nu_\omega}) $};

\draw[halfline] (\xRs-1,\yRoneBot) -- (\xR,\yRoneBot);
\node at (10.65,4.65) {$ (\mu_{\nu_\omega}-1) $};

\draw[halfline] (\xL,\yLoneA) -- (5.55,\yLoneA);
\node at (2.70,6.42) {$ (\nu_\omega) $};

\draw[halfline] (\xL,\yLoneB) -- (6.20,\yLoneB);
\node at (3.10,5.47) {$ (\nu_{\omega+1}-1) $};

\node at (1.15,5.68) {\Large$\vdots$};

\draw[brace] (-0.05,\yLoneA+0.08) -- (-0.05,\yLoneB-0.08)
node[midway,left=1pt,align=right]
{\small $\nu_\omega,\ldots,\nu_{\omega+1}-1$};

\node at (5.20,4.78) {$S^-_{\mu_{\nu_\omega}}$};


\node at (1.10,\yGap) {\Large$\vdots$};
\node at (10.30,\yGap) {\Large$\vdots$};

\node at (7.00,\yRoneTop+0.5) {\Large$\vdots$};
\node at (7.00,\yRtwoBot-0.5) {\Large$\vdots$};

\fill[bandgray] (0,\yRtwoTop) rectangle (\xR,\yRtwoBot);

\draw[halfline] (\xRs-1.5,\yRtwoTop) -- (\xR,\yRtwoTop);
\node at (10.85,2.62) {$ (\mu_{\nu_{\omega+1}}) $};

\draw[halfline] (\xRs+0.5,\yRtwoBot) -- (\xR,\yRtwoBot);
\node at (11.00,0.25) {$ (\mu_{\nu_{\omega+1}}-1) $};

\draw[halfline] (\xL,\yLtwoA) -- (5.35,\yLtwoA);
\node at (2.75,1.83) {$ (\nu_{\omega+1}) $};

\draw[halfline] (\xL,\yLtwoB) -- (6.00,\yLtwoB);
\node at (3.10,0.98) {$ (\nu_{\omega+2}-1) $};

\node at (1.15,1.12) {\Large$\vdots$};

\draw[brace] (-0.05,\yLtwoA+0.08) -- (-0.05,\yLtwoB-0.08)
node[midway,left=1pt,align=right]
{\small $\nu_{\omega+1},\ldots,\nu_{\omega+2}-1$};

\node at (5.35,0.45) {$S^-_{\mu_{\nu_{\omega+1}}}$};

\end{tikzpicture}%
}
\caption{The indices \(\nu_1<\cdots<\nu_\lambda\) mark the maximal
consecutive blocks on which the values \(\mu_\nu\) are the same.
For every \(\nu\in\{\nu_\omega,\dots,\nu_{\omega+1}-1\}\), the left
half-lines corresponding to $\nu$ lie in the same right strip
\(S^-_{\mu_{\nu_\omega}}\). In this way, we cover all strips $S_{\mu}^-$ of $\Omega$ that contain left half-lines. Every other strip $S_j^-$ does not contain left half-lines, hence $\infty_{j}^-$ is not accessed by any left tip point.}
\label{fig: definition of nu blocks}
\end{figure}

   For the sake of completeness, we also describe this procedure in algebraic steps. Since
\[
M\geq \mu_1\geq \mu_2\geq\dots \geq \mu_{N-1}\geq 1,
\]
the
equal values of $(\mu_{\nu})_{\nu=1}^{N-1}$ appear in consecutive groups. In other words, whenever a
certain value of \(\mu_\nu\) occurs more than once, all its occurrences
form one uninterrupted block of indices.

We record these blocks as follows. We denote by
\[
1=\nu_1<\nu_2<\cdots<\nu_\lambda\le N-1
\]
the indices at which the value of the sequence \((\mu_\nu)\) changes.
More precisely, \(\nu_1=1\), and once \(\nu_\omega\) has been chosen, we
define
\[
\nu_{\omega+1}
=
\min\bigl\{\nu>\nu_\omega:\mu_\nu<\mu_{\nu_\omega}\bigr\},
\]
provided that this set is nonempty. If it is empty, the construction
stops. Finally, we set \(\nu_{\lambda+1}=N\).

By construction, for every \(\omega=1,\dots,\lambda\), the sequence
\((\mu_\nu)\) is constant on the interval
\[
\nu_\omega\le \nu\le \nu_{\omega+1}-1.
\]
That is,
\[
\mu_{\nu_\omega}
=
\mu_{\nu_\omega+1}
=
\cdots
=
\mu_{\nu_{\omega+1}-1}.
\]
Moreover, these intervals are maximal with this property: immediately
after the end of such an interval, if another interval follows, the value
of \(\mu_\nu\) strictly decreases. Thus
\[
\mu_{\nu_1}>\mu_{\nu_2}>\cdots>\mu_{\nu_\lambda},
\]
and the numbers \(\mu_{\nu_1},\dots,\mu_{\nu_\lambda}\) are precisely the
distinct values assumed by the sequence \((\mu_\nu)_{\nu=1}^{N-1}\).
   
Having explained how the number $\lambda$ is computed and how the critical indices $\nu_\omega$, $\omega\in\{1,\dots,\lambda\}$, are chosen, we proceed to the following lemmata which are aftermaths of Propositions \ref{prop:rates of left tips} and \ref{prop:rates of right tips}.

    \begin{lemma} \label{min rates}
    The following hold:
       \begin{enumerate}
           \item[\textup{(1)}] For each $\omega=1,\dots,\lambda$, 
            \begin{equation*}
             r(\rho_{\nu_{\omega}}^+)=r(\rho_{j}^+),\quad\text{for all} \quad j\in\{\nu_{\omega},\dots,\nu_{\omega+1}-1\}.   
            \end{equation*}
            
            \item[\textup{(2)}] If the relations
            \begin{equation*}
                \frac{\displaystyle b_{\nu_{\omega+1}}^+-\sum_{j=\mu_{\nu_{\omega+1}+1}}^{\mu_{\nu_{\omega}-1}}b_j^-}{\displaystyle \sum_{j=1}^{\mu_{\nu_{\omega+1}}}b_j^--\sum_{j=\nu_{\omega+1}+1}^Nb_j^+}\le\frac{\theta_{\mu_{\nu_{\omega+1}}}^--\theta_{\mu_{\nu_{\omega}}}^-}{\theta_{\nu_{\omega+1}}^+-\theta_{\mu_{\nu_{\omega}}}^-}
            \end{equation*}
          hold for every $\omega=1,\dots,\lambda-1$, then and only then 
          \begin{equation*}
           r(\rho_{\nu_1}^+)\le r(\rho_{\nu_2}^+)\le\dots\le r(\rho_{\nu_{\lambda}}^+). 
          \end{equation*}
       \end{enumerate}     
    \end{lemma}
    \begin{proof}
        \textup{(1)} Fix $\omega\in\{1,\dots\lambda\}$. By Proposition \ref{prop:rates of left tips}, for every $n=\nu_{\omega},\dots,\nu_{\omega+1}-2$,
      \begin{align*}
          r(\rho_{n+1}^+)-r(\rho_{n}^+)&=-\sum_{j=2}^{n+1}b_j^+\theta_j^++\left(\sum_{j=1}^{\mu_{\nu_{\omega}}}b_j^--\sum_{j=n+2}^{N}b_j^+\right)\theta_{\mu_{\nu_{\omega}}}^-+\sum_{j=\mu_{\nu_{\omega}}+1}^Mb_j^-\theta_j^-\\
          &+\sum_{j=2}^{n}b_j^+\theta_j^+-\left(\sum_{j=1}^{\mu_{\nu_{\omega}}}b_j^--\sum_{j=n+1}^{N}b_j^+\right)\theta_{\mu_{\nu_{\omega}}}^--\sum_{j=\mu_{\nu_{\omega}}+1}^Mb_j^-\theta_j^-\\
          &=-b_{n+1}^+\theta_{n+1}^++b_{n+1}^+\theta_{\mu_{\nu_\omega}}^-,
      \end{align*}
where the first the first equality takes into account the fact that $\mu_n=\mu_{n+1}=\mu_{\nu_\omega}$ due to our construction. Nevertheless, the ordering of the exponents $\theta_j^\pm$ signifies that $\theta_{n+1}^+=\theta_{\mu_{\nu_\omega}}^-$ since $n+1\in\{\nu_\omega+1,\dots,\nu_{\omega+1}-1\}$; see Figures \ref{fig: definition of nu blocks}, \ref{fig:thetas}. Going back to the previous equalities, we obtain $r(\rho_{n+1}^+)=r(\rho_n^+)$ and the desired result follows.

\textup{(2)} Let $\omega\in\{1,\dots,\lambda-1\}$. We will compare the rates $r(\rho_{\nu_{\omega}}^+)$ and $r(\rho_{\nu_{\omega+1}}^+)$. By Proposition \ref{prop:rates of left tips}, we have
\begin{align}\label{eq:nu omega 1}
 \notag   &r(\rho_{\nu_{\omega}}^+)-r(\rho_{\nu_{\omega+1}}^+)=-\sum_{j=2}^{\nu_{\omega}}b_j^+\theta_j^++\sum_{j=2}^{\nu_{\omega+1}}b_j^+\theta_j^++\sum_{j=\mu_{\nu_{\omega}}+1}^Mb_j^-\theta_j^--\sum_{j=\mu_{\nu_{\omega+1}}+1}^Mb_j^-\theta_j^-\\
 \notag   &+\left(\sum_{j=1}^{\mu_{\nu_{\omega}}}b_j^--\sum_{j=\nu_{\omega}+1}^Nb_j^+\right)\theta_{\mu_{\nu_{\omega}}}^--\left(\sum_{j=1}^{\mu_{\nu_{\omega+1}}}b_j^--\sum_{j=\nu_{\omega+1}+1}^Nb_j^+\right)\theta_{\mu_{\nu_{\omega+1}}}^-\\
 \notag   &=\sum_{j=\nu_{\omega}+1}^{\nu_{\omega+1}}b_j^+(\theta_j^+-\theta_{\mu_{\nu_{\omega}}}^-)-\sum_{j=\mu_{\nu_{\omega+1}+1}}^{\mu_{\nu_{\omega}-1}}b_j^-(\theta_j^--\theta_{\mu_{\nu_{\omega}}}^-)\\
    &+\sum_{j=1}^{\mu_{\nu_{\omega+1}}}b_j^-(\theta_{\mu_{\nu_{\omega}}}^--\theta_{\mu_{\nu_{\omega+1}}}^-)-\sum_{j=\nu_{\omega+1}+1}^{N}b_j^+(\theta_{\mu_{\nu_{\omega}}}^--\theta_{\mu_{\nu_{\omega+1}}}^-)
\end{align}
where we have used the inequalities $\nu_\omega<\nu_{\omega+1}$ and $\mu_{\nu_\omega}>\mu_{\nu_{\omega+1}}$. However, considering the geometric configurations with regard to the exponents $\theta_j^\pm$ and the heights of the slits, we have $\theta_j^+=\theta_{\mu_{\nu_{\omega}}}^-$, for $j=\nu_{\omega}+1,\dots,\nu_{\omega+1}-1$, and $\theta_j^-=\theta_{\nu_{\omega+1}}^+$, for $j=\mu_{\nu_{\omega+1}}+1,\dots,\mu_{\nu_{\omega}}-1$. As a result, returning to \eqref{eq:nu omega 1}, all but one terms of the first sum vanish and
\begin{align*}
    r(\rho_{\nu_{\omega}}^+)-r(\rho_{\nu_{\omega+1}}^+)&=b_{{\nu_{\omega+1}}}^+(\theta_{\nu_{\omega+1}}^+-\theta_{\mu_{\nu_{\omega}}}^-)-\sum_{j=\mu_{\nu_{\omega+1}}+1}^{\mu_{\nu_{\omega}-1}}b_j^-(\theta_{\nu_{\omega+1}}^+-\theta_{\mu_{\nu_{\omega}}}^-)\\
    &\quad +\sum_{j=1}^{\mu_{\nu_{\omega+1}}}b_j^-(\theta_{\mu_{\nu_{\omega}}}^--\theta_{\mu_{\nu_{\omega+1}}}^-)-\sum_{j=\nu_{\omega+1}+1}^Nb_j^+(\theta_{\mu_{\nu_{\omega}}}^--\theta_{\mu_{\nu_{\omega+1}}}^-),
\end{align*}
which means that $ r(\rho_{\nu_{\omega}}^+)-r(\rho_{\nu_{\omega+1}}^+)\le0$ if and only if 
\begin{equation}\label{eq:nu omega 2}
\left( b_{\nu_{\omega+1}}^+-\sum_{j=\mu_{\nu_{\omega+1}}+1}^{\mu_{\nu_{\omega}-1}}b_j^-\right)(\theta_{\nu_{\omega+1}}^+-\theta_{\mu_{\nu_{\omega}}}^-)+\left(\sum_{j=1}^{\mu_{\nu_{\omega+1}}}b_j^--\sum_{j=\nu_{\omega+1}+1}^Nb_j^+\right)(\theta_{\mu_{\nu_{\omega}}}^--\theta_{\mu_{\nu_{\omega+1}}}^-)\le0.    
\end{equation}
By \eqref{height with b}, we know that $\sum_{j=1}^{\mu_{\nu_\omega+1}}b_j^->\sum_{j=\nu_{\omega+1}+1}^Nb_j^+$, while our configuration implies $\theta_{\nu_\omega+1}^+>\theta_{\mu_{\nu_{\omega}}}^-$. As a consequence, \eqref{eq:nu omega 2} readily leads to the desired result.
\end{proof}
Next, we proceed with a similar study of how the rates of the right tip points are compared to each other. Briefly, we will see that any right tip points lying in the same strip $S_\nu^+$, $\nu\in\{1,\dots,N\}$, actually share the same rate. For the proof, we follow steps analogous to the previous one.
\begin{lemma}\label{max rates}
The following hold:
\begin{enumerate}
    \item[\textup{(1a)}] Let $\nu\in\{1,\dots,N-1\}$. Then, $r(\rho_j^-)=r(\rho_{\mn}^-)$ for any possible index $j$ such that $h(\rho_j^-)\in S_{\nu}^+$ (excluding the case $j=M$ when $\nu=1$).

    \item[\textup{(1b)}] Let $\nu=N$. Then, $r(\rho_j^-)=r(\rho_{\mu_{N-1}-1}^-)$ for any possible index $j$ such that $h(\rho_j^-)\in S_{\nu}^+$.
    \item[\textup{(2a)}] It is always true that $r(\rho_1^+)>r(\rho_{\mu_1}^-)$.
    \item[\textup{(2b)}] For each $\nu\in\{2,\dots,N-1\}$, $r(\rho_1^+)>r(\rho_{\mn}^-)$ if and only if
    \begin{equation*}
        \sum_{j=2}^{\nu-1}b_j^+(\theta_j^+-\theta_{\mu_1}^-)+\sum_{j=\mn+1}^{\mu_1}b_j^-(\theta_{\mu_1}^--\theta_j^-)+\left(\sum_{j=1}^{\mn}b_j^--\sum_{j=\nu}^{N}b_j^+\right)(\theta_{\mu_1}^--\theta_{\nu}^+)>0.
    \end{equation*}
    \item[\textup{(2c)}] In addition, $r(\rho_1^+)>r(\rho_{\mu_{N-1}-1}^-)$ if and only if
        \begin{equation*}
      \sum_{j=2}^{N-1}b_j^+(\theta_j^+-\theta_{\mu_1}^-)+\sum_{j=\mu_{N-1}}^{\mu_1}b_j^-(\theta_{\mu_1}^--\theta_j^-)+\left(\sum_{j=1}^{\mu_{N-1}-1}b_j^--b_N^+\right)(\theta_{\mu_1}^--\theta_{\mu_{N-1}-1}^-)>0.
    \end{equation*}
\end{enumerate}
\end{lemma}
\begin{proof}
  (1a) We start with the upper slits, that is with the indices $\mu$ such that the tip points $h(\rho_\mu^-)$ lie in $S_{1}^+$. Our configuration shows that such indices necessarily exist and are exactly $\{\mu_1,\dots,M\}$. As usual, we disregard the index $M$ since the tip point $h(\rho_M^-)$ converges to an interior point of $\Omega_0$, see Proposition \ref{prop:convergence of the roots}. By Proposition \ref{prop:rates of right tips}, for each index $\mu\in\{\mu_1+1,\dots,M-1\}$ we have
    \begin{align*}
        r(\rho_{\mu}^-)-r(\rho_{\mu-1}^-)&=\left(-\sum_{j=2}^{N}b_j^++\sum_{j=1}^{\mu+1}b_j^-\right)\theta_{\mu+1}^-+\sum_{j=\mu+2}^Mb_j^-\theta_j^-\\
        &\hspace{1cm}-\left(-\sum_{j=2}^{N}b_j^++\sum_{j=1}^{\mu}b_j^-\right)\theta_{\mu}^--\sum_{j=\mu+1}^Mb_j^-\theta_j^-\\
        &=b_{\mu+1}^-\theta_{\mu+1}^-+\sum_{j=2}^Nb_j^+(\theta_{\mu}^--\theta_{\mu+1}^-)+\sum_{j=1}^{\mu}b_j^-(\theta_{\mu+1}^--\theta_{\mu}^-)-b_{\mu+1}^-\theta_{\mu+1}^-\\
        &=\left(\sum_{j=2}^Nb_j^+-\sum_{j=1}^{\mu}b_j^-\right)(\theta_{\mu}^--\theta_{\mu+1}^-)=0
    \end{align*}
    where the last equality follows from the fact that $\theta_{\mu_1+1}^-=\dots=\theta_{M-1}^-=\theta_M^-$ which is due to our construction, because $S_{\mu_1+1}^-\cup\dots\cup S_{M-1}^-\cup S_M^-\subset S_1^+$. We continue with the middle slits, namely the tip points lying in some $S_{\nu}^+$, for $\nu\in\{2,\dots,N-1\}$ and hence with the indices $\{\mu_{\nu},\dots,\mu_{\nu-1}-1\}$, assuming such indices exist  or equivalently assuming that $h(\rho_{\mu_{\nu}+1}^-)\in S_\nu^+$. Let $\mu\in\{\mu_{\nu}+1,\dots,\mu_{\nu-1}-1\}$. Then, Proposition \ref{prop:rates of right tips} yields
    \begin{align*}
        r(\rho_{\mu}^-)-r(\rho_{\mu-1}^-)&=-\sum_{j=2}^{\nu}b_j^+\theta_j^++\left(-\sum_{j=\nu+1}^{N}b_j^++\sum_{j=1}^{\mu}b_j^-\right)\theta_{\nu}^++\sum_{j=\mu+1}^Mb_j^-\theta_j^-\\
        &\hspace{1cm}+\sum_{j=2}^{\nu}b_j^+\theta_j^+-\left(-\sum_{j=\nu+1}^{N}b_j^++\sum_{j=1}^{\mu-1}b_j^-\right)\theta_{\nu}^+-\sum_{j=\mu}^Mb_j^-\theta_j^-\\
        &=b_{\mu}^-\theta_{\nu}^+-b_{\mu}^-\theta_{\mu}^-.
    \end{align*}
    Nevertheless, $S_\mu^-\subset S_{\nu}^+$, for all $\mu=\mu_{\nu}+1,\dots,\mu_{\nu-1}-1$, and thus $\theta_\mu^-=\theta_\nu^+$ for any such $\mu$. As a result, $r(\rho_{\mu}^-)=r(\rho_{\mu-1}^-)$. 
    
    (1b) We move on to the lower slits, or equivalently those whose tip points lie in $S_N^+$. By our configuration, these slits, whenever they exist, correspond to the indices $\{1,\dots,\mu_{N-1}-1\}$. Through Proposition \ref{prop:rates of right tips}, for each $\mu\in\{2,\dots,\mu_{N-1}-1\}$, we obtain
    \begin{align*}
         r(\rho_{\mu}^-)-r(\rho_{\mu-1}^-)&=-\sum_{j=2}^{N-1}b_j^+\theta_j^++\left(\sum_{j=1}^{\mu}b_j^--b_N^+\right)\theta_{\mu}^-+\sum_{j=\mu+1}^Mb_j^-\theta_j^-\\
         &+\sum_{j=2}^{N-1}b_j^+\theta_j^+-\left(\sum_{j=1}^{\mu-1}b_j^--b_N^+\right)\theta_{\mu-1}^--\sum_{j=\mu}^Mb_j^-\theta_j^-\\
         &=\left(\sum_{j=1}^{\mu-1}b_j^--b_N^+\right)(\theta_{\mu}^--\theta_{\mu-1}^-)=0,
    \end{align*}
    with the last equality following via similar arguments as above.

(2a) If $\mu_1=M$, then the calculation is trivial, since $r(\rho_M^-)=0<r(\rho_1^+)$ by Propositions \ref{prop:convergence of the roots} and \ref{prop:rates of left tips}. Otherwise, recalling that the index $\mu_1$ corresponds to an upper slit and using Propositions \ref{prop:rates of left tips} and \ref{prop:rates of right tips}, we infer that
\begin{align*}
    r(\rho_1^+)-r(\rho_{\mu_1}^-&)=\left(\sum_{j=1}^{\mu_1}b_j^--\sum_{j=2}^{N}b_j^+\right)\theta_{\mu_1}^-+\sum_{j=\mu_1+1}^Mb_j^-\theta_j^-\\
    &-\left(-\sum_{j=2}^{N}b_j^++\sum_{j=1}^{\mu_1+1}b_j^-\right)\theta_{\mu_1+1}^--\sum_{j=\mu_1+2}^Mb_j^-\theta_j^-\\
    &=\left(-\sum_{j=2}^{N}b_j^++\sum_{j=1}^{\mu_1}b_j^-\right)(\theta_{\mu_1}^--\theta_{\mu_1+1}^-)>0,
\end{align*}
where the inequality is due to \eqref{height with b} and the fact that $\mu_1+1>\mu_1$ implies $\theta_{\mu_1}^-\ge\theta_{\mu_1+1}^-$ with our configuration denying the possibility of equality. 

(2b) On top of the previous case, for $\nu\in\{2,\dots,N-1\}$, using the same propositions and keeping in mind that $\mu_1\ge\mu_\nu$, we deduce
\begin{align*}
    r(\rho_1^+)-r(\rho_{\mn}^-)&=\left(\sum_{j=1}^{\mu_1}b_j^--\sum_{j=2}^{N}b_j^+\right)\theta_{\mu_1}^-+\sum_{j=\mu_1+1}^Mb_j^-\theta_j^-\\
    &+\sum_{j=2}^{\nu}b_j^+\theta_j^+-\left(-\sum_{j=\nu+1}^{N}b_j^++\sum_{j=1}^{\mu_{\nu}}b_j^-\right)\theta_{\nu}^+-\sum_{j=\mu_{\nu}+1}^Mb_j^-\theta_j^-\\
    &=\sum_{j=2}^{\nu}b_j^+(\theta_j^+-\theta_{\mu_1}^-)+\sum_{j=\mn+1}^{\mu_1}b_j^-(\theta_{\mu_1}^--\theta_j^-)\\
    &\hspace{1cm}+\left(\sum_{j=1}^{\mn}b_j^--\sum_{j=\nu+1}^{N}b_j^+\right)(\theta_{\mu_1}^--\theta_{\nu}^+)\\
    &=\sum_{j=2}^{\nu-1}b_j^+(\theta_j^+-\theta_{\mu_1}^-)+\sum_{j=\mn+1}^{\mu_1}b_j^-(\theta_{\mu_1}^--\theta_j^-)\\
    &\hspace{1cm}+\left(\sum_{j=1}^{\mn}b_j^--\sum_{j=\nu}^{N}b_j^+\right)(\theta_{\mu_1}^--\theta_{\nu}^+),
\end{align*}
where in the last equality, we just moved one term from the first sum to the last sum. Thus, the desired outcome is proved.

(2c) To conclude the proof, we work with the highest lower slit, assuming it exists. In this scenario, Propositions \ref{prop:rates of left tips} and \ref{prop:rates of right tips} lead to
\begin{align*}
      r(\rho_{1}^+)-r(\rho_{\mu_{N-1}-1}^-)=&\left(\sum_{j=1}^{\mu_1}b_j^--\sum_{j=2}^{N}b_j^+\right)\theta_{\mu_1}^-+\sum_{j=\mu_1+1}^Mb_j^-\theta_j^-\\
&+\sum_{j=2}^{N-1}b_j^+\theta_j^+-\left(\sum_{j=1}^{\mu_{N-1}-1}b_j^--b_N^+\right)\theta_{\mu_{N-1}-1}^--\sum_{j=\mu_{N-1}}^Mb_j^-\theta_j^-\\
&=\sum_{j=\mu_{N-1}}^{\mu_1}b_j^-(\theta_{\mu_1}^--\theta_j^-)+\sum_{j=2}^{N-1}b_j^+(\theta_j^+-\theta_{\mu_1}^-)\\
&\hspace{1cm}+\left(\sum_{j=1}^{\mu_{N-1}-1}b_j^--b_N^+\right)(\theta_{\mu_1}^--\theta_{\mu_{N-1}-1}^-),
\end{align*}
which is the desired result.    
\end{proof}

The two lemmata above actually dictate a choice for the exponents in such a way that $r(\rho_1^+)=\min_{\nu\in\{1,\dots,N-1\}} r(\rho_{\nu}^+)$ and another choice so that $r(\rho_1^+)>r(\rho_{\mu}^-)$, for every $\mu\in\{1,\dots,M-1\}$. Below, we show that these two conditions can be taken simultaneously.
\begin{corollary}\label{existance of thetas}
    There exists a choice of exponents $\theta_j^{\pm}$, so that 
    \begin{equation*}
      \min_{\nu\in\{1.\dots,N-1\}}r(\rho_{\nu}^+)>\max_{\mu\in\{1,\dots,M-1\}}r(\rho_{\mu}^-).   
    \end{equation*}
   
\end{corollary}
\begin{proof}
    First of all, recalling the notation in \eqref{new indices}, we have $\nu_1=1$. By Lemma \ref{max rates} (2a), we know that $r(\rho_1^+)>r(\rho_{\mu_1^-})=r(\rho_{\mu_{\nu_1}}^-)$. Thus, by part (1a) of the same result, we also get $r(\rho_1^+)>r(\rho_\mu^-)$, for all indices $\mu\in\{\mu_1,\dots,M-1\}$. The following procedure stops until $\nu_\lambda=N$. Then, we compare only $r(\rho_1^+)$ with $r(\rho_{\mu_{N-1}-1}^-)$, as the final step of the proof shows. Next, by  Lemma \ref{max rates} (2b), we have $r(\rho_1^+)>r(\rho_{\mu_{\nu_2}}^-)$ if and only if
    \begin{equation}\label{eq:existence of gamma 1}
         \sum_{j=2}^{\nu_2-1}b_j^+(\theta_j^+-\theta_{\mu_1}^-)+\sum_{j=\mu_{\nu_2}+1}^{\mu_1}b_j^-(\theta_{\mu_1}^--\theta_j^-)+\left(\sum_{j=1}^{\mu_{\nu_2}}b_j^--\sum_{j=\nu_2}^{N}b_j^+\right)(\theta_{\mu_1}^--\theta_{\nu_2}^+)>0.   
    \end{equation}
    To ensure the positivity of the last quantity, we may choose a large enough $\theta_{\nu_2}^+$ so that the preceding inequality holds. Note that in the above expression, $\theta_{\nu_2}^+$ is the largest exponent due to our configuration and 
    \begin{align*}
        \sum_{j=(\nu_2-1)+1}^Nb_j^+\overset{\eqref{height with b}}{>}\sum_{j=1}^{\mu_{\nu_2-1}-1}b_j^-=\sum_{j=1}^{\mu_{\nu_1}-1}b_j^-=\sum_{j=1}^{\mu_{\nu_2}}b_j^-
    \end{align*}
      since $\mu_{\nu_1}-1=\mu_{\nu_2}$.
     So, enlarging $\theta_{\nu_2}^+$ actually enlarges the whole expression, without messing with the rest of the exponents. Due to the equalities in Lemma \ref{max rates} (1a), the same exponent provides $r(\rho_1^+)>r(\rho_\mu^-)$, for all indices $\mu\in\{\mu_{\nu_2},\dots,\mu_1-1\}$. Next, by Lemma \ref{min rates} (2), we have that $r(\rho_1^+)\le r(\rho_{\nu_2}^+)$ if and only if
            \begin{equation}\label{eq:existence of gamma 2}
                \frac{b_{\nu_{2}}^+-\sum_{j=\mu_{\nu_{2}+1}}^{\mu_{\nu_{1}-1}}b_j^-}{\sum_{j=1}^{\mu_{\nu_{2}}}b_j^--\sum_{j=\nu_{2}+1}^Nb_j^+}\le\frac{\theta_{\mu_{\nu_{2}}}^--\theta_{\mu_{\nu_{1}}}^-}{\theta_{\nu_{2}}^+-\theta_{\mu_{\nu_{1}}}^-}.
            \end{equation}
    Geometrically, we see that $\theta_{\mu_{\nu_2}}^-$ is the largest exponent in the preceding relation. This means that we can pick $\theta_{\mu_{\nu_2}}^-$ large enough to render \eqref{eq:existence of gamma 2} true, without affecting $\theta_{\nu_2}^+$ though. In this way, relations \eqref{eq:existence of gamma 1} and \eqref{eq:existence of gamma 2} hold simulatenously. In addition, due to the equalities in Lemma \ref{min rates} (1), we have that $r(\rho_1^+)\le r(\rho_{\nu}^+)$, for all $\nu\in\{\nu_2,\dots,\nu_3-1\}$. In the next step, we know that $r(\rho_1^+)>r(\rho_{\mu_{\nu_3}}^-)$ if and only if
    \begin{equation}\label{eq:existence of gamma 3}
        -\sum_{j=2}^{\nu_{3}-1}b_j^+(\theta_j^+-\theta_{\mu_1}^-)+\sum_{j=\mu_{\nu_3}+1}^{\mu_1}b_j^-(\theta_{\mu_1}^--\theta_j^-)+\left(\sum_{j=1}^{\mu_{\nu_3}}b_j^--\sum_{j=\nu_3}^{N}b_j^+\right)(\theta_{\mu_1}^--\theta_{\nu_3}^+)>0.
    \end{equation}
    As before, \eqref{height with b} dictates
    \begin{equation*}
        \left(\sum_{j=1}^{\mu_{\nu_3}}b_j^--\sum_{j=\nu_3}^{N}b_j^+\right)<0.
    \end{equation*}
    So, picking a large enough exponent $\theta_{\nu_3}^+$ ensures the positivity in \eqref{eq:existence of gamma 3}. Again, such a choice is allowed since $\theta_{\nu_3}^+$ is the dominant exponent in \eqref{eq:existence of gamma 3} in view of our configuration. In addition, $\theta_{\nu_3}^+$ is strictly larger than $\theta_{\mu_{\nu_2}}^-$, and thus, our choice of $\theta_{\nu_3}^+$ certifies that $r(\rho_1^+)>r(\rho_{\mu_{\nu_3}}^-)$ while also allowing \eqref{eq:existence of gamma 1} and \eqref{eq:existence of gamma 2} to still hold. Of course, Lemma \ref{max rates} (1a) shows that $r(\rho_1^+)>r(\rho_\mu^-)$, for all indices $\mu\in\{\mu_{\nu_3},\dots,\mu_{\nu_2}-1\}$. So, up to this point, we have made an appropriate choice for all exponents $\theta_{j}^{\pm}\le\theta_{\nu_3}^+$. Then, Lemma \ref{min rates} (1) shows that $r(\rho_{\nu_2}^+)\le r(\rho_{\nu_3}^+)$ if and only if 
            \begin{equation}\label{eq:existence of gamma 4}
                \frac{b_{\nu_{3}}^+-\sum_{j=\mu_{\nu_{3}+1}}^{\mu_{\nu_{2}-1}}b_j^-}{\sum_{j=1}^{\mu_{\nu_{3}}}b_j^--\sum_{j=\nu_{3}+1}^Nb_j^+}\le\frac{\theta_{\mu_{\nu_{3}}}^--\theta_{\mu_{\nu_{2}}}^-}{\theta_{\nu_{3}}^+-\theta_{\mu_{\nu_{2}}}^-}.
            \end{equation}
    Following an analogous to before reasoning, choosing a large enough $\theta_{\mu_{\nu_3}}^-$ does not affect any other exponent and relations \eqref{eq:existence of gamma 1}, \eqref{eq:existence of gamma 2}, \eqref{eq:existence of gamma 3}, \eqref{eq:existence of gamma 4}, hold concurrently. Clearly, we may continue the process inductively, so that at each step we suitably choose all exponents $\theta_j^{\pm}\le\theta_{\nu_{\omega}}^+$ and then $\theta_j^{\pm}\le\theta_{\mu_{\nu_{\omega}}}^-$, as we alternate from step to step between Lemmata \ref{min rates} and \ref{max rates}, up to the exponents $\theta_{N-1}^+$ and $\theta_{\mu_{N-1}}^-$. Observe that $\mu_{N-1}=\mu_{\nu_{\lambda}}$ by \eqref{new indices}. So, the current choice of exponents already provides
    \begin{equation}\label{eq:existence of gamma 5}
        r(\rho_1^+)\le r(\rho_{\nu_2}^+) \le \dots \le r(\rho_{\nu_\lambda}^+).
    \end{equation}
    Keeping Lemma \ref{min rates} (1) in mind, we obtain $r(\rho_1^+)=\min_{\nu\in\{1,\dots,N-1\}}r(\rho_\nu^+)$. In addition, the current choice provides $r(\rho_1^+)>r(\rho_\mu^-)$, for all $\mu\in\{\mu_{N-1},\dots,M-1\}$ as we have proved. 
    \par All that remains is the final step which is to compare the rate $r(\rho_1^+)$ with the rates of the lower right slits. Via Lemma \ref{max rates} (2c), $r(\rho_1^+)>r(\rho_{\mu_{N-1}-1}^-)$ if and only if 
    \begin{equation}\label{eq:existence of gamma 6}
      \sum_{j=2}^{N-1}b_j^+(\theta_j^+-\theta_{\mu_1}^-)+\sum_{j=\mu_{N-1}}^{\mu_1}b_j^-(\theta_{\mu_1}^--\theta_j^-)+\left(\sum_{j=1}^{\mu_{N-1}-1}b_j^--b_N^+\right)(\theta_{\mu_1}^--\theta_{\mu_{N-1}-1}^-)>0.
    \end{equation}
    As usual, \eqref{height with b} shows that 
    \begin{equation*}
        \left(\sum_{j=1}^{\mu_{N-1}-1}b_j^--b_N^+\right)<0,
    \end{equation*}
    and thus taking a sufficienlty large $\theta_{\mu_{N-1}-1}^-$ ensures that \eqref{eq:existence of gamma 6} is satisfied, without influencing all the preceding equalities, since by construction, $\theta_{\mu_{N-1}-1}^-$ is actually the largest of all the exponents $\theta_j^\pm$, up to equality. Finally, combining \eqref{eq:existence of gamma 6} with Lemma \ref{max rates} (1b) yields $r(\rho_1^+)>r(\rho_\mu^-)$, for all $\mu\in\{1,\dots,\mu_{N-1}-1\}$. All in all, we obtain $r(\rho_1^+)>r(\rho_\mu^-)$, for all $\mu\in\{1,\dots,M-1\}$ and therefore
    \begin{equation}\label{eq:existence of gamma 7}
        r(\rho_1^+)> \max_{\mu\in\{1,\dots,M-1\}}r(\rho_\mu^-).
    \end{equation}
    Considering \eqref{eq:existence of gamma 5} and \eqref{eq:existence of gamma 7}, we get the desired outcome and the proof is complete.
\end{proof}

\begin{figure}[ht]
\centering
\resizebox{0.92\linewidth}{!}{%
\begin{tikzpicture}[
    x=1cm,y=1cm,
    every node/.style={font=\normalsize},
    >=Latex
]

\definecolor{linegray}{RGB}{120,120,120}
\definecolor{bandgray}{RGB}{238,238,238}

\tikzset{
  halfline/.style={draw=linegray, line width=1.25pt},
  brace/.style={
    decorate,
    decoration={brace, amplitude=5pt},
    draw=linegray,
    line width=1.25pt
  }
}

\def\xL{0}
\def\xLm{6.00}
\def\xR{13.80}
\def\xRs{8.20}

\def\yTop{8.00}

\def\yRoneTop{7.10}
\def\yLoneA{6.15}
\def\yLoneB{5.20}
\def\yRoneBot{4.35}

\def\yGap{3.30}

\def\yRtwoTop{2.35}
\def\yLtwoA{1.55}
\def\yLtwoB{0.70}
\def\yRtwoBot{-0.05}

\def\yBot{-0.90}

\draw[halfline, opacity=.35] (\xL,\yTop) -- (\xR,\yTop);
\draw[halfline, opacity=.35] (\xL,\yBot) -- (\xR,\yBot);

\fill[bandgray] (0.00,\yRoneTop) rectangle (\xR,\yRoneBot);

\draw[halfline] (\xRs,\yRoneTop) -- (\xR,\yRoneTop);
\node at (10.55,7.35) {$ (\mu_{\nu_\omega}) $};

\draw[halfline] (\xRs-1,\yRoneBot) -- (\xR,\yRoneBot);
\node at (10.65,4.65) {$ (\mu_{\nu_\omega}-1) $};

\draw[halfline] (\xL,\yLoneA) -- (5.55,\yLoneA);
\node at (2.70,6.42) {$ (\nu_\omega) $};

\draw[halfline] (\xL,\yLoneB) -- (6.20,\yLoneB);
\node at (3.10,5.47) {$ (\nu_{\omega+1}-1) $};

\node at (1.15,5.68) {\Large$\vdots$};

\node at (8.15,5.78) {$S^-_{\mu_{\nu_\omega}}\sim\theta_{\mu_{\nu_{\omega}}}^-$};


\node at (1.10,\yGap) {\Large$\vdots$};
\node at (10.30,\yGap) {\Large$\vdots$};

\node at (7.00,\yRoneTop+0.5) {\Large$\vdots$};
\node at (7.00,\yRtwoBot-0.5) {\Large$\vdots$};

\fill[bandgray] (0,\yRtwoTop) rectangle (\xR,\yRtwoBot);

\draw[halfline] (\xRs-1.5,\yRtwoTop) -- (\xR,\yRtwoTop);
\node at (10.85,2.62) {$ (\mu_{\nu_{\omega+1}}) $};

\draw[halfline] (\xRs+0.5,\yRtwoBot) -- (\xR,\yRtwoBot);
\node at (11.00,0.25) {$ (\mu_{\nu_{\omega+1}}-1) $};

\draw[halfline] (\xL,\yLtwoA) -- (5.35,\yLtwoA);
\node at (2.75,1.83) {$ (\nu_{\omega+1}) $};

\draw[halfline] (\xL,\yLtwoB) -- (6.00,\yLtwoB);
\node at (3.10,0.98) {$ (\nu_{\omega+2}-1) $};

\node at (1.15,1.12) {\Large$\vdots$};

\node at (3.35,3.5) {$S^+_{\nu_{\omega+1}}\sim\theta_{\nu_{\omega+1}}^+$};

\node at (8.35,1.05) {$S^-_{\mu_{\nu_{\omega+1}}}\sim\theta_{\mu_{\nu_{\omega+1}}}^-$};

\draw[->, line width=0.45pt, color=linegray] (5,3) --(7,1.5) ;
\draw[->, line width=0.45pt, color=linegray] (7,4.75) -- (5,3.85);

\end{tikzpicture}%
}
\caption{We divide the strip $\Omega_0$ into the consecutive strips $S_{\nu_{\omega}}^{+}$ and $S_{\mu_{\nu_{\omega}}}^-$ and we inductively choose the exponents $\theta_{\nu_{\omega}}^+$ and $\theta_{\mu_{\nu_{\omega}}}^-$ until we cover all these strips.}
\label{fig: strips}
\end{figure}
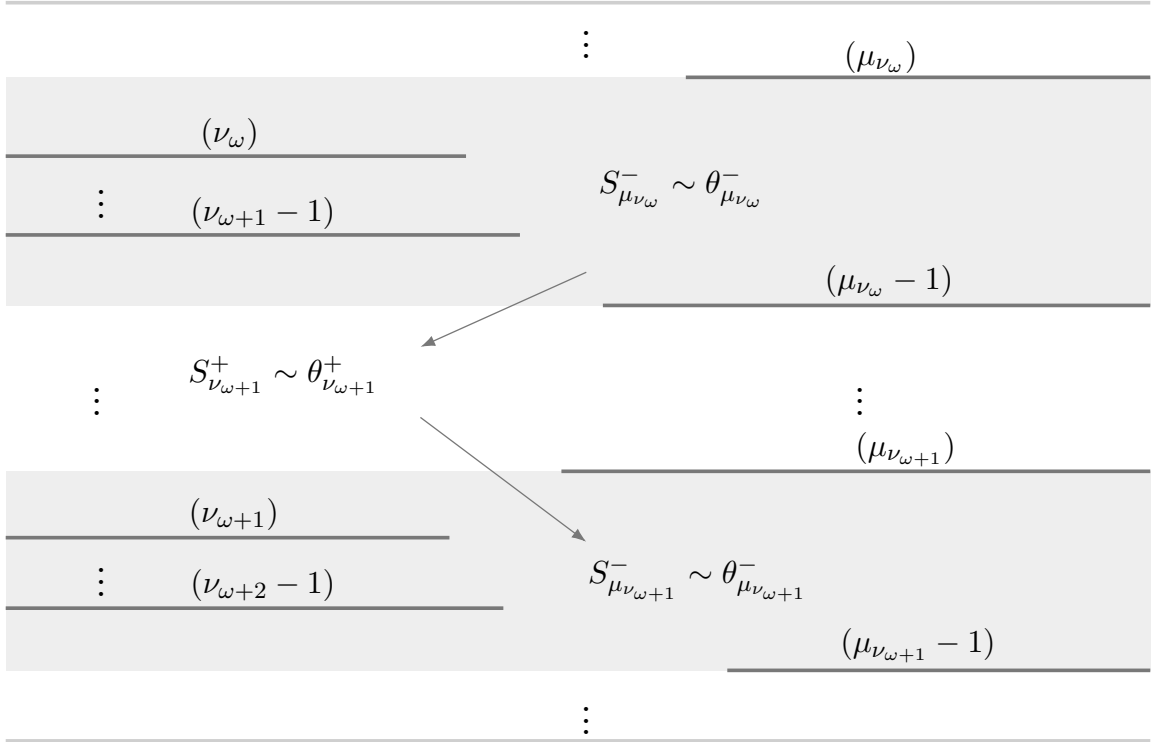

We investigate the idea of the above proof geometrically through Figure \ref{fig: strips}. We cover $\Omega_0$ with the (overlapping) strips $S_{\nu_{\omega}}^+$ and $S_{\mu_{\nu_{\omega}}}^-$, for all $\omega=1,\dots,\lambda.$ By the configuration of the exponents, we intuitively relate each such strip with $\theta_{\nu_{\omega}}^+$ and $\theta_{\mu_{{\nu_{\omega}}}}^-$, respectively. In Figure \ref{fig: strips} we demonstrate this via the notation $\sim$. Also, $\theta_{\nu_{\omega}}^+=\theta_j^-$ for all $j$, so that $S_j^-\subset S_{\nu_{\omega}}^+$ and similarly, $\theta_{\mu_{\nu_{\omega}}}^-=\theta_j^+$, for every $j$, such that $S_j^+\subset S_{\mu_{\nu_{\omega}}}^-$. This means that if we pick $\theta_{\nu_{\omega}}^+$ large enough, then this choice does not affect the the exponents $\theta_j^{\pm}$ that are smaller than $\theta_{\nu_{\omega}}^+$. We argue similarly with $\theta_{\mu_{\nu_{\omega}}}^-$.

\subsection{Monotonicity of the tip points}Evidently, there can be more than one choices of exponents so that Lemma \ref{existance of thetas} is true. Starting with such a choice, there exists a parameter $\gamma\in(\max_{\mu\in\{1,\dots,M-1\}}r(\rho_\mu^-),\min_{\nu\in\{1,\dots,N-1\}}r(\rho_\nu^+))$. Given that $\min_{\nu\in\{1,\dots,N-1\}}r(\rho_\nu^+)=r(\rho_1^+)>0$ as evidenced by Proposition \ref{prop:rates of left tips}, this parameter $\gamma$ can be chosen positive. Then, taking the definition of the rates into account, we see that relations \eqref{limits of the tips+} and \eqref{limits of the tips-} hold at once. In fact, we shall see that the asymptotic behavior dictated by these two relations is monotonous with respect to time. 

\begin{lemma}\label{monotonicity}
Let $\gamma$ be as above. Then the following hold:
\begin{enumerate}
    \item[\textup{(1)}] For each \(\nu\in\{1,\dots,N-1\}\), the function
    \[
        t\mapsto \operatorname{Re} h(\rho_\nu^+(t),t)-\gamma t
    \]
    is strictly increasing on \([0,+\infty)\), and $\lim_{t\to+\infty}(\mathrm{Re}h(\rho_\nu^+(t),t)-\gamma t)=+\infty$.
    \item[\textup{(2)}] For each \(\mu\in\{1,\dots,M\}\), the function
    \[
        t\mapsto \operatorname{Re} h(\rho_\mu^-(t),t)-\gamma t
    \]
    is strictly decreasing on \([0,+\infty)\), and $\lim_{t\to+\infty}(\mathrm{Re}(h(\rho_\mu^-(t),t)-\gamma t)=-\infty$.
\end{enumerate}
\end{lemma}
\begin{proof}
We begin with a computation that will be used several times. Let
\(\xi(t)\) be any of the points \(\rho_\nu^+(t)\) or
\(\rho_\mu^-(t)\). Since \(h'(\xi(t),t)=0\), the chain rule gives
\[
\frac{\partial}{\partial t}
\left(
    \operatorname{Re}h(\xi(t),t)-\gamma t
\right)
=
\frac{\partial h}{\partial t}(\xi(t),t)-\gamma .
\]
Using \eqref{h_t(N,M)}, and recalling that \(k_1^+\) and \(k_N^+\) are
independent of \(t\), we get
\begin{equation}\label{eq:monotonicity 1}
   \frac{\partial h}{\partial t}(\xi(t),t)
=
-\sum_{j=2}^{N-1}
b_j^+
\frac{(k_j^+)'(t)}{\xi(t)-k_j^+(t)}
+
\sum_{j=1}^{M}
b_j^-
\frac{(k_j^-)'(t)}{\xi(t)-k_j^-(t)}. 
\end{equation}
However, for all possible indices $j$, by \eqref{k(t)} we have
\[
    (k_j^\pm)'(t)
    =
    -\theta_j^\pm\bigl(k_j^\pm(t)-k_N^+\bigr).
\]
Therefore
\[
\frac{(k_j^\pm)'(t)}{\xi(t)-k_j^\pm(t)}
=
-\theta_j^\pm
\left(
    -1+
    \frac{\xi(t)-k_N^+}{\xi(t)-k_j^\pm(t)}
\right).
\]
Consequently, applying on \eqref{eq:monotonicity 1}, we obtain
\begin{equation}\label{eq:monotonicity 2}
\frac{\partial}{\partial t}\left(\operatorname{Re}h(\xi(t),t)-\gamma t\right)
=
\sum_{j=2}^{N-1}\theta_j^+b_j^+\left(-1+\frac{\xi(t)-k_N^+}{\xi(t)-k_j^+(t)}\right)
-
\sum_{j=1}^{M}\theta_j^-b_j^-\left(-1+\frac{\xi(t)-k_N^+}{\xi(t)-k_j^-(t)}\right)-\gamma.
\end{equation}

(1) We proceed to the main body of the proof, commencing with the monotonicity of the left tip points. Fix \(\nu\in\{2,\dots,N-1\}\). Rewriting \eqref{eq:monotonicity 2} with \(\xi(t)=\rho_\nu^+(t)\) and recalling Definition \ref{ratios +-}, we obtain
\begin{align*}
\frac{\partial}{\partial t}
\left(
    \operatorname{Re}h(\rho_\nu^+(t),t)-\gamma t
\right)
&=
\sum_{j=2}^{N-1}
\theta_j^+b_j^+
\bigl(-1+l_{j,\nu}^+(t)\bigr)-
\sum_{j=1}^{M}
\theta_j^-b_j^-
\bigl(-1+l_{j,\nu}^-(t)\bigr)
-\gamma \\
&=
\sum_{j=2}^{N-1}
\theta_j^+b_j^+l_{j,\nu}^+(t)
-
\sum_{j=1}^{M}
\theta_j^-b_j^-l_{j,\nu}^-(t) 
-
\sum_{j=2}^{N-1}\theta_j^+b_j^+
+
\sum_{j=1}^{M}\theta_j^-b_j^-
-\gamma .
\end{align*}
Since \(\gamma<r(\rho_\nu^+)\), it follows that

\begin{equation}\label{eq:monotonicity 3}
\frac{\partial}{\partial t}\left(\operatorname{Re}h(\rho_\nu^+(t),t)-\gamma t\right)
>
\sum_{j=2}^{N-1}\theta_j^+b_j^+l_{j,\nu}^+(t)-\sum_{j=1}^{M}\theta_j^-b_j^-l_{j,\nu}^-(t)
-\sum_{j=2}^{N-1}\theta_j^+b_j^+
+
\sum_{j=1}^{M}\theta_j^-b_j^-
-r(\rho_\nu^+).
\end{equation}
But for \(\nu\in\{2,\dots,N-1\}\), Proposition~\ref{prop:rates of left tips} gives,
\[
r(\rho_\nu^+)
=
-\sum_{j=2}^{\nu}b_j^+\theta_j^+
+
\left(
    \sum_{j=1}^{\mu_\nu}b_j^-
    -
    \sum_{j=\nu+1}^{N}b_j^+
\right)\theta_{\mu_\nu}^-
+
\sum_{j=\mu_\nu+1}^{M}b_j^-\theta_j^- .
\]
Substituting this expression into \eqref{eq:monotonicity 3}, we obtain
\begin{align*}
\frac{\partial}{\partial t}
\left(
    \operatorname{Re}h(\rho_\nu^+(t),t)-\gamma t
\right)
&>
\sum_{j=2}^{N-1}
\theta_j^+b_j^+l_{j,\nu}^+(t)
-
\sum_{j=1}^{M}
\theta_j^-b_j^-l_{j,\nu}^-(t) 
-
\sum_{j=2}^{N-1}\theta_j^+b_j^+
+
\sum_{j=1}^{M}\theta_j^-b_j^- \\
&\quad
+
\sum_{j=2}^{\nu}b_j^+\theta_j^+
-
\left(
    \sum_{j=1}^{\mu_\nu}b_j^-
    -
    \sum_{j=\nu+1}^{N}b_j^+
\right)\theta_{\mu_\nu}^-
-
\sum_{j=\mu_\nu+1}^{M}b_j^-\theta_j^- .
\end{align*}
Executing simple algebraic recombinations, the preceding estimate becomes
\begin{align}\label{eq:monotonicity 4}
\frac{\partial}{\partial t}
\left(
    \operatorname{Re}h(\rho_\nu^+(t),t)-\gamma t
\right)
&>
\sum_{j=2}^{N-1}
\theta_j^+b_j^+l_{j,\nu}^+(t)
-
\sum_{j=1}^{M}
\theta_j^-b_j^-l_{j,\nu}^-(t) \notag\\
&\quad
-
\sum_{j=\nu+1}^{N-1}
(\theta_j^+-\theta_{\mu_\nu}^-)b_j^+
+
\sum_{j=1}^{\mu_\nu}
(\theta_j^--\theta_{\mu_\nu}^-)b_j^- 
+
b_N^+\theta_{\mu_\nu}^- .
\end{align}
Next, we modify the first two sums in
\eqref{eq:monotonicity 4} by rewriting the exponents $\theta_j^\pm$ in a different way. More specifically, we may write $\theta_j^\pm=\theta_{\mu_\nu}^-+(\theta_j^\pm-\theta_{\mu_\nu}^-)$. Thus, returning to \eqref{eq:monotonicity 4}, we infer that
\begin{align}\label{eq:monotonicity 5}
\notag\frac{\partial}{\partial t}
\left(
    \operatorname{Re}h(\rho_\nu^+(t),t)-\gamma t
\right)
&>
\theta_{\mu_\nu}^-
\left(
    \sum_{j=2}^{N-1}b_j^+l_{j,\nu}^+(t)
    -
    \sum_{j=1}^{M}b_j^-l_{j,\nu}^-(t)
\right)
+
\sum_{j=2}^{N-1}
(\theta_j^+-\theta_{\mu_\nu}^-)
b_j^+l_{j,\nu}^+(t)\\
&
-
\sum_{j=1}^{M}
(\theta_j^--\theta_{\mu_\nu}^-)
b_j^-l_{j,\nu}^-(t)
-
\sum_{j=\nu+1}^{N-1}
(\theta_j^+-\theta_{\mu_\nu}^-)b_j^+
+
\sum_{j=1}^{\mu_\nu}
(\theta_j^--\theta_{\mu_\nu}^-)b_j^-
+
b_N^+\theta_{\mu_\nu}^- .
\end{align}
Besides, via \eqref{sums for l}, we see that
\[
    \theta_{\mu_\nu}^-
    \left(
        \sum_{j=2}^{N-1}b_j^+l_{j,\nu}^+(t)
        -
        \sum_{j=1}^{M}b_j^-l_{j,\nu}^-(t)
    \right)=
    \theta_{\mu_\nu}^-(-b_1^+l_{1,\nu}^+(t)-b_N^+)=
    -\theta_{\mu_\nu}^-b_1^+l_{1,\nu}^+(t)
    -
    b_N^+\theta_{\mu_\nu}^-.
\]
As a result, applying on \eqref{eq:monotonicity 5} and breaking up some of the sums, we find
\begin{align}\label{eq:monotonicity 6}
\frac{\partial}{\partial t}
\left(
    \operatorname{Re}h(\rho_\nu^+(t),t)-\gamma t
\right)
&>
-\theta_{\mu_\nu}^-b_1^+l_{1,\nu}^+(t)+
\sum_{j=2}^{\nu}
(\theta_j^+-\theta_{\mu_\nu}^-)
b_j^+l_{j,\nu}^+(t) \notag\\
&\quad
+
\sum_{j=\nu+1}^{N-1}
(\theta_j^+-\theta_{\mu_\nu}^-)
b_j^+
\bigl(l_{j,\nu}^+(t)-1\bigr) 
-
\sum_{j=\mu_\nu+1}^{M}
(\theta_j^--\theta_{\mu_\nu}^-)
b_j^-
l_{j,\nu}^-(t) \notag\\
&\quad
+
\sum_{j=1}^{\mu_\nu}
(\theta_j^--\theta_{\mu_\nu}^-)
b_j^-
\bigl(1-l_{j,\nu}^-(t)\bigr).
\end{align}
We now inspect the sign of each term in \eqref{eq:monotonicity 6}. By \eqref{eq:orderings of l_n} (and
by \eqref{orderings of lN} in the special case \(\nu=N-1\)), we have
\[
    l_{1,\nu}^+(t)<0,\qquad
    l_{j,\nu}^+(t)<0 \quad (2\le j\le \nu),
\]
\[
    l_{j,\nu}^+(t)>1 \quad (j\ge\nu+1),
    \qquad
    0<l_{j,\nu}^-(t)<1 \quad (1\le j\le M).
\]
On the other hand, the ordering of the exponents gives
\[
    \theta_j^+\le \theta_\nu^+\le\theta_{\mu_\nu}^-,
    \qquad 2\le j\le\nu,
\]
\[
    \theta_j^+\ge \theta_{\nu+1}^+\ge\theta_{\mu_\nu}^-,
    \qquad j\ge\nu+1,
\]
and
\[
    \theta_j^-\ge\theta_{\mu_\nu}^- \quad (j\le\mu_\nu),
    \qquad
    \theta_j^-\le\theta_{\mu_\nu}^- \quad (j\ge\mu_\nu+1).
\]
Thus every term on the right-hand side of \eqref{eq:monotonicity 6} is non-negative, while the first one is strictly positive. Hence
\[
\frac{\partial}{\partial t}
\left(
    \operatorname{Re}h(\rho_\nu^+(t),t)-\gamma t
\right)>0,
\qquad t\ge0.
\]
This proves the desired monotonicity for
\(\nu\in\{2,\dots,N-1\}\). All that is left is to treat the endpoint case \(\nu=1\). The proof follows an almost identical procedure, so we will just provide a brief sketch for the sake of avoiding repetition. Writing \eqref{eq:monotonicity 2} again, but with \(\xi(t)=\rho_1^+(t)\), we get
\begin{equation*}
\frac{\partial}{\partial t}
\left(
    \operatorname{Re}h(\rho_1^+(t),t)-\gamma t
\right)
=
\sum_{j=2}^{N-1}
\theta_j^+b_j^+l_{j,1}^+(t)
-
\sum_{j=1}^{M}
\theta_j^-b_j^-l_{j,1}^-(t) 
-
\sum_{j=2}^{N-1}\theta_j^+b_j^+
+
\sum_{j=1}^{M}\theta_j^-b_j^-
-\gamma .
\end{equation*}
Once more, \(\gamma<r(\rho_1^+)\) and so we may use the rate in Proposition \ref{prop:rates of left tips} to get
\begin{align*}
\frac{\partial }{\partial t}
\left(
    \operatorname{Re}h(\rho_1^+(t),t)-\gamma t
\right)
&>
\sum_{j=2}^{N-1}
\theta_j^+b_j^+l_{j,1}^+(t)
-
\sum_{j=1}^{M}
\theta_j^-b_j^-l_{j,1}^-(t)
-
\sum_{j=2}^{N-1}\theta_j^+b_j^+\\
&\quad
+
\sum_{j=1}^{M}\theta_j^-b_j^-
-
\left(
    \sum_{j=1}^{\mu_1}b_j^-
    -
    \sum_{j=2}^{N}b_j^+
\right)\theta_{\mu_1}^-
-
\sum_{j=\mu_1+1}^{M}b_j^-\theta_j^- .
\end{align*}
A quick simplification leads to
\begin{align}\label{eq:monotonicity 7}
\frac{\partial}{\partial t}
\left(
    \operatorname{Re}h(\rho_1^+(t),t)-\gamma t
\right)
&>
\sum_{j=2}^{N-1}
\theta_j^+b_j^+l_{j,1}^+(t)
-
\sum_{j=1}^{M}
\theta_j^-b_j^-l_{j,1}^-(t) \notag\\
&\quad
-
\sum_{j=2}^{N-1}
(\theta_j^+-\theta_{\mu_1}^-)b_j^+
+
\sum_{j=1}^{\mu_1}
(\theta_j^--\theta_{\mu_1}^-)b_j^-+
b_N^+\theta_{\mu_1}^- .
\end{align}
Using the same trick as before, we may write $\theta_j^\pm=\theta_{\mu_1}^-+(\theta_j^\pm-\theta_{\mu_1}^-)$. Applying this on \eqref{eq:monotonicity 7} and using \eqref{sums for l} exactly as in the previous case, we have
\begin{align}\label{eq:monotonicity 8}
\frac{\partial}{\partial t}
\left(
    \operatorname{Re}h(\rho_1^+(t),t)-\gamma t
\right)
&>
-\theta_{\mu_1}^-b_1^+l_{1,1}^+(t)
+
\sum_{j=2}^{N-1}
(\theta_j^+-\theta_{\mu_1}^-)
b_j^+
\bigl(l_{j,1}^+(t)-1\bigr) \notag\\
&\quad
+
\sum_{j=\mu_1+1}^{M}
(\theta_j^--\theta_{\mu_1}^-)
b_j^-
\bigl(-l_{j,1}^-(t)\bigr) +
\sum_{j=1}^{\mu_1}
(\theta_j^--\theta_{\mu_1}^-)
b_j^-
\bigl(1-l_{j,1}^-(t)\bigr).
\end{align}
Analogous considerations on \eqref{eq:ordering for l_1} and the ordering of the exponents $\theta_j^\pm$ show every term on the right-hand side of
\eqref{eq:monotonicity 8} is non-negative, while the first one is strictly positive. Hence
\[
\frac{\partial}{\partial t}
\left(
    \operatorname{Re}h(\rho_1^+(t),t)-\gamma t
\right)>0, \quad t\ge0,
\]
and the desired monotonicity is proved.

(2) We turn our attention to the right tip points. Since the proof follows similar steps, we will omit some details of computational nature in order to avoid redundancy. We first consider an upper right slit, which by construction corresponds to an index $\mu\in\{\mu_1,\dots,M-1\}$. Using \eqref{eq:monotonicity 2} with $\rho_\mu^-(t)$ instead of $\xi(t)$ and the fact that \(\gamma>r(\rho_\mu^-)\), we get
\begin{equation*}
\frac{\partial}{\partial t}
\left(
    \operatorname{Re}h(\rho_\mu^-(t),t)-\gamma t
\right)
<
\sum_{j=2}^{N-1}
\theta_j^+b_j^+\alpha_{j,\mu}^+(t)
-
\sum_{j=1}^{M}
\theta_j^-b_j^-\alpha_{j,\mu}^-(t)
-
\sum_{j=2}^{N-1}\theta_j^+b_j^+
+
\sum_{j=1}^{M}\theta_j^-b_j^-
-
r(\rho_\mu^-).
\end{equation*}
Using the formula from Proposition~\ref{prop:rates of right tips} (1), writing $\theta_j^\pm=\theta_{\mu+1}^-+(\theta_j^\pm-\theta_{\mu+1}^-)$, and then applying \eqref{sums for a}, we obtain
\begin{align}\label{eq:monotonicity 9}
\frac{\partial}{\partial t}
\left(
    \operatorname{Re}h(\rho_\mu^-(t),t)-\gamma t
\right)
&<
\sum_{j=2}^{N-1}
(\theta_j^+-\theta_{\mu+1}^-)
b_j^+
\bigl(\alpha_{j,\mu}^+(t)-1\bigr) 
-
\theta_{\mu+1}^-b_1^+\alpha_{1,\mu}^+(t) \notag\\
&\quad
+
\sum_{j=1}^{\mu}
(\theta_j^--\theta_{\mu+1}^-)
b_j^-
\bigl(1-\alpha_{j,\mu}^-(t)\bigr)
-
\sum_{j=\mu+2}^{M}
(\theta_j^--\theta_{\mu+1}^-)
b_j^-\alpha_{j,\mu}^-(t).
\end{align}
By \eqref{orderings of a_m}, we know that
\[
    0<\alpha_{j,\mu}^+(t)<1 \quad (2\le j \le N-1), \qquad 
    \alpha_{j,\mu}^-(t)>1 \quad (j\le\mu),
    \qquad
    \alpha_{j,\mu}^-(t)<0 \quad (j\ge\mu+1).
\]
Moreover, for the upper slits, our configuration imposes
\[
    \theta_j^+\ge\theta_{\mu+1}^-,
    \quad (2\le j\le N-1), 
    \qquad
    \theta_j^-\ge\theta_{\mu+1}^- \quad (j\le\mu),
    \qquad
    \theta_j^-\le\theta_{\mu+1}^- \quad (j\ge\mu+2).
\]
Therefore, every term on the right-hand side of \eqref{eq:monotonicity 9} is non-positive, while $-\theta_{\mu+1}^-b_1^+\alpha_{1,\mu}^+(t)<0$. Ergo
\[
\frac{\partial}{\partial t}
\left(
    \operatorname{Re}h(\rho_\mu^-(t),t)-\gamma t
\right)<0, \quad t\ge0,
\]
and we get the monotonicity for every upper right slit. Next, we consider a middle right slit. Thus, for some appropriate \(\nu\), we have $\mu\in\{\mu_\nu,\dots,\mu_{\nu-1}-1\}$. Repeating the preceding computation, this time using the formula
from Proposition~\ref{prop:rates of right tips} (2), we write $\theta_j^\pm=\theta_\nu^++(\theta_j^\pm-\theta_\nu^+)$ and apply \eqref{sums for a} to obtain
\begin{align}\label{eq:monotonicity 10}
\frac{\partial}{\partial t}
\left(
    \operatorname{Re}h(\rho_\mu^-(t),t)-\gamma t
\right)
&<
\sum_{j=2}^{\nu}
(\theta_j^+-\theta_\nu^+)
b_j^+
\alpha_{j,\mu}^+(t)
+
\sum_{j=\nu+1}^{N-1}
(\theta_j^+-\theta_\nu^+)
b_j^+
\bigl(\alpha_{j,\mu}^+(t)-1\bigr)
-
\theta_\nu^+b_1^+\alpha_{1,\mu}^+(t) \notag\\
&\quad
+
\sum_{j=1}^{\mu}
(\theta_j^--\theta_\nu^+)
b_j^-
\bigl(1-\alpha_{j,\mu}^-(t)\bigr)
-
\sum_{j=\mu+1}^{M}
(\theta_j^--\theta_\nu^+)
b_j^-\alpha_{j,\mu}^-(t).
\end{align}
As in the previous case, \eqref{orderings of a_m} and the ordering of the exponents determine that every term on the right-hand side of \eqref{eq:monotonicity 10} is non-positive, while $-\theta_\nu^+b_1^+\alpha_{1,\mu}^+(t)<0$. Hence, the derivative is strictly negative for middle right slits. Finally, for a lower right slit, or equivalently for an index $\mu\in\{1,\dots,\mu_{N-1}-1\}$, successively, we use the formula from Proposition~\ref{prop:rates of right tips} (3), we write $\theta_j^\pm=\theta_\mu^-+(\theta_j^\pm-\theta_\mu^-)$ and we use \eqref{sums for a}, to deduce
\begin{align}\label{eq:monotonicity 11}
\frac{\partial}{\partial t}
\left(
    \operatorname{Re}h(\rho_\mu^-(t),t)-\gamma t
\right)
&<
\sum_{j=2}^{N-1}
(\theta_j^+-\theta_\mu^-)
b_j^+
\alpha_{j,\mu}^+(t)
-
\theta_\mu^-b_1^+\alpha_{1,\mu}^+(t) \notag\\
&\quad
+
\sum_{j=1}^{\mu}
(\theta_j^--\theta_\mu^-)
b_j^-
\bigl(1-\alpha_{j,\mu}^-(t)\bigr)
-
\sum_{j=\mu+1}^{M}
(\theta_j^--\theta_\mu^-)
b_j^-\alpha_{j,\mu}^-(t).
\end{align}
Again, the ordering of the exponents together with \eqref{orderings of a_m} show that every term on the right-hand side of \eqref{eq:monotonicity 11} is non-positive, while $-\theta_\mu^-b_1^+\alpha_{1,\mu}^+(t)<0$. Thus, the derivative is strictly negative for the lower right slits too. All that remains is to treat the extremal case \(\mu=M\). Since
\(\rho_M^-(t)>k_M^-(t)\), all the quantities $\rho_M^-(t)-k_j^+(t),\rho_M^-(t)-k_j^-(t)$, are positive, for all $t\ge0$. Moreover, by definition, $(k_j^+)'(t)>0$, whereas, $(k_j^-)'(t)<0$, for all $t\ge0$. Therefore, returning to \eqref{eq:monotonicity 1} and setting $\xi(t)=\rho_\mu^-(t)$, we find
\[
\frac{\partial h}{\partial t}(\rho_M^-(t),t)<0.
\]
Hence, by the chain rule and the fact that $\gamma>0$,
\[
\frac{\partial}{\partial t}
\left(
    \operatorname{Re}h(\rho_M^-(t),t)-\gamma t
\right)
=
\frac{\partial h}{\partial t}(\rho_M^-(t),t)-\gamma<0,
\]
and we reach the desired conclusion.

Leaving the monotonicity aside, we also need to verify the limiting behavior. 

(1) Fix $\nu\in\{1,\dots,N-1\}$. By the very definition of the rates, $\lim_{t\to+\infty}(\frac{1}{t}\mathrm{Re}(h(\rho_\nu^+(t),t))=r(\rho_\nu^+)$, while $\gamma<r(\rho_\nu^+)$ by construction. It obviously follows that
\begin{equation*}
    \lim_{t\to+\infty}(\mathrm{Re}h(\rho_\nu^+(t),t)-\gamma t)=+\infty.
\end{equation*}

(2) Fix $\mu\in\{1,\dots,M\}$. In similar fashion, we have $\lim_{t\to+\infty}(\frac{1}{t}\mathrm{Re}h(\rho_\mu^-(t),t))=r(\rho_\mu^-)$, whereas $\gamma>r(\rho_\mu^-)$. This directly implies
\begin{equation*}
     \lim_{t\to+\infty}(\mathrm{Re}h(\rho_\mu^-(t),t)-\gamma t)=-\infty,
\end{equation*}
and the proof is complete.
\end{proof}

Lemma \ref{monotonicity} is of utmost importance because it will allow us to eventually construct the Loewner chain. In particular, we must have that the chain of the Riemann maps forms a decreasing family of domains $(\Omega_t)_{t\ge0}$. Indeed, writing $\Omega_t=h(\mathbb{H},t)-\gamma t$, the preceding lemma implies that $\Omega_t\subset\Omega_s$, for all $s<t$, as seen in Figures \ref{fig: orbits on Omega 1} and \ref{fig: orbits on Omega 2}.

\begin{figure}[ht]
\centering
\resizebox{0.8\linewidth}{!}{%
\begin{tikzpicture}[
  x=0.02cm,
  y=-0.02cm,
  every node/.style={font=\scriptsize, inner sep=0.6pt}
]

\definecolor{framegray}{RGB}{95,95,95}
\definecolor{rayblack}{RGB}{35,35,35}      
\definecolor{timeviolet}{RGB}{210,180,210} 
\definecolor{timeblue}{RGB}{180,190,235}   

%
%
%

\draw[rayblack, line width=0.5pt] (22,14) -- (520,14);
\draw[rayblack, line width=0.5pt] (22,250) -- (520,250);


\draw[rayblack,   line width=0.6pt] (340,62) -- (520,62);
\draw[timeviolet, line width=0.6pt] (200,62)--(340,62);

\draw[rayblack, line width=0.6pt] (22,94) -- (245,94);
\draw[timeblue, line width=0.7pt] (245,94) -- (400,94);

\draw[timeviolet, line width=0.7pt] (215,125) -- (392,125);
\draw[rayblack,   line width=0.6pt] (392,125) -- (520,125);



\draw[timeviolet, line width=0.7pt] (156,178) -- (250,178);
\draw[rayblack,   line width=0.6pt] (250,178) -- (520,178);

\draw[rayblack, line width=0.6pt] (22,210) -- (160,210);
\draw[timeblue, line width=0.7pt] (160,210) -- (300,210);


\node at (50,46)  {$\vdots$};
\node at (460,34) {$\vdots$};

\node at (468,230) {$\vdots$};
\node at (78,230)  {$\vdots$};
\node at (468,145) {$\vdots$};

\node[anchor=west] at (12,155) {$\infty_{\nu+1}^{+}$};

\node[anchor=west] at (480,93) {$\infty_{\mu_{\nu}}^{-}$};
\node[anchor=west] at (460,197) {$\infty_{\mu_{\nu+1}}^{-}$};

\node at (392,45)  {$h(\rho_{\mu_{\nu}}^{-})$};
\node at (200,45)  {$h_t(\rho_{\mu_{\nu}}^{-}(t))-\gamma t$};
\node at (255,80)  {$h(\rho_{\nu}^{+})$};
\node at (400,80)  {$h_t(\rho_{\nu}^{+}(t))-\gamma t$};
\node at (400,140){$h(\rho_{\mu_{\nu}-1}^{-})$};
\node at (200,140){$h_t(\rho_{\mu_{\nu}-1}^{-}(t))-\gamma t$};
\node at (275,188) {$h(\rho_{\mu_{\nu+1}}^{-})$};
\node at (140,188) {$h_t(\rho_{\mu_{\nu+1}}^{-}(t))-\gamma t$};
\node at (165,230) {$h(\rho_{\nu+1}^{+})$};
\node at (350,230) {$h_t(\rho_{\nu+1}^{+}(t))-\gamma t$};
\node at (40,25) {\large$ \Omega_t$};
\def\TipDot#1#2{%
  \fill[black] (#1,#2) circle[radius=1pt];
}
\TipDot{200}{62}
\TipDot{340}{62}

\TipDot{245}{94}
\TipDot{400}{94}

\TipDot{215}{125}
\TipDot{392}{125}

\TipDot{156}{178}
\TipDot{250}{178}

\TipDot{160}{210}
\TipDot{300}{210}

\end{tikzpicture}%
}
\caption{Orbits of the tip points on $\Omega$, through $h_t-\gamma t$, with respect to time $t$. For every index $\nu$, the left half-lines corresponding to each index $\mu\in\{\mn,\dots,\mu_{\nu-1}-1\}$ access $\infty_{\nu}^+$. The tip points of $\partial\Omega_t$ are given by $h_t(\rho_j^{\pm}(t))-\gamma t$, that extend the initial tip points $h(\rho_j^{\pm})$ towards the points at infinity they access.}
\label{fig: orbits on Omega 1}
\end{figure}
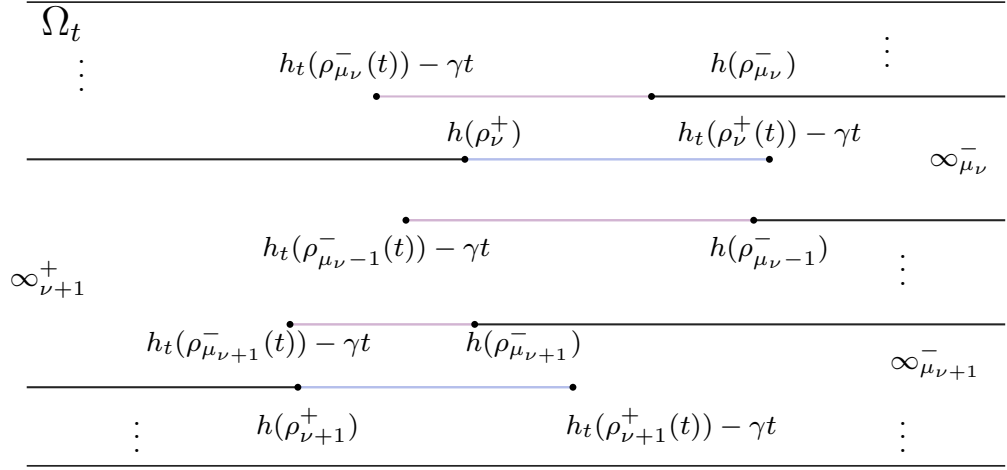

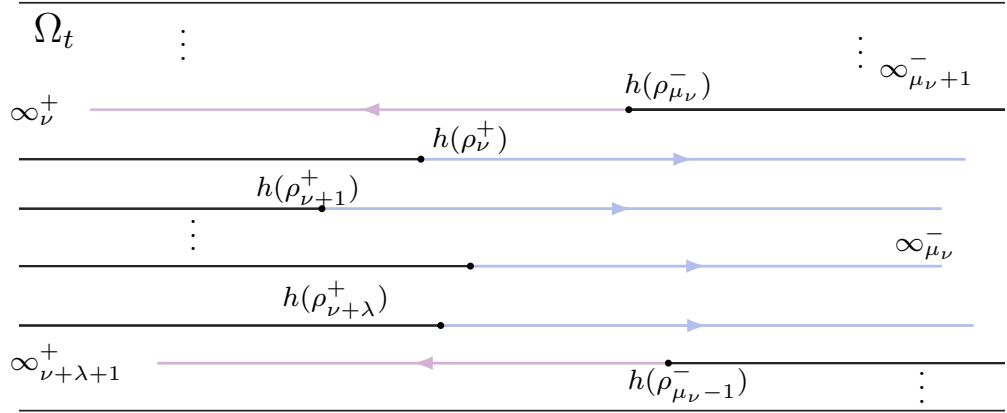
\begin{figure}[ht]
\centering
\resizebox{0.8\linewidth}{!}{%
\begin{tikzpicture}[
  x=0.02cm,
  y=-0.02cm,
  every node/.style={font=\scriptsize, inner sep=0.6pt}
]

\definecolor{framegray}{RGB}{95,95,95}
\definecolor{rayblack}{RGB}{35,35,35}      
\definecolor{timeviolet}{RGB}{210,180,210} 
\definecolor{timeblue}{RGB}{180,190,235}   

%
%
%


\draw[rayblack, line width=0.5pt] (22,14) -- (520,14);
\draw[rayblack, line width=0.5pt] (22,220) -- (520,220);


\draw[timeviolet, line width=0.6pt] (58,68) -- (330,68);
\draw[rayblack, line width=0.7pt] (330,68) -- (520,68);

\draw[rayblack, line width=0.6pt] (22,93) -- (225,93);
\draw[timeblue, line width=0.7pt] (225,93) -- (500,93);

\draw[rayblack, line width=0.6pt] (22,118) -- (175,118);
\draw[timeblue, line width=0.7pt] (175,118) -- (488,118);

\draw[rayblack, line width=0.6pt] (22,147) -- (250,147);
\draw[timeblue,line width=0.7pt] (250,147)--(488,147);

\draw[rayblack, line width=0.6pt] (22,177) -- (235,177);
\draw[timeblue, line width=0.7pt] (235,177) -- (504,177);


\draw[timeviolet, line width=0.7pt] (92,196) -- (350,196);
\draw[rayblack, line width=0.6pt] (350,196) -- (520,196);

\node at (105,30) {$\vdots$};
\node at (447,34) {$\vdots$};
\node at (110,126) {$\vdots$};
\node at (478,203) {$\vdots$};

\node[anchor=west] at (16,68) {$\infty_{\nu}^{+}$};
\node[anchor=west] at (16,196) {$\infty_{\nu+\lambda+1}^{+}$};

\node[anchor=west] at (455,48) {$\infty_{\mu_\nu+1}^{-}$};
\node[anchor=west] at (463,135) {$\infty_{\mu_\nu}^{-}$};

\node at (350,55) {$h(\rho_{{\mu}_{\nu}}^-)$};
\node at (250,82) {$h(\rho_{{\nu}}^{+})$};
\node at (168,108) {$h(\rho_{{\nu}+1}^{+})$};
\node at (182,164) {$h(\rho_{{\nu}+\lambda}^{+})$};
\node at (360,206) {$h(\rho_{{\mu}_{\nu}-1}^-)$};
\tikzset{
  midarrowright/.style={
    postaction={
      decorate,
      decoration={
        markings,
        mark=at position 0.5 with {
          \arrow{Latex[length=2mm,width=1.4mm]}
        }
      }
    }
  },
  midarrowleft/.style={
    postaction={
      decorate,
      decoration={
        markings,
        mark=at position 0.5 with {
          \arrowreversed{Latex[length=2mm,width=1.4mm]}
        }
      }
    }
  }
}

\draw[timeviolet, line width=0.6pt, midarrowleft] (58,68) -- (330,68);
\draw[rayblack, line width=0.7pt] (330,68) -- (520,68);

\draw[rayblack, line width=0.6pt] (22,93) -- (225,93);
\draw[timeblue, line width=0.7pt, midarrowright] (225,93) -- (500,93);

\draw[rayblack, line width=0.6pt] (22,118) -- (175,118);
\draw[timeblue, line width=0.7pt, midarrowright] (175,118) -- (488,118);

\draw[rayblack, line width=0.6pt] (22,147) -- (250,147);
\draw[timeblue, line width=0.7pt, midarrowright] (250,147) -- (488,147);

\draw[rayblack, line width=0.6pt] (22,177) -- (235,177);
\draw[timeblue, line width=0.7pt, midarrowright] (235,177) -- (504,177);

\draw[timeviolet, line width=0.7pt, midarrowleft] (92,196) -- (350,196);
\draw[rayblack, line width=0.6pt] (350,196) -- (520,196);
\def\TipDot#1#2{%
  \fill[black] (#1,#2) circle[radius=1pt];
}

\TipDot{330}{68}

\TipDot{225}{93}

\TipDot{175}{118}

\TipDot{250}{147}

\TipDot{235}{177}

\TipDot{350}{196}

\node at (40,28) {\large${\Omega_t}$};
\end{tikzpicture}%
}
\caption{Orbits of the tip points on $\Omega$, through $h_t-\gamma t$, with respect to time $t\to +\infty$. For every index $\nu$, the right half-line corresponding $\nu$ accesses $\infty_{\mn}^-$.}
\label{fig: orbits on Omega 2}
\end{figure}

\section{Construction of the Loewner chain and the driving function}\label{sec:flow}
 We may now proceed to the essence of the present article. In this section, we are going to formally construct our Loewner chain based on the computations and considerations realized in the preceding section. 
 
 For this reason, we priorly summarize what we have worked on thus far. We start with two arbitrary numbers $N,M\in\mathbb{N}$ and a choice of positive numbers $b_1^+,\dots,b_N^+$, $b_1^-,\dots,b_M^-$ satisfying
 $\sum_{j=1}^Nb_j^+>\sum_{j=1}^Mb_j^-$. We denote by $b^+,b^-$, the sums $\sum_{j=1}^Nb_j^+,$$\sum_{j=1}^Mb_j^-$ respectively. Also, we consider real numbers
 \begin{equation*}
  k_1^+<\dots k_{N-1}^+<k_N^+<k_1^-<\dots<k_M^-    
 \end{equation*}
which pose as our initial choice. Then, we pick exponents $\theta_2^+\le\dots\le\theta_{N-1}^+$ and $\theta_1^-\ge\dots\ge\theta_M^-$ and define the functions $k_j^{\pm}(t)=k_N^++(k_j^+-k_N^+)e^{-\theta_j^{\pm}t}$ for all possible indices, while setting $k_1^+(t)\equiv k_1^+$, $k_N^+(t)\equiv k_N^+,$ so that the ordering
\begin{equation*}
 k_1^+(t)<\dots<k_N^+(t)<k_1^-(t)<\dots<k_M^-(t)   
\end{equation*}
is maintained, for all $t\ge0$. By definition, $\lim_{t\to+\infty}k_j^{\pm}(t)=k_N^+$, for all the above functions, excluding $k_1^+(t)$. With this configuration, we consider the function $h\vcentcolon\mathbb{H}\times[0,+\infty)\to\mathbb{C}$ with
 \begin{equation}
     h(z,t)=\sum_{j=1}^Nb_j^+\log(z-k_j^+(t))-\sum_{j=1}^Mb_j^-\log(z-k_j^-(t))
 \end{equation}
Denoting by $\rho_\nu^+(t)\in(k_\nu^+(t),k_{\nu+1}^+(t))$, for $\nu=1,\dots,N-1$, by $\rho_\mu^-(t)\in(k_\mu^-(t),k_{\mu+1}^-(t))$, for $\mu=1,\dots,M-1$, and by $\rho_M^-(t)>k_M^-(t)$ the roots of $h'(\cdot,t)$, we then calculate the numbers 
\begin{equation}
    r_{\nu}^+:=r(\rho_{\nu}^+)=\lim_{t\rightarrow+\infty}\frac{1}{t}\mathrm{Re}(h(\rho^+_{\nu}(t),t))\quad\text{and}\quad r_{\mu}^-:=r(\rho_{\mu}^-)=\lim_{t\rightarrow+\infty}\frac{1}{t}\mathrm{Re}(h(\rho^-_{\mu}(t),t))
\end{equation}
 for all $\nu=1,\dots,N-1$ and $\mu=1,\dots,M-1$. By Corollary \ref{existance of thetas}, it is possible to proceed to a choice on the exponents $\theta_j^{\pm}$, so that the conditions of Lemmata \ref{min rates} and \ref{max rates} can hold simultaneously. 

Then, Lemma \ref{monotonicity} implies that if we pick any number $\gamma\in(\max r_{\mu}^-,\min r_{\nu}^+)$, the family of functions $h(\cdot,t)-\gamma t$, $t\ge0$, produces a strictly decreasing family of domains $\Omega_t$. Thus, if $\Omega_t:=h(\mathbb{H},t)-\gamma t$, then $\Omega_s\subset\Omega_t$ for all $s>t$. Writing $h$ for $h(\cdot,0)$ and remembering that $h$ is univalent, we then define 
\begin{equation*}
f(z,t)\vcentcolon=h^{-1}(h(z,t)-\gamma t),   
\end{equation*}
for all $z\in\mathbb{H}$ and $t\ge0$. Setting $H_t\vcentcolon=h^{-1}(\Omega_t)$, we understand that $f(\cdot,t)$ maps the upper half-plane $\mathbb{H}$ conformally onto $H_t$ for all $t\ge0$. Specifically, if we define the line segments
 \begin{equation}\label{line segments}
      L_{\nu}^+(t)\vcentcolon=[h(\rho_{\nu}^+(0)),h(\rho_{\nu}^+(t),t)-\gamma t]\quad\text{and}\quad L_{\mu}^-(t)\vcentcolon=[h(\rho_{\mu}^-(t),t)-\gamma t,h(\rho_{\mu}^-(0))],
  \end{equation}
for all $\nu=1.\dots,N-1$, and all $\mu=1,\dots,M$, then 
 \begin{equation}\label{H_t}
  H_t=\mathbb{H}\setminus\left[\bigcup_{\nu=1}^{N-1}\Gamma_{\nu}^+(t)\cup\bigcup_{\mu=1}^M\Gamma^-_{\mu}(t)\right],
  \end{equation}
where $\Gamma_j^{\pm}(t):=h^{-1}(L_j^{\pm}(t))$, for all $t\ge0$.

By construction, for each $t\ge0$, $f(\cdot,t)$ maps infinity to infinity and has single pole there. As a consequence, the Laurent expansion of $f(\cdot,t)$ at infinity is 
\begin{equation}\label{flow}
f(z,t)=a_tz+b_t+\sum_{n=1}^{+\infty}\frac{c_n(t)}{z^n},
\end{equation}
for some $a_t>0$ and $b_t\in\mathbb{R}$ since $f(\cdot,t)$ maps the upper half-plane into the upper half-plane. Note that the hydrodynamic condition, as mentioned in the Introduction, demands that $a_t=1$ and $b_t=0$, for all $t\ge0$. This is evidently true for $t=0$, since $f(\cdot,0)=\mathrm{id}_{\mathbb{H}}$. Thus, in order to consider the Loewner chain, we need to determine the functions $a_t,b_t$ and set the Loewner chain to be $\tilde{f}(z,t)=f(\frac{z-b_t}{a_t},t)$. Note that the geometry of the flow is not changed by this normalization. We need it, however, as far as the driving function is concerned. Throughout the rest of this work, we will use the notation $f^{\circ}:=\partial_t f$ for the derivative with respect to time $t$. Differentiating \eqref{flow} with respect to $t$, we get
\begin{equation}\label{flow_t}
    f^{\circ}(z,t)=a_t^{\circ}z+b_t^{\circ}+\sum_{n=1}^{+\infty}\frac{c_n^{\circ}(t)}{z^n},
\end{equation}
for all $z\in\mathbb{H}$ and $t\ge0$. 

\begin{lemma}\label{lm:coefficients}
Given the configuration described above:
    \begin{itemize}
        \item[\textup{(1)}] $a_t=\exp(-\frac{\gamma}{b^+-b^-}t)$, for all $t\ge0$;
        \item[\textup{(2)}] $b_t=\int_0^ta_\sigma(A(\sigma)+\sum_{\nu=1}^{N-1}B_{\nu}^+(\sigma)+\sum_{\mu=1}^{M}B_{\mu}^-(\sigma))d\sigma$, for all $t\ge0$, where
    \end{itemize}
    \begin{equation*}
     A(t)=\frac{h_t^{\circ}(0)-\gamma}{h'_t(0)}\quad\text{and}\quad B_{j}^{\pm}(t)=\frac{h_t^{\circ}(\rho_j^{\pm}(t))-\gamma}{\rho_j^{\pm}(t)h''_t(\rho_j^{\pm}(t))}, \quad \text{for all }t\ge0.   
    \end{equation*}
\end{lemma}
\begin{proof}
    (1) We denote by $\Pi_t(z)$ the principal part of $f(z,t)$ in \eqref{flow} for the sake simplicity. By the definition of $f$, we have the $h(f(z,t))=h_t(z)-\gamma t$ and therefore, differentiating with respect to $z$ and $t$, and keeping in mind that $h$ is univalent, we get
    \begin{equation*}
     \frac{f^\circ(z,t)}{zf'(z,t)}=\frac{h'(f(z,t))f^{\circ}(z,t)}{h'(f(z,t))zf'(z,t)}=\frac{(h\circ f)^{\circ}(z,t)}{z(h\circ f)'(z,t)}=\frac{h^{\circ}_t(z)-\gamma}{zh'_t(z)}=:R_t(z), 
    \end{equation*}
    where $R_t$ is a rational function, as we see by the formula of $h_t$ in \eqref{h_t(N,M)}. Taking derivatives in \eqref{h_t(N,M)}, we see that
    \begin{equation}\label{eq:hydrodynamic 1}
      zh_t'(z)=z\left(\sum_{j=1}^{N}\frac{b_j^+}{z-k_j^+(t)}-\sum_{j=1}^{M}\frac{b_j^-}{z-k_j^-(t)}\right),
    \end{equation}
    whereas
    \begin{equation}\label{eq:hydrodynamic 2}
     h_t^\circ (z)=-\sum_{j=2}^{N-1}\frac{b_j^+ (k_j^+)^\circ(t)}{z-k_j^+(t)}+\sum_{j=1}^{M}\frac{b_j^- (k_j^-)^\circ (t)}{z-k_j^-(t)},
    \end{equation}
    where we took under consideration the facts that $k_1^+,k_N^+$ do not depend on $t$. It is easy to see through \eqref{eq:hydrodynamic 1} and \eqref{eq:hydrodynamic 2} that $\lim_{z\to\infty}(zh_t'(z))=b^+-b^-$ and $\lim_{z\to\infty}h_t^\circ(z)=0$. As a result, $R_t$ is indeed a rational function and $\lim_{z\to\infty}R_t(z)=-\frac{\gamma}{b^+-b^-}$. On the other hand, a combination of \eqref{flow} and \eqref{flow_t} provides
    \begin{equation*}
    \lim_{z\to\infty}R_t(z)=\lim_{z\to\infty}\frac{f^{\circ}(z,t)}{zf'(z,t)}=\lim_{z\to\infty}\frac{a_t^{\circ}z+b_t^{\circ}+\Pi_t^{\circ}(z)}{a_tz+z\Pi_t'(z)}=\frac{a_t^{\circ}}{a_t}.    
    \end{equation*}
    Thus, $a_t^{\circ}=-a_t\frac{\gamma}{b^+-b^-}$ and the result follows by solving the ordinary differential equation in the $t$-variable, since $a_0=1$.

(2) We now turn to $b_t$. All the preceding computations yield
    \begin{equation*}
    a_t^{\circ}+\frac{b_t^{\circ}}{z}+\frac{L_t^{\circ}(z)}{z}=\frac{f^{\circ}(z,t)}{z}=f'(z,t)R_t(z)=a_tR_t(z)+\Pi_t'(z)R_t(z).    
    \end{equation*}
Let \(r>0\) be sufficiently large so that the circle \(\partial D(0,r)\) lies in the annulus of convergence of the Laurent expansions at infinity
and encloses all finite poles of \(R_t\). Integrating the above identity over
\(\partial D(0,r)\), we get
\[
\frac{1}{2\pi i}
\int_{\partial D(0,r)}
\left(
a_t^\circ+\frac{b_t^\circ}{z}+\frac{\Pi_t^\circ(z)}{z}
\right)\,dz
=
\frac{1}{2\pi i}
\int_{\partial D(0,r)}
\left(
a_tR_t(z)+\Pi_t'(z)R_t(z)
\right)\,dz.
\]
Standard computations show that 
\begin{equation}\label{eq:hydrodynamic 3}
  b_t^\circ=a_t\frac{1}{2\pi i}\int_{\partial D(0,r)}R_t(z)\,dz.  
\end{equation}
All that remains is to compute this integral. Since
\[
R_t(z)=\frac{h_t^\circ(z)-\gamma}{z h_t'(z)},
\]
the possible finite poles of \(R_t\) are \(0\) and the zeros of \(h_t'\),
namely $\rho_\nu^+(t)$, $\nu=1,\dots,N-1$, and $\rho_\mu^-(t)$, $\mu=1,\dots,M$. The possible singularities at the points \(k_j^\pm(t)\) are removable, since
both \(h_t^\circ\) and \(h_t'\) have simple poles there. Thus \(R_t\) admits
a partial fraction decomposition of the form
\begin{equation}\label{eq:hydrodynamic partial decomposition}
  R_t(z)=C(t)+\frac{A(t)}{z}+\sum_{\nu=1}^{N-1}\frac{B_\nu^+(t)}{z-\rho_\nu^+(t)}+\sum_{\mu=1}^{M}\frac{B_\mu^-(t)}{z-\rho_\mu^-(t)}.  
\end{equation}
The coefficients \(A(t)\) and \(B_j^\pm(t)\) are the corresponding residues.
At \(z=0\), we have
\[
A(t)
=
\lim_{z\to0}(zR_t(z))
=
\frac{h_t^\circ(0)-\gamma}{h_t'(0)}.
\]
Similarly, if \(\rho=\rho_j^\pm(t)\) is a zero of \(h_t'\), then
\[
h_t'(z)=h_t''(\rho)(z-\rho)+O\bigl((z-\rho)^2\bigr),
\qquad \text{as }z\to\rho.
\]
Hence
\[
\lim_{z\to\rho}
\left[(z-\rho)
\frac{h_t^\circ(z)-\gamma}{z h_t'(z)}\right]
=
\frac{h_t^\circ(\rho)-\gamma}{\rho h_t''(\rho)}.
\]
Therefore
\[
B_j^\pm(t)
=
\frac{h_t^\circ(\rho_j^\pm(t))-\gamma}
{\rho_j^\pm(t)h_t''(\rho_j^\pm(t))}.
\]
Given the fact that $D(0,r)$ contains all finite poles of \(R_t\), the residue theorem gives
\[
\frac{1}{2\pi i}
\int_{\partial D(0,r)}R_t(z)\,dz
=
A(t)
+
\sum_{\nu=1}^{N-1}B_\nu^+(t)
+
\sum_{\mu=1}^{M}B_\mu^-(t).
\]
Consequently, returning back to \eqref{eq:hydrodynamic 3},
\[
b_t^\circ
=
a_t
\left(
A(t)
+
\sum_{\nu=1}^{N-1}B_\nu^+(t)
+
\sum_{\mu=1}^{M}B_\mu^-(t)
\right).
\]
Since \(b_0=0\), integration with respect to \(t\) yields
\[
b_t
=
\int_0^t
a_\sigma
\left(
A(\sigma)
+
\sum_{\nu=1}^{N-1}B_\nu^+(\sigma)
+
\sum_{\mu=1}^{M}B_\mu^-(\sigma)
\right)\,d\sigma .
\]
\end{proof}

To conclude our investigation, we turn our attention to the driving function of $\tilde{f}$. Through the definition of $R_t$ and \eqref{eq:hydrodynamic partial decomposition}, we have proved that 
\begin{align}\label{eq:driving function 0}
\notag\frac{f^{\circ}(z,t)}{f'(z,t)}&=zR_t(z)=A(t)+\sum_{\nu=1}^{N-1}\frac{B^+_\nu(t)z}{z-\rho_{\nu}^+(t)}+\sum_{\mu=1}^{M}\frac{B^-_\mu(t)z}{z-\rho_{\mu}^-(t)}+C(t)z\\
&=A(t)+\sum_{\nu=1}^{N-1}B^+_\nu(t)+\sum_{\mu=1}^MB^-_\mu(t)+\sum_{\nu=1}^{N-1}\frac{B^+_\nu(t)\rho_{\nu}^+(t)}{z-\rho_{\nu}^+(t)}+\sum_{\mu=1}^{M}\frac{B^-_\mu(t)\rho_{\mu}^-(t)}{z-\rho_{\mu}^-(t)}+C(t)z.
\end{align}
But in the previous proof we saw that $\lim_{z\to\infty}R_t(z)=\lim_{z\to\infty}\frac{f^\circ(z,t)}{zf'(z,t)}=\frac{a_t^\circ}{a_t}$. As a consequence, it is easy to verify that $C(t)=\frac{a_t^{\circ}}{a_t}$. Recall that our Loewner chain $(\tilde{f}(\cdot,t))_{t\ge0}$ is defined through the normalization
\begin{equation}\label{eq:driving function 1}
    \tilde{f}(z,t)=f\left(\frac{z-b_t}{a_t},t\right), \quad z\in\mathbb{H},\,t\ge0,
\end{equation}
and its driving function $Q$ is given by the formula
\begin{equation}\label{eq:driving function 2}
    Q(z,t)=-\frac{\tilde{f}^\circ(z,t)}{\tilde{f}'(z,t)}, \quad z\in\mathbb{H}, \, t\ge0.
\end{equation}
For the sake of readability, we proceed to the change of variables $y=\frac{z-b_t}{a_t}$. Then, via the previous proof,
\begin{align}\label{eq:driving function 3}
\notag    -Q(z,t)&=\frac{\tilde{f}^\circ(z,t)}{\tilde{f}'(z,t)}=\frac{f'(y,t)\left(-\frac{b_t^\circ}{a_t}-\frac{a_t^\circ}{a_t}y\right)+f^\circ(y,t)}{\frac{1}{a_t}f'(y,t)}\\
\notag    &=-b_t^\circ-a_t^\circ y+a_t\frac{f^\circ(y,t)}{f'(y,t)}\\
    &=-b_t^\circ-a_t^\circ y+a_t y R_t(y).
\end{align}
However, by Lemma \ref{lm:coefficients} (2), we infer that
\begin{equation}\label{eq:driving function 4}
    b_t^\circ=a(t)\left(A(t)+\sum_{\nu=1}^{N-1}B_\nu^+(t)+\sum_{\mu=1}^{M}B_\mu^-(t)\right).
\end{equation}
So, combining the fact that $C(t)=\frac{a_t^\circ}{a_t}$ with relations \eqref{eq:driving function 0}, \eqref{eq:driving function 3}, \eqref{eq:driving function 4}, and executing simple algebraic computations leads to
\[
Q(z,t)=-\frac{\tilde f^\circ(z,t)}{\tilde f'(z,t)}
=
-a_t\sum_{\nu=1}^{N-1}
\frac{B_\nu^+(t)\rho_\nu^+(t)}{y-\rho_\nu^+(t)}
-
a_t\sum_{\mu=1}^{M}
\frac{B_\mu^-(t)\rho_\mu^-(t)}{y-\rho_\mu^-(t)}.
\]
Since \(y=(z-b_t)/a_t\), we have
\[
y-\rho_j^\pm(t)
=
\frac{z-b_t-a_t\rho_j^\pm(t)}{a_t}.
\]
Thus
\[
\frac{\tilde f^\circ(z,t)}{\tilde f'(z,t)}
=
\sum_{\nu=1}^{N-1}
\frac{a_t^2B_\nu^+(t)\rho_\nu^+(t)}
{z-b_t-a_t\rho_\nu^+(t)}
+
\sum_{\mu=1}^{M}
\frac{a_t^2B_\mu^-(t)\rho_\mu^-(t)}
{z-b_t-a_t\rho_\mu^-(t)},
\]
and renaming $\tilde{\rho}_j^\pm(t)=b_t+a_t\rho_j^\pm(t)$, we obtain
\begin{equation*}
    Q(z,t)=-\sum_{\nu=1}^{N-1}\frac{a_t^2B_\nu^+(t)\rho_\nu^+(t)}{z-\tilde{\rho}_\nu^+(t)}-\sum_{\mu=1}^M\frac{a_t^2B_\mu^-(t)\rho_\mu^-(t)}{z-\tilde{\rho}_\mu^-(t)}.
\end{equation*}
Taking into account the formulas for the quantities $B_j^\pm$ resulting from Lemma \ref{lm:coefficients} (2), we see that the driving function of the Loewner chain we constructed is written as
\begin{equation}\label{driving function}
    Q(z,t)=-\sum_{\nu=1}^{N-1}\frac{\beta_{\nu}^+(t)}{z-\tilde{\rho}_{\nu}^+(t)}-\sum_{\mu=1}^{M}\frac{\beta_{\mu}^-(t)}{z-\tilde{\rho}_{\mu}^-(t)},
\end{equation}
where
\begin{equation}\label{h-capacity}
    \beta_{\nu}^+(t):=a_t^2\frac{h_t^{\circ}(\rho_{\nu}^{+}(t))-\gamma}{h''_t(\rho_{\nu}^{+}(t))}\quad\text{and}\quad\beta_{\mu}^-(t):=a_t^2\frac{h_t^{\circ}(\rho_{\mu}^{-}(t))-\gamma}{h''_t(\rho_{\mu}^{-}(t))}.
\end{equation}

We write down all the above in a form of a theorem below.

\begin{theorem}\label{Loenwer Chain}
    With the above configuration, we have that the solution to the PDE
    $$\frac{\partial f}{\partial t}(z,t)=-f'(z,t)Q(z,t), \quad f(z,0)=z$$
    for all $z\in\mathbb{H}$, $t\ge0$, where $Q$ is given by \eqref{driving function}, is a multi-slit Loewner chain with range $H_t$ as given by \eqref{H_t}.

    The trajectories (orbits) of the tip points are the continuous curves 
    \begin{equation*}
     \Gamma_{\nu}^+\vcentcolon[0,+\infty)\to\mathbb{H} \text{ with }\Gamma_\nu^+(t)=h^{-1}(h_t(\rho_{\nu}^+(t))-\gamma t), \quad\nu=1,\dots,N-1,   
    \end{equation*}
    and
    \begin{equation*}
     \Gamma_{\mu}^-\vcentcolon[0,+\infty)\to\mathbb{H} \text{ with } \Gamma_\mu^-(t)=h^{-1}(h_t(\rho_{\mu}^-(t))-\gamma t), \quad\mu=1,\dots,M.
    \end{equation*}
In particular, for every $\nu=1,\dots,N-1,$ the orbit $\Gamma_{\nu}^+$ emanates from $\rho_{\nu}^+=\rho_{\nu}^+(0)\in\mathbb{R}$ and converges to the boundary point $k_{\mn}^-=k_{\mn}^-(0)\in\mathbb{R}$. Also, for every $\mu=1,\dots,M$, $\Gamma_{\mu}^-$ emanates from $\rho_{\mu}^-=\rho_{\mu}^-(0)\in\mathbb{R}$ and converges to the boundary point $k_{j}^+=k_{j}^+(0)\in\mathbb{R}$ if and only if $h(\rho_{\mu}^-)\in S_j^+.$

\end{theorem}
\begin{figure}[ht]
\centering
\tikzset{
  midarrowright/.style={
    postaction={
      decorate,
      decoration={
        markings,
        mark=at position 0.5 with {
          \arrow{Latex[length=2mm,width=1.4mm]}
        }
      }
    }
  },
  midarrowleft/.style={
    postaction={
      decorate,
      decoration={
        markings,
        mark=at position 0.5 with {
          \arrowreversed{Latex[length=2mm,width=1.4mm]}
        }
      }
    }
  }
}

\begin{subfigure}[t]{0.48\linewidth}
\centering
\resizebox{\linewidth}{!}{%
\begin{tikzpicture}[
    x=1cm,y=1.5cm,
    every node/.style={font=\small}
]

\definecolor{myblue}{RGB}{70,80,180}
\definecolor{myred}{RGB}{170,90,90}
\definecolor{mygray}{RGB}{150,150,150}

\draw[gray] (0,0) -- (12.2,0);

\coordinate (k1) at (0.4,0);
\coordinate (k2) at (5.8,0);
\coordinate (k3) at (10,0);
\coordinate (k4) at (11.0,0);

\node at (3.3,0.2) {$\cdots$};

\draw[myred, line width=0.8pt, midarrowleft]
(k1) .. controls (2.6,2.8) and (10.8,4.3) .. (k4);

\draw[myblue, line width=0.8pt, midarrowright]
(1.8,0) .. controls (1.8,2.3) and ($(k3)+(120:2.2)$) .. (k3);

\draw[myblue, line width=0.8pt, midarrowright]
(2.9,0) .. controls (2.9,1.9) and ($(k3)+(135:1.9)$) .. (k3);

\draw[myblue, line width=0.8pt, midarrowright]
(4.4,0) .. controls (4.4,1.3) and ($(k3)+(150:1.4)$) .. (k3);

\draw[myred, line width=0.8pt, midarrowleft]
(k2) .. controls (6.8,0.45) and (8.3,1.0) .. (8.6,0);

\node at (1.5,0.30) {$\rho_{\nu}^{+}$};
\node at (2.5,0.30) {$\rho_{\nu+1}^{+}$};
\node at (4.0,0.30) {$\rho_{\nu+\lambda}^{+}$};
\node at (9.1,0.2) {$\rho_{\mu_{\nu}-1}^{-}$};
\node at (11.5,0.4) {$\rho_{\mu_{\nu}}^{-}$};

\node[below] at (k1) {$k_{\nu}^{+}$};
\node[below] at (k2) {$k_{\nu+\lambda+1}^{+}$};
\node[below] at (k3) {$k_{\bar\mu_{\nu}}^{-}$};

\draw (0.4,0.08) -- (0.4,-0.08);
\draw (5.8,0.08) -- (5.8,-0.08);
\draw (11.0,0.08) -- (11.0,-0.08);

\end{tikzpicture}%
}
\caption{}
\label{fig:orbits-flow-left}
\end{subfigure}
\hfill
\begin{subfigure}[t]{0.48\linewidth}
\centering
\resizebox{\linewidth}{!}{%
\begin{tikzpicture}[
    x=1cm,
    y=1cm,
    every node/.style={font=\small}
]

\definecolor{myblue}{RGB}{70,80,180}
\definecolor{myred}{RGB}{170,90,90}
\definecolor{mygray}{RGB}{150,150,150}

\draw[gray] (0,0) -- (12.6,0);

\coordinate (k1) at (0.6,0);   
\coordinate (k2) at (2.1,0);   
\coordinate (k3) at (5.8,0);   
\coordinate (k4) at (7.8,0);   
\coordinate (k5) at (10,0);    
\coordinate (k6) at (12.1,0);  

\coordinate (rhoNu) at (1.35,0);
\coordinate (rhoInner) at (3.8,0);

\coordinate (redEndA) at (10.15,0);
\coordinate (redEndB) at (9.4,0);
\coordinate (redEndC) at (8.2,0);

\node at (8.8,-0.18) {$\cdots$};

\draw[myblue, line width=0.8pt, midarrowright]
(rhoNu) .. controls (3.35,3.1) and (12.1,3.25) .. (k6);

\draw[myred, line width=0.8pt, midarrowleft]
(k2) .. controls ($(k2)+(70:1.8)$) and (10.15,2.45) .. (10.15,0);

\draw[myred, line width=0.8pt, midarrowleft]
(k2) .. controls ($(k2)+(50:1.6)$) and (9.4,2.05) .. (9.4,0);

\draw[myred, line width=0.8pt, midarrowleft]
(k2) .. controls ($(k2)+(45:1.35)$) and (8.2,1.60) .. (8.2,0);

\draw[myblue, line width=0.8pt, midarrowright]
(rhoInner) .. controls (4.5,0.75) and (5.8,0.95) .. (k3);

\node at (1.15,0.33) {$\rho_{\nu}^{+}$};
\node at (3.25,0.25) {$\rho_{\nu+1}^{+}$};
\node at (8.65,0.48) {$\rho_{\mu_{\nu+1}}^{-}$};
\node at (10.55,0.45) {$\rho_{\mu_{\nu}-1}^{-}$};

\node[below] at (k1) {$k_{\nu}^{+}$};
\node[below] at (k2) {$k_{\nu+1}^{+}$};
\node[below] at (k3) {$k_{\mu_{\nu+1}-1}^{-}$};
\node[below] at (k4) {$k_{\mu_{\nu+1}}^{-}$};
\node[below] at (k5) {$k_{\mu_{\nu}-1}^{-}$};
\node[below] at (k6) {$k_{\mu_{\nu}}^{-}$};

\end{tikzpicture}%
}
\caption{}
\label{fig:orbits-flow-right}
\end{subfigure}

\caption{Evolution of the Loewner chain.}
\label{fig:combined-orbits}
\end{figure}
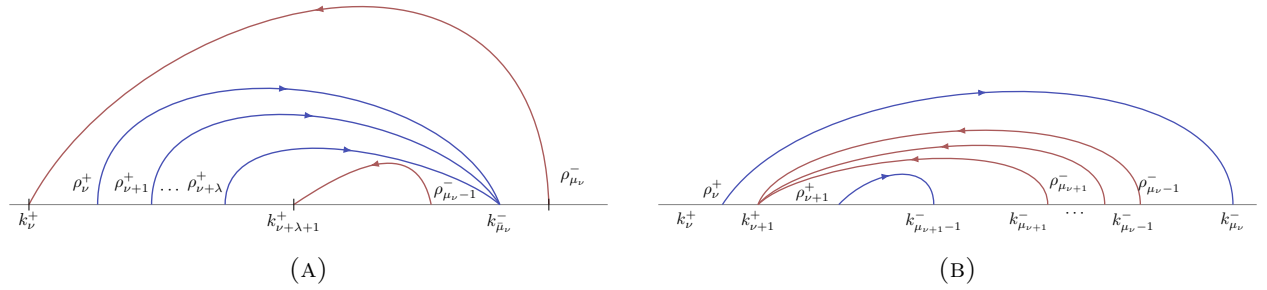

\section{Angles of convergence}\label{sec:angles}

In this section, we will work with the initial and the asymptotic behavior of the orbits of our Loewner chain, with regard to their Euclidean geometry. More specifically, we explicitly compute the angles by which the orbits converge to their respective attraction points. On top of that, we execute the corresponding computation of the angle by which each orbit emanates from its respective starting point. In short, we will discover that each orbit starts its journey orthogonally, but arrives at its attraction point non-tangentially by a distinct angle that depends solely on the choice of the positive numbers $b_j^\pm$. In order to compute these angles, we use a different method compared to all our previous work. Our technique heavily relies on careful estimates of harmonic measure and we draw upon the information analyzed in Subsection \ref{sub:harmonic measure}. 

We begin with the angles of convergence.
\begin{theorem}\label{thm:convergence angles}
    For each $\nu=1,\dots,N-1$, the curve $\Gamma_{\nu}^+(t)$ converges, as $t\to+\infty$, to $\mathbb{R}$ by angle
    \begin{equation*}
    \pi\frac{\sum\limits_{j=1}^{\mu_\nu}b_j^--\sum\limits_{j=\nu+1}^N b_j^+}{b_{\mu_\nu}^-}\in(0,\pi).
    \end{equation*}
    For each $\mu=1,\dots,M$, $\Gamma_{\mu}^-(t)$ converges, as $t\to+\infty$, to $\mathbb{R}$ by angle
     \begin{itemize}
        \item[\textup{(1)}]  $\pi\frac{\sum\limits_{j=1}^\mu b_j^--\sum\limits_{j=2}^Nb_j^+}{b_1^+}\in(0,\pi),$ if $\mu\in\{\mu_1,\dots,M\}$;
         \item[\textup{(2)}]   $\pi\frac{\sum\limits_{j=1}^{\mu}b_j^- - \sum\limits_{j=\nu_\omega+1}^N b_j^+}{b_{\nu_\omega}^+}\in(0,\pi),$ if $\mu\in\{\mu_{N-1},\dots,\mu_1-1\}$ and $h(\rho_{\mu}^-)\in S_{\nu_{\omega}}^+$;
         \item[\textup{(3)}] $\pi\frac{\sum\limits_{j=1}^{\mu}b_j^-}{b_N^+}\in(0,\pi)$, if $\mu\in\{1,\dots,\mu_{N-1}-1\}$.
     \end{itemize}
\end{theorem}
\begin{proof}
    For the sake of convenience during the proof, we are going to need certain pieces of notation. We set $\Lambda_\nu^+\vcentcolon=\{h(\rho_\nu^+)-s\vcentcolon s\ge0\}$, $\nu=1,\dots,N-1$, and $\Lambda_\mu^-\vcentcolon=\{h(\rho_\mu^-)+s\vcentcolon s\ge0\}$, $\mu=1,\dots,M$, for the horizontal half-lines contained in the boundary of $\Omega_0$. By the shape of $\Omega_0$, each such half-line, excluding the tip point, is the impression of two sets of prime ends. For any plausible index $j$, we write $\Lambda_{j,\mathrm{upper}}^\pm$ for the prime ends with impression on $\Lambda_j^+$ and defined through crosscuts with imaginary parts larger than $\mathrm{Im}h(\rho_j^\pm)$. In similar fashion, we define the sets of prime ends $\Lambda_{j,\mathrm{lower}}^\pm$. Finally, set $\partial\Omega_0^+=\mathbb{R}+i(b^+-b^-)\pi$ and $\partial\Omega_0^-=\mathbb{R}-ib^-\pi$, for the horizontal lines which bound the smallest horizontal strip containing $\Omega_0$.

    We start with the orbits $\Gamma_\nu^+$, $\nu\in\{1,\dots,N-1\}$. Fix such a $\nu$. By our configuration, $h_t(\rho_\nu^+(t))-\gamma t\in S_{\mu_\nu}^-$, for all $t\ge0$. We start with the case when $\mu_\nu>1$, and thus
    \begin{equation}\label{eq:harmonic 1}
        \mathrm{Im}h(\rho_{\mu_\nu-1}^-)<\mathrm{Im}[h_t(\rho_\nu^+(t))-\gamma t]=\mathrm{Im}h(f(\rho_\nu^+(t),t)) <\mathrm{Im}h(\rho_{\mu_\nu}^-),
    \end{equation}
    for all $t\ge0$. We understand that $\lim_{t\to+\infty}h_t(\rho_\nu^+(t))-\gamma t=\infty$ and in particular $h_t(\rho_\nu^+(t))-\gamma t$ converges to the prime end $\infty_{\mu_\nu}^-$ in the Carath\'{e}odory topology of $\Omega_0$. Returning to the upper half-plane and keeping in mind that $h$ is extended to a homeomorphism between $\partial\mathbb{H}$ and $\Omega_0\cup\partial_C\Omega_0$, we understand that the orbit $\Gamma_\nu^+(t)$ converges to $k_{\mu_\nu}^-$, as $t\to+\infty$. Our objective is to estimate the harmonic measure $\omega(\Gamma_\nu^+(t),[k_{\mu_\nu}^-,+\infty),\mathbb{H})$, as $t\to+\infty$. In view of Remark \ref{rem:harmonic measure angles}, this estimation will yield the slope by which the orbit converges. Keeping in mind that $h$ preserves the orientation and that $h(k_{\mu_\nu}^-)=\infty_{\mu_\nu}^-$, we comprehend that $h$ maps the interval $[k_{\mu_\nu}^-,+\infty)$ injectively onto the sets of prime ends $\Lambda_{\mu,\mathrm{upper}}^-\cup\Lambda_{\mu,\mathrm{lower}}^-$, $\mu=\mu_\nu,\dots,M$, and the prime ends $\infty_{\mu}^-$, $\mu=\mu_\nu,\dots,M$. The boundary points of $\Omega_0$ corresponding to these sets of prime ends are exactly $\Lambda_\mu^-$, $\mu=\mu_\nu,\dots,M$. This correspondence also takes into account that $h$ maps $\infty$ to the prime end of $\Omega_0$ with impression $\infty$ defined by crosscuts joining $\Lambda_M^-$ and $\partial\Omega_0^+$. Therefore, utilizing the conformal invariance of the harmonic measure and its property of acting as a Borel probability measure, we obtain
    \begin{equation}\label{eq:harmonic 2}
        \omega(\Gamma_\nu^+(t),[k_{\mu_\nu}^-,+\infty],\mathbb{H})=\sum_{\mu=\mu_\nu}^M\omega(h(\Gamma_\nu^+(t)),\Lambda_\mu^-,\Omega_0)=\sum_{\mu=\mu_\nu}^M\omega(h_t(\rho_\nu^+(t))-\gamma t,\Lambda_\mu^-,\Omega_0).
    \end{equation}
    Since $h_t(\rho_\nu^+(t))-\gamma t$ converges to $\infty_{\mu_\nu}^-$ in the Carath\'{e}odory topology of $\Omega_0$, and $\infty_{\mu_\nu}^-\in\overline{\Lambda_{\mu_\nu,\mathrm{lower}}^-}$, whereas it is disjoint from the rest of the set $h([k_{\mu_\nu}^-,+\infty))$, arguing as in Example \ref{ex:zero harmonic measure}, we understand that
    \begin{equation*}
        0=\lim_{t\to+\infty}\omega(h_t(\rho_\nu^+(t))-\gamma t,\Lambda_{\mu_\nu,\mathrm{upper}}^-,\Omega_0)=\lim_{t\to+\infty}\omega(h_t(\rho_\nu^+(t))-\gamma t,\Lambda_\mu^-,\Omega_0),
    \end{equation*}
    for all $\mu=\mu_\nu+1,\dots,M$. As a consequence, applying on \eqref{eq:harmonic 2}, we deduce
    \begin{equation}\label{eq:harmonic 3}
        \lim_{t\to+\infty}\left(\omega(\Gamma_\nu^+(t),[k_{\mu_\nu}^-,+\infty),\mathbb{H})-\omega(h_t(\rho_\nu^+(t))-\gamma t,\Lambda_{\mu_\nu,\mathrm{lower}}^-,\Omega_0)\right)=0.
    \end{equation}
\begin{figure}[ht]
\centering
\begin{tikzpicture}[
    x=1cm,
    y=1cm,
    every node/.style={font=\small},
    boundary/.style={black!70, line width=0.7pt},
    auxline/.style={black!70, line width=0.7pt},
    traj/.style={red!80, line width=0.9pt, dashed},
    bluecurve/.style={blue!80, line width=1pt},
    dot/.style={circle, fill=black, inner sep=1.4pt}
]

\begin{scope}[shift={(0,0)}]

\draw[boundary] (0.1,5.6) -- (7.3,5.6);
\draw[boundary] (0.1,1.1) -- (7.3,1.1);

\node at (3.7,5.92) {$\partial\Omega_0^{+}$};
\node at (3.7,0.72) {$\partial\Omega_0^{-}$};

\draw[auxline] (3.8,2.35) -- (6.75,2.35);

\draw[traj] (2.8,3.10) -- (6.95,3.10);

\node[dot] at (2.85,3.10) {};
\node[dot] at (5.65,3.10) {};

\node[below=4pt] at (2.85,3.10) {$h(\rho_{\nu}^{+})$};
\node[below=4pt, text=red!80] at (5.65,3.10)
{$h_t(\rho_{\nu}^{+}(t))-\gamma t$};

\node[right] at (6.95,3.10) {$\infty_{\mu_\nu}^{-}$};

\draw[auxline] (4.85,4.35) -- (7.30,4.35);

\draw[
  blue,
  line width=0.8pt,
  decorate,
  decoration={brace, mirror, amplitude=5pt, raise=2pt}
]
(4.85,4.35) -- (7.30,4.35)
node[midway, below=8pt, blue]
{$\Lambda^{-}_{\mu_\nu,\mathrm{lower}}$};
\end{scope}

\begin{scope}[shift={(8.9,0)}]

\draw[boundary] (0.1,5.6) -- (7.3,5.6);
\draw[boundary] (0.1,1.1) -- (7.3,1.1);

\node at (3.7,5.92) {$\partial\Omega_0^{+}$};
\node at (3.7,0.72) {$\partial\Omega_0^{-}$};

\draw[auxline] (0.15,4.42) -- (3.30,4.42);

\draw[traj] (0.95,3.62) -- (4.55,3.62);

\node[left] at (0.95,3.62) {$\infty_{\nu}^{+}$};

\node[dot] at (1.70,3.62) {};
\node[dot] at (4.60,3.62) {};

\node[above=5pt, text=red!80] at (1.70,3.62)
{$h_t(\rho_{\mu_\nu}^{-}(t))-\gamma t$};
\node[above=5pt] at (4.60,3.62)
{$h(\rho_{\mu_\nu}^{-})$};

\draw[auxline] (0.10,2.55) -- (3.85,2.55);
\draw[
  blue,
  line width=0.8pt,
  decorate,
  decoration={brace, amplitude=5pt, raise=2pt}
]
(0.10,2.55) -- (3.85,2.55)
node[midway, above=8pt, blue]
{$\Lambda^{+}_{\nu,\mathrm{upper}}$};

\end{scope}

\end{tikzpicture}
\caption{A schematic description of the two configurations.}
\label{fig:two-configurations}
\end{figure}
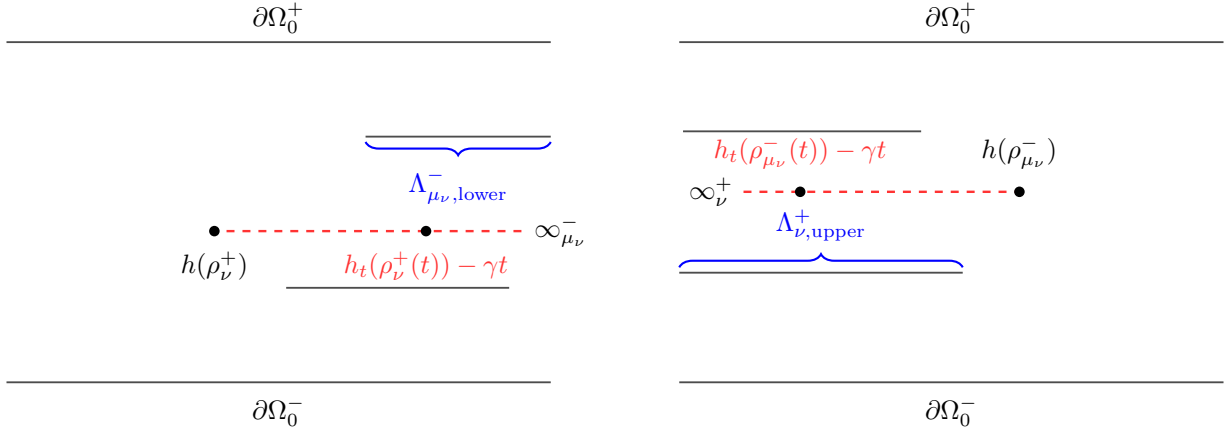
    By construction, there exists an appropriate set of indices $J$ which necessarily contains $\nu$ so that the slit horizontal strip $\Sigma\vcentcolon=S_{\mu_\nu}^-\setminus\cup_{j\in J}\Lambda_j^+$ is contained inside $\Omega_0$. As a matter of fact, these two simply connected domains share the set of prime ends $\Lambda_{\mu_\nu,\mathrm{lower}}^-$ (in the case of $\Sigma$ this set of prime ends can be identified as the actual set $\Lambda_{\mu_\nu}^-$, but this does not affect the proof). Therefore, using the Strong Markov Property from \eqref{eq:Markov}
    \begin{align}\label{eq:harmonic 4}
      \notag  &\omega(h_t(\rho_\nu^+(t))-\gamma t,\Lambda_{\mu_\nu,\mathrm{lower}}^-,\Omega_0)=\\
        &\omega(h_t(\rho_\nu^+(t))-\gamma t,\Lambda_{\mu_\nu,\mathrm{lower}}^-,\Sigma)+\int_{\partial\Sigma\cap\Omega_0}\omega(\zeta,\Lambda_{\mu_\nu,\mathrm{lower}}^-,\Omega_0)\omega(h_t(\rho_\nu^+(t))-\gamma t,d\zeta,\Sigma).
    \end{align}
    Setting $I(t)$ the integral in \eqref{eq:harmonic 4}, we are interested in its asymptotic behavior, as $t\to+\infty$. Remembering that the harmonic measure is always bounded above by $1$ due to being a probability measure, we get
    \begin{equation}\label{eq:harmonic 5}
        I(t)\le\int_{\partial\Sigma\cap\Omega_0}1\omega(h_t(\rho_\nu^+(t))-\gamma t,d\zeta,\Sigma)=\omega(h_t(\rho_\nu^+(t))-\gamma t,\partial\Sigma\cap\Omega_0,\Sigma),
    \end{equation}
    for all $t\ge0$. But evidently, $\partial\Sigma\cap\Omega_0$ consists of two horizontal half-lines which stretch to the left, whereas $h_t(\rho_\nu^+(t))-\gamma t$ diverges to $\infty$ moving towards the left (even in the geometry of $\Sigma$). Arguing as previosuly, it is straightforward that $\lim_{t\to+\infty}\omega(h_t(\rho_\nu^+(t))-\gamma t,\partial\Sigma\cap\Omega_0,\Sigma)=0$, and thus returning to \eqref{eq:harmonic 5}, we understand that $\lim_{t\to+\infty}I(t)=0$. Ergo, taking limits as $t\to+\infty$ in \eqref{eq:harmonic 4}, we deduce
    \begin{equation}\label{eq:harmonic 6}
    \lim_{t\to+\infty}\left(\omega(h_t(\rho_\nu^+(t))-\gamma t,\Lambda_{\mu_\nu,\mathrm{lower}}^-,\Omega_0)-\omega(h_t(\rho_\nu^+(t))-\gamma t,\Lambda_{\mu_\nu,\mathrm{lower}}^-,\Sigma)\right)=0.
    \end{equation}
    So, in a sense, we moved on from estimating a harmonic measure on $\Omega_0$ to estimating a harmonic measure on the simple domain $\Sigma$. But we need to further simplify it. Evidently, $\Sigma\subset S_{\mu_\nu}^-$ and once again $\Lambda_{\mu_\nu,\mathrm{lower}}^-\subset \partial_C\Sigma\cap\partial_C S_{\mu_\nu}^-$ (this inclusion can be stated without prime ends, exactly as previously). By a second application of the Strong Markov Property,
    \begin{align*}
       & \omega(h_t(\rho_\nu^+(t))-\gamma t,\Lambda_{\mu_\nu,\mathrm{lower}}^-,S_{\mu_\nu}^-)=\\
       &\omega(h_t(\rho_\nu^+(t))-\gamma t,\Lambda_{\mu_\nu,\mathrm{lower}}^-,\Sigma)+\int_{\partial\Sigma\cap S_{\mu_\nu}^-}\omega(\zeta,\Lambda_{\mu_\nu,\mathrm{lower}}^-,S_{\mu_\nu}^-)\omega(h_t(\rho_\nu^+(t))-\gamma t,d\zeta,\Sigma),
    \end{align*}
    for all $t\ge0$. Given that $\partial\Sigma\cap S_{\mu_\nu}^-=\cup_{j\in J}\Lambda_j^+$, a similar procedure yields that the above integral also converges to $0$, as $t\to+\infty$. Thus, we get
    \begin{equation}\label{eq:harmonic 7}
    \lim_{t\to+\infty}\left(\omega(h_t(\rho_\nu^+(t))-\gamma t,\Lambda_{\mu_\nu,\mathrm{lower}}^-,S_{\mu_\nu}^-)-\omega(h_t(\rho_\nu^+(t))-\gamma t,\Lambda_{\mu_\nu,\mathrm{lower}}^-,\Sigma)\right)=0.
    \end{equation}
    Combining \eqref{eq:harmonic 3}, \eqref{eq:harmonic 6} and \eqref{eq:harmonic 7}, we obtain
    \begin{equation*}
    \lim_{t\to+\infty}\left(\omega(\Gamma_\nu^+(t),[k_{\mu_\nu}^-,+\infty),\mathbb{H})-\omega(h_t(\rho_\nu^+(t))-\gamma t,\Lambda_{\mu_\nu,\mathrm{lower}}^-,S_{\mu_\nu}^-)\right)=0.
    \end{equation*}
    Set $\Lambda$ the upper horizontal line bounding $S_{\mu_\nu}^-$. Clearly $\Lambda_{\mu_\nu}^-\subsetneq\Lambda$. But in the Carath\'{e}odory topology of $S_{\mu_\nu}^-$ there is no difference between the sets $\Lambda_{\mu_\nu}^-$ and $\Lambda_{\mu_\nu,\mathrm{lower}}^-$ since $\Lambda_{\mu_\nu,\mathrm{upper}}^-\cap\partial_C S_{\mu_\nu}^-=\emptyset$. In addition, since $h_t(\rho_\nu^+(t))-\gamma t$ converges to a right prime end, the harmonic measure $\omega(h_t(\rho_\nu^+(t))-\gamma t,\Lambda\setminus\Lambda_{\mu_\nu}^-,S_{\mu_\nu}^-)=\omega(h_t(\rho_\nu^+(t))-\gamma t,\Lambda\setminus \Lambda_{\mu_\nu,\mathrm{lower}}^-,S_{\mu_\nu}^-)$ converges to $0$, as $t\to+\infty$. Taking advantage of the nature of the harmonic measure as a Borel probability measure, we get
    \begin{equation*}
        \lim_{t\to+\infty}\left(\omega(\Gamma_\nu^+(t),[k_{\mu_\nu}^-,+\infty),\mathbb{H})-\omega(h_t(\rho_\nu^+(t))-\gamma t,\Lambda,S_{\mu_\nu}^-)\right)=0.
    \end{equation*}
    Then, taking relation \eqref{eq:harmonic 1} into account and applying \eqref{eq:harmonic measure strip upper} from Example \ref{ex:strip} shows that
    \begin{equation*}
        \lim_{t\to+\infty}\omega(\Gamma_\nu^+(t),[k_{\mu_\nu}^-,+\infty),\mathbb{H})=\frac{\mathrm{Im}h(\rho_\nu^+)-\mathrm{Im}h(\rho_{\mu_\nu-1}^-)}{\mathrm{Imh(\rho_{\mu_\nu}^-)-\mathrm{Im}h(\rho_{\mu_\nu-1}^-)}}.
    \end{equation*}
    However, from \eqref{height with b}, we transform the above relation to
    \begin{align}\label{eq:harmonic 8}
    \notag \lim_{t\to+\infty}\omega(\Gamma_\nu^+(t),[k_{\mu_\nu}^-,+\infty),\mathbb{H})&=\frac{\pi\left(\sum_{j=\nu+1}^Nb_j^+-\sum_{j=1}^Mb_j^-+\sum_{j=\mu_\nu}^Mb_j^-\right)}{\pi\left(-\sum_{j=\mu_\nu+1}^Mb_j^-+\sum_{j=\mu_\nu}^Mb_j^-\right)}\\
    &=\frac{\sum_{j=\nu+1}^Nb_j^+-\sum_{j=1}^{\mu_\nu-1}b_j^-}{b_{\mu_\nu}^-}\in(0,1).
    \end{align}
    Finally, as we explained in Remark \ref{rem:harmonic measure angles}, the slope by which $\Gamma_\nu^+$ converges to $k_{\mu_\nu}^-$ can be discovered through the cluster set of $\pi(1-\omega(\Gamma_\nu^+(t),[k_{\mu_\nu}^-,+\infty),\mathbb{H}))$, as $t\to+\infty$. Since by \eqref{eq:harmonic 8} the limit actually exists, we infer that $\Gamma_\nu^+(t)$ converges, as $t\to+\infty$, to $k_{\mu_\nu}^-$ by an angle equal to  
    \begin{equation*}
        \pi\frac{\sum\limits_{j=1}^{\mu_\nu}b_j^--\sum\limits_{j=\nu+1}^N b_j^+}{b_{\mu_\nu}^-}\in(0,\pi).
    \end{equation*}

    To conclude our work on the orbit emanating from the point $\rho_\nu^+$, we need to treat the case $\mu_\nu=1$. In this case the proof follows almost identically with the only difference being that \eqref{eq:harmonic 1} is replaced by
    \begin{equation}\label{eq:harmonic 9}
        -\pi b^-<\mathrm{Im}[h_t(\rho_\nu^+(t))-\gamma t]=\mathrm{Im}h(f(\rho_\nu^+(t),t)) <\mathrm{Im}h(\rho_{1}^-),
    \end{equation}
    due to the fact that the lower boundary of the corresponding strip $S_{\mu_\nu}^-=S_1^-$ is the horizontal line $\partial\Omega_0^-$. Then, constructing again the slit strip $\Sigma$, using the Strong Markov Property, and taking limits as $t\to+\infty$, we find
    \begin{equation*}
        \lim_{t\to+\infty}\omega(\Gamma_\nu^+(t),[k_{\mu_\nu}^-,+\infty),\mathbb{H})=\frac{\mathrm{Im}h(\rho_\nu^+)+\pi b^-}{\mathrm{Im}h(\rho_1^-)+\pi b^-}=\frac{\pi\left(\sum_{j=\nu+1}^Nb_j^+-\sum_{j=1}^Mb_j^-+b^-\right)}{-\pi\sum_{j=2}^Mb_j^-+\pi b^-}=\frac{\sum_{j=\nu+1}^Nb_j^+}{b_1^-},
    \end{equation*}
    because of \eqref{eq:harmonic 9} and \eqref{height with b}. Arguing as in the general case, we comprehend that $\Gamma_\nu^+(t)$ converges, as $t\to+\infty$, to $k_{\mu_\nu}^-$ by an angle equal to
    \begin{equation*}
        \pi\frac{b_1^--\sum\limits_{j=\nu+1}^Nb_j^+}{b_1^-}\in(0,\pi).
    \end{equation*}

    We now move on to the orbits $\Gamma_\mu^-$, $\mu\in\{1,\dots,M\}$. Fix such a $\mu$. The calculations are very similar to the ones before, so we are going to omit certain details. We commence with the case where the index $\mu$ corresponds to a middle right half-line. Our configuration dictates that there exists some $\omega\in\{2,\dots,\lambda\}$ such that
    \begin{equation}\label{eq:harmonic 10}
        \mathrm{Im}h(\rho_{\nu_\omega}^+)<\mathrm{Im}[h_t(\rho_\mu^-(t))-\gamma t]=\mathrm{Im}h(f(\rho_\mu^-(t),t))<\mathrm{Im}h(\rho_{\nu_\omega-1}^+).
    \end{equation}
    In such a situation, $h_t(\rho_\mu^-(t))-\gamma t$ converges to the prime end $\infty_{\nu_\omega}^+$ in the Carath\'{e}ododry topology of $\Omega_0$. Accordingly, the orbit $\Gamma_\mu^-$ converges to $k_{\nu_\omega}^+$. Hence, we want to estimate the harmonic measure $\omega(\Gamma_\mu^-(t),[k_{\nu_\omega}^+,+\infty),\mathbb{H})$. But this time, inside $\Omega_0$, the quantity $h_t(\rho_\mu^-(t))-\gamma t$ converges to $\infty_{\nu_\omega}^+$ by travelling towards the left. This reverses all the orientations compared to the first part of the proof, even though the methodology remains the same. More specifically, the extension of $h$ into a homeomorphism between $\overline{\mathbb{H}}$ and $\Omega_0\cup\partial_C\Omega_0$ maps the interval $[k_{\nu_\omega}^+,+\infty)$ injectively onto the half-lines $\Lambda_\nu^+$, $\nu\in\{\nu_\omega,\dots,N-1\}$, and the half-lines $\Lambda_i^-$, $i\in\{1,\dots,M\}$ (or rather the two sets of prime ends with impression on each of the aforementioned half-lines), along with the prime ends $\infty_\nu^+$, $\nu\in\{\nu_\omega,\dots,N\}$, and $\infty_i^-$, $i\in\{1,\dots,M\}$. But out of all the boundary components above, the only one that has $\infty_{\nu_\omega}^+$ in its closure is the set of prime ends $\Lambda_{\nu_\omega,\mathrm{upper}}^+$. As a consequence,
    \begin{equation}\label{eq:harmonic 11}
        \lim_{t\to+\infty}\left(\omega(\Gamma_\mu^-(t),[k_{\nu_\omega}^+,+\infty),\mathbb{H})-\omega(h_t(\rho_\mu^-(t))-\gamma t,\Lambda_{\nu_\omega,\mathrm{upper}}^+,\Omega_0)\right)=0.
    \end{equation}
    Via our configuration, $h_t(\rho_\mu^-(t))-\gamma t\in S_{\nu_\omega}^+$, for all $t\ge0$, and there exists a suitable set of indices $J$ containing $\mu$ so that the slit horizontal strip $T\vcentcolon= S_{\nu_\omega}^+\setminus\cup_{j\in J}\Lambda_j^-$ is a subset of $\Omega$. But $\Lambda_{\nu_\omega,\mathrm{upper}}^+$ is contained both in $\partial_C\Omega_0$ and in $\partial_CT$. Using the Strong Markov Property and demonstrating that the respective integral from \eqref{eq:Markov} converges to $0$, we arrive at
    \begin{equation}\label{eq:harmonic 12}
        \lim_{t\to+\infty}\left(\omega(h_t(\rho_\mu^-(t))-\gamma t,\Lambda_{\nu_\omega,\mathrm{upper}}^+,\Omega_0)-\omega(h_t(\rho_\mu^-(t))-\gamma t,\Lambda_{\nu_\omega,\mathrm{upper}}^+,T)\right)=0.
    \end{equation}
    Executing a second application of the Strong Markov Property, this time with the simply connected domains $T$ and $S_{\nu_\omega}^+$, we find
    \begin{equation}\label{eq:harmonic 13}
        \lim_{t\to+\infty}\left(\omega(h_t(\rho_\mu^-(t))-\gamma t,\Lambda_{\nu_\omega,\mathrm{upper}}^+,S_{\nu_\omega}^+)-\omega(h_t(\rho_\mu^-(t))-\gamma t,\Lambda_{\nu_\omega,\mathrm{upper}}^+,T)\right)=0.
    \end{equation}
    Combining relations \eqref{eq:harmonic 11}, \eqref{eq:harmonic 12} and \eqref{eq:harmonic 13} yields
    \begin{equation*}
        \lim_{t\to+\infty}\left(\omega(\Gamma_\mu^-(t),[k_{\nu_\omega}^+,+\infty),\mathbb{H})-\omega(h_t(\rho_\mu^-(t))-\gamma t,\Lambda_{\nu_\omega,\mathrm{upper}}^+,S_{\nu_\omega}^+)\right)=0.
    \end{equation*}
    But thinking analogously to the first part of the proof and setting $\Lambda$ the lower horizontal line bounding $S_{\nu_\omega}^+$, we understand that $\omega(h_t(\rho_\mu^-(t))-\gamma t,\Lambda\setminus \Lambda_{\nu_\omega}^+,S_{\nu_\omega}^+)=\omega(h_t(\rho_\mu^-(t))-\gamma t,\Lambda\setminus\Lambda_{\nu_\omega,\mathrm{upper}}^+,S_{\nu_\omega}^+)$ converges to $0$, as $t\to+\infty$. Therefore,
    \begin{equation*}
        \lim_{t\to+\infty}\left(\omega(\Gamma_\mu^-(t),[k_{\nu_\omega}^+,+\infty),\mathbb{H})-\omega(h_t(\rho_\mu^-(t))-\gamma t,\Lambda,S_{\nu_\omega}^+)\right)=0.
    \end{equation*}
    However, we may once more explicitly compute the latter harmonic measure through \eqref{eq:harmonic 10} and relation \eqref{eq:harmonic measure strip lower} from Example \ref{ex:strip} in order to deduce
    \begin{align*}
        \lim_{t\to+\infty}\omega(\Gamma_\mu^-(t),[k_{\nu_\omega}^+,+\infty),\mathbb{H})&=\frac{\mathrm{Im}h(\rho_{\nu_\omega-1}^+)-\mathrm{Im}h(\rho_\mu^-)}{\mathrm{Im}h(\rho_{\nu_\omega-1}^+)-\mathrm{Im}h(\rho_{\nu_\omega}^+)}\\
        &=\frac{\pi\left(\sum_{j=\nu_\omega}^Nb_j^+-\sum_{j=1}^Mb_j^-\right)+\pi\sum_{j=\mu+1}^Mb_j^-}{\pi\left(\sum_{j=\nu_\omega}^Nb_j^+-\sum_{j=1}^Mb_j^-\right)-\pi\left(\sum_{j=\nu_\omega+1}^Nb_j^+-\sum_{j=1}^Mb_j^-\right)}\\
        &=\frac{\sum_{j=\nu_\omega}^Nb_j^+-\sum_{j=1}^\mu b_j^-}{b_{\nu_\omega}^+}\in(0,1),
    \end{align*}
    where the second equality is derived from \eqref{height with b}. In view of Remark \ref{rem:harmonic measure angles}, we discover that $\Gamma_\mu^-$ converges to $k_{\nu_\omega}^+$ by an angle equal to
    \begin{equation*}
        \pi\frac{\sum\limits_{j=1}^{\mu}b_j^- - \sum\limits_{j=\nu_\omega+1}^N b_j^+}{b_{\nu_\omega}^+}\in(0,\pi).
    \end{equation*}

    We proceed to the upper slits and the case when $\mu\in\{\mu_1,\dots,M\}$. In this case, we have $\omega=1$, the orbit $\Gamma_\mu$ converges to $\rho_{\nu_1}^+=\rho_1^+$, and we work with the strip $S_{\nu_1}^+=S_1^+$. This time, the upper boundary of $S_1^+$ is exactly $\partial\Omega_0^+$ and \eqref{eq:harmonic 10} turns into
    \begin{equation}\label{eq:harmonic 14}
        \mathrm{Im}h(\rho_1^+)<\mathrm{Im}[h_t(\rho_\mu^-(t))-\gamma t]=\mathrm{Im}h(f(\rho_\mu^-(t),t))<\pi(b^+-b^-).
    \end{equation}
    Exactly as before, \eqref{eq:harmonic 14} and \eqref{eq:harmonic measure strip lower} will eventually lead to
    \begin{equation*}
        \lim_{t\to+\infty}\omega(\Gamma_\mu^-(t),[k_1^+,+\infty),\mathbb{H})=\frac{\pi(b^+-b^-)-\mathrm{Im}h(\rho_\mu^-)}{\pi(b^+-b^-)-\mathrm{Im}h(\rho_1^+)}=\frac{\pi(b^+-b^-)+\pi\sum_{j=\mu+1}^{M}b_j^-}{\pi(b^+-b^-)-\pi(\sum_{j=2}^N b_j^+-\sum_{j=1}^Mb_j^-)}.
    \end{equation*}
    Following the usual procedure and by means of Remark \ref{rem:harmonic measure angles}, we infer that $\Gamma_\mu^-$ converges by an angle equal to
    \begin{equation*}
        \pi\frac{\sum\limits_{j=1}^\mu b_j^--\sum\limits_{j=2}^Nb_j^+}{b_1^+}\in(0,\pi).
    \end{equation*}

    To conclude the proof, we treat the lower slits, namely the case $\mu\in\{1,\dots,\mu_{N-1}-1\}$. Under this circumstance, $h_t(\rho_\mu^-(t))-\gamma t$ converges to $\infty_N^+$ in the Carath\'{e}odory topology of $\Omega_0$ and the orbit $\Gamma_\mu^-$ converges to $k_N^+$. For this reason, we need to work with the strip $S_N^+$. This time, \eqref{eq:harmonic 10} becomes
    \begin{equation}\label{eq:harmonic 15}
        -\pi b^-<\mathrm{Im}[h_t(\rho_\mu^-(t))-\gamma t]=\mathrm{Im}h(f(\rho_\mu^-(t),t))<\mathrm{Im}h(\rho_{N-1}^+).
    \end{equation}
    We realize that the only boundary component that belongs in the image $h([k_N^+,+\infty))$ and contains $\infty_N^+$ in its closure, is exactly the horizontal line $\partial\Omega_0^-$. So, this time, we need to estimate the asymptotic behavior of the quantity $\omega(h_t(\rho_\mu^-(t))-\gamma t,\partial\Omega_0^-,\Omega_0)$. Acting as above and using \eqref{eq:harmonic 15},\eqref{eq:harmonic measure strip lower}, it turns out that 
    \begin{equation*}
        \lim_{t\to+\infty}\omega(\Gamma_\mu^-(t),[k_N^+,+\infty),\mathbb{H})=\frac{\mathrm{Im}h(\rho_{N-1}^+)-\mathrm{Im}h(\rho_\mu^-)}{\mathrm{Im}h(\rho_{N-1}^+)+\pi b^-}=\frac{\pi\left(b_N^+-\sum_{j=1}^Mb_j^-\right)+\pi\sum_{j=\mu+1}^Mb_j^-}{\pi\left(b_N^+-\sum_{j=1}^Mb_j^-\right)+\pi b^-}.
    \end{equation*}
    In conclusion, based on Remark \ref{rem:harmonic measure angles}, we extrapolate the angle of convergence of $\Gamma_\mu^-$ as
    \begin{equation*}
        \pi\frac{\sum\limits_{j=1}^{\mu}b_j^-}{b_N^+}\in(0,\pi).
    \end{equation*}
    \end{proof}

    We end our study by investigating the angle by which the orbits launch from the real line. An examination of how the launching angle of the orbit is related to analytical properties of the driving function has been realized in \cite{slei1}. In this work, we follow a different approach, since we compute the anlges through harmonic measure and without dealing with the driving function.
    \begin{theorem}\label{thm:initial angles}
        Every orbit emanates from $\mathbb{R}$ orthogonally.
    \end{theorem}
    \begin{proof}
        We will only execute the proof for the orbit $\Gamma_1^+$. The proof for the rest of the orbits follows identical arguments and we omit it for the sake of not being repetitive. Since we care about the initial behavior of the orbit and thinking in similar manner to the previous proof, we will estimate the harmonic measure $\omega(\Gamma_1^+(t),[\rho_1^+,+\infty),\mathbb{H})$, as $t\to 0^+$. Maintaining the notation from the proof of Theorem \ref{thm:convergence angles}, we know that the extension of $h$ as a homeomorphism between $\mathbb{H}\cup\partial_C\mathbb{H}$ and $\Omega_0\cup\partial_C\Omega_0$ maps the interval $[\rho_1^+,+\infty)$ injectively onto the set 
        \begin{equation*}
            \Lambda_{1,\mathrm{lower}}^+\cup\bigcup\limits_{\nu=2}^{N-1}\Lambda_\nu^+\cup\partial\Omega_0^-\cup\bigcup\limits_{\mu=1}^{M}\Lambda_\mu^-
        \end{equation*}
        along with the left prime ends $\infty_\nu^+$, $\nu\in\{2,\dots,N\}$, and the right prime ends $\infty_\mu^-$, $\mu\in\{1,\dots,M\}$. But since the tip point $h(\rho_1^+)$ belongs only to the closure of $\Lambda_{1,\mathrm{lower}}^+$ and to none of the aforementioned sets, we obtain
        \begin{align}\label{eq:initial angles 1}
        \notag  0 &=  \lim_{t\to0}\left(\omega(\Gamma_1^+(t),[\rho_1^+,+\infty),\mathbb{H})-\omega(h(\Gamma_1^+(t)),h([\rho_1^+,+\infty)),\Omega_0)\right)\\
          &=\lim_{t\to0}\left(\omega(\Gamma_1^+(t),[\rho_1^+,+\infty),\mathbb{H})-\omega(h_t(\rho_1^+(t))-\gamma t,\Lambda_{1,\mathrm{lower}}^+,\Omega_0)\right),
        \end{align}
        where the first equality is derived from the conformal invariance of the harmonic measure. Set $K\vcentcolon=\mathbb{C}\setminus \Lambda_1^+$ the ``Koebe-like'' domain resulting by removing the slit $\Lambda_1^+$. Clearly $\Omega_0\subset K$ and $\Lambda_{1,\mathrm{lower}}^+\subset\partial_C\Omega_0\cap\partial_C K$. As a result, by the Strong Markov Property,
        \begin{align}\label{eq:initial angles 2}
         \notag   \omega(h_t(\rho_1^+(t))-\gamma t,\Lambda_{1,\mathrm{lower}}^+,K)&=\omega(h_t(\rho_1^+(t))-\gamma t,\Lambda_{1,\mathrm{lower}}^+,\Omega_0)+\\
            &\int_{\partial\Omega_0\setminus\Lambda_1^+}\omega(\zeta,\Lambda_{1,\mathrm{lower}}^+,K)\omega(h_t(\rho_1^+(t))-\gamma t,d\zeta,\Omega_0),
        \end{align}
        for all $t\ge0$, due to the fact that $\partial\Omega_0\cap K=\Lambda_1^+$. Following the usual procedure, denoting by $J(t)$ the integral in the preceding relation, we have
        \begin{equation*}
            J(t)\le \int_{\partial\Omega\setminus\Lambda_1^+}1\cdot\omega(h_t(\rho_1^+(t))-\gamma,d\zeta,\Omega_0)=\omega(h_t(\rho_1^+(t))-\gamma t,\partial\Omega_0\setminus\Lambda_1^+,\Omega_0),
        \end{equation*}
        for all $t\ge0$. Taking limits as $t\to0^+$, the convergence of $h_t(\rho_1^+(t))-\gamma_t$ towards $h(\rho_1^+)$ implies that $\lim_{t\to0}J(t)=0$. Applying on \eqref{eq:initial angles 2} and combining with \eqref{eq:initial angles 1}, we see that
        \begin{equation}\label{eq:initial angles 3}
            \lim_{t\to0}\left(\omega(\Gamma_1^+(t),[\rho_1^+,+\infty),\mathbb{H})-\omega(h_t(\rho_1^+(t))-\gamma t,\Lambda_{1,\mathrm{lower}}^+,K)\right)=0.
        \end{equation}
        However, due to the symmetry of $K$ with respect $\Lambda_1^+$ and the half-line $\{h(\rho_1^+)+s\vcentcolon s\ge0\}$, we understand that 
        \begin{equation*}
            \omega(h_t(\rho_1^+(t))-\gamma t,\Lambda_{1,\mathrm{lower}}^+,K)=\omega(h_t(\rho_1^+(t))-\gamma t,\Lambda_{1,\mathrm{upper}}^+,K)=\frac{1}{2},
        \end{equation*}
        for all $t\ge0$; see also Remark \ref{rem:Brownian}. Therefore, returning to \eqref{eq:initial angles 3}, we infer that
        \begin{equation*}
            \lim_{t\to0}\omega(\Gamma_1^+(t),[\rho_1^+,+\infty),\mathbb{H})=\frac{1}{2}
        \end{equation*}
        and in view of Remark \ref{rem:harmonic measure angles}, the orbit $\Gamma_1^+$ launches from $\rho_1^+$ at angle $\frac{\pi}{2}$, that is orthogonally. As we already mentioned, for the remaining orbits we work analogously to extract the orthogonal launch.
    \end{proof}

    \addtocontents{toc}{\protect\setcounter{tocdepth}{1}}
    \section{Concluding remarks}

    We conclude our work by pointing out a few useful remarks and observations. First of all, we point out the chronological order of the choices of the parameters. It is clear that we deal with some sets of parameters, some of them are arbitrarily chosen while others arise during the construction. Namely, these sets of parameters are the real points $k_j^{\pm}$, the positive numbers $b_j^{\pm}$, the exponents $\theta_j^{\pm}$, the roots $\rho_{j}^{\pm}$ which determine the driving function of the Loewner chain and finally the parameter $\gamma$.

    \subsection*{The attraction points} 
    To build the close-to-convex function $h$ in \eqref{h for multiple points} we make a choice of real numbers 
    \begin{equation}
        \label{real choice k}
        k_1^+<\dots<k_{N-1}^+<k_N^+<k_1^-<\dots<k_M^-.
    \end{equation}
    in increasing order. Through $h$, these points are mapped to the points at infinity as we saw in Figure \ref{fig:image of h}. In addition to these, we opt for a choice of the positive numbers 
    \begin{equation}
        \label{choice of betas}
        b_1^+,\dots,b_{N}^+,b_1^-,\dots,b_M^->0
    \end{equation}
    which determine the height of the half-lines. Because the tip points are pushed to infinity as we saw in Figures \ref{fig: orbits on Omega 1} and \ref{fig: orbits on Omega 2}, we then understand that the choices \eqref{real choice k} and \eqref{choice of betas} combined together, give all the attraction points of the chains. As a result, it is important to note that in our construction we can arbitrarily choose the number and the position of the attraction points of the chain! 

    \subsection*{Geometric interpretation of the exponents} 
    The most important part of this work is how we choose the exponents $\theta_j^{\pm}$. Having chosen the attraction points above, we next have to determine the correct variation for the parameters $k_j^{\pm}$, so that they converge to $k_N^+$. Namely, we define the functions $k_j^{\pm}(t)=k_j^{\pm}+(k_j^{\pm}-k_N^+)e^{-\theta_j^{\pm}t}$, for an appropriate choice of exponents $\theta_j^{\pm}$, with
    
\begin{equation}\label{eq:thetas2}
   \theta_2^+\le\dots\le\theta_{N-1}^+\quad \text{and} \quad \theta_1^-\ge\dots\ge\theta_M^-.
\end{equation}
Now, the preceding ordering is only important for maintaining the ordering \eqref{orderings of k} for the points $k_j^{\pm}(t)$. Besides \ref{eq:thetas2}, we next have to pose conditions on the exponents, so that Corollary \ref{existance of thetas} is true. Recall that this result allows for our aim, that is the construction of a family of decreasing domains $(\Omega_t)_{t\ge0}$. 

However, it is not clear what this choice should be. For example, it may seems natural to try  

\begin{equation}\label{eq:thetas3}
   \theta_{N-1}^+>\dots>\theta_{2}^+>\theta_1^->\dots>\theta_M^- \quad\text{or}\quad\theta_1^->\dots>\theta_M^->\theta_{N-1}^+>\dots>\theta_{2}^+.
\end{equation}
If we consider such a choice, then by repeating the procedure of Section 4, we will deduce that for every tuple of exponents satisfying the above ordering, Corollary \ref{existance of thetas} fails! Thus, it is always true that 
\begin{equation*}
      \min_{\nu\in\{1.\dots,N-1\}}r(\rho_{\nu}^+)\le\max_{\mu\in\{1,\dots,M-1\}}r(\rho_{\mu}^-).   
    \end{equation*}
This blocks the existence of the parameter $\gamma$, hence the limits \eqref{limits of the tips+} and \eqref{limits of the tips-} cannot hold, so we cannot make the construction. 

The fact that the "correct" choice is to consider strict inequalities in \eqref{eq:thetas2} unless a strip $S_{\nu}^+$ lies in some $S_{j}^-$ and a strip $S_{\mu}^-$ lies in some strip $S_j^+$ is not clear neither is it intuitive, nevertheless, it works. This is really exciting and it means that there is a hidden geometric mechanism behind this choice!

\subsection*{The driving function}
Lastly we comment on the driving function of the Loewner chain, which is basically determined by the continuous functions $\rho_j^{\pm}(t)$. Recall that these are the roots of $h'(\cdot,t)$ and also the preimages of the tip points. Of course, it is well known from the general theory that the poles of the driving function are the preimages of the tip points. However, we see that the Loewner chain (before normalization) is written in the conjugation formula $f(z,t)=h^{-1}(h(z,t)-\gamma t)$, where $(h(\cdot,t)_{t\ge0})$ is a family of close-to-convex functions. Therefore, as we mentioned in the Introduction, we obtain a similar form to the conjugation formula for semigroups, which use the convex in the right direction functions, a special case of close-to-convex function.

As in Semigroups, the conjugation formula, $h^{-1}_0(h(\cdot,t)-\gamma t)$ linearizes the orbits of the tip points, as it maps them onto half-lines. For this reason, a natural question is what form should the driving function have, so as to produce a Loewner chain of the aforementioned type, which would lead to Loewner families that slightly generalize Semigroups.

\subsection*{Acknowledgments} The authors would like to express their gratitude to D. Betsakos, P. Galanopoulos, A. Siskakis and A. Sola for their suggestions and carefull consideration of this article.

\end{document}